\documentclass[11pt,a4paper]{amsart}
\usepackage[utf8]{inputenc}
\usepackage[british,english,american]{babel}
\usepackage[a4paper,margin=3cm]{geometry}
\usepackage{amsmath,amssymb,amsthm}
\usepackage{graphicx}
\usepackage{caption}
\usepackage{subcaption}
\usepackage{float}
\usepackage{listings}
\usepackage[svgnames]{xcolor}
\usepackage{colortbl}
\usepackage{lscape}
\usepackage{multicol}
\usepackage{dsfont}
\usepackage{float}
\usepackage{algorithm}
\usepackage{enumerate}
\usepackage{algpseudocode}
\usepackage{pgf}
\usepackage[colorinlistoftodos]{todonotes}
\usepackage{booktabs}
\usepackage{multirow}
\usepackage{siunitx}
\DeclareMathOperator{\Div}{div}

\newtheorem{teo}{Theorem}
\newtheorem{prop}{Proposition}
\newtheorem{lema}{Lemma}
\newtheorem{defi}{Definition}
\newtheorem{remark}{Remark}
\newtheorem{coro}{Corollary}

\DeclareMathOperator{\argmin}{argmin}
\DeclareMathOperator{\cont}{cont}
\begin{document}
\title[Numerical Exact Penalization for incompressible viscoplastic fluids]{Nonsmooth Exact Penalization Second--order Methods for incompressible Bingham flows}
\author{Sergio Gonz\'alez-Andrade}
\author{Sof\'ia L\'opez-Ord\'o\~nez}
\author{Pedro Merino}
\address{Research Center on Mathematical Modeling (MODEMAT) and  Departamento de Matem\'atica -  Escuela Polit\'ecnica Nacional\\\footnotesize Ladr\'on de
Guevara E11-253, Quito 170413, Ecuador}
\email{sergio.gonzalez@epn.edu.ec, sofia.lopezo@epn.edu.ec and pedro.merino@epn.edu.ec}

\thanks{$^*$This research has been partially supported by Escuela Polit\'ecnica Nacional within the project PIGR-18-03.}

\begin{abstract}
We consider the exact penalization of the incompressibility condition $\text{div}(\mathbf{u})=0$  for the velocity field of a Bingham fluid in terms of the $L^1$--norm. This penalization procedure results in a nonsmooth optimization problem for which we propose an  algorithm using generalized second--order information. Our method solves the resulting nonsmooth problem by considering the steepest descent direction and extra generalized second--order information associated to the nonsmooth term. This method has the advantage that the divergence-free property is enforced by the descent direction proposed by the method without the need of build-in divergence-free approximation schemes. The inexact penalization approach, given by the $L^2$--norm, is also considered in our discussion and comparison.

\end{abstract}
\maketitle
\section{Introduction}
Within the vast and interesting literature related to the numerical solution of viscoplastic fluids, there are mainly two general strategies: purely nonsmooth algorithms and regularization based procedures. The first approach deals with the purely nonsmooth variational problem by  considering  classical algorithms in the framework of  Augmented Lagrangian methods (see  for example \cite{Gabay, Uzawa, Glowinski}). The second procedure copes with the nondifferentiability involved in the viscoplastic fluid models by regularizing the non-differentiable term imposed on the strain rate tensor by, for instance, using a local  $C^1$ regularization (see \cite{DlRG1, DlRG2, DlRG3, Gonzalez-Andrade1, Gonzalez-Andrade2, GALO}). Then, well known first-order or second-order methods can be applied to solve the regularized problems.  In this context, the study of superlinear methods have been addressed in \cite{DlRG1, DlRG2, DlRG3}. There, a Huber local regularization of the problem is formulated in order to use an efficient semismooth Newton method. 

Recently, several first-order methods from nonsmooth optimization have been successfully applied to solve these models, without the need of regularizing the original problem, see \cite{Treskatis}. However, higher order methods require the consideration of second--order information or acceleration techniques based on the utilization of surrogate differentiable problems.

If the finite--element--method is used for the numerical approximation of viscoplastic fluids, several important stability issues arise that need to be addressed. Particularly, the well--known Ladyzhenskaya–Babu\v{s}ka–Brezzi condition \cite[Ch.2 Sec.1.4]{Girault} must hold in order to guarantee a stable approximation. This conditions is satisfied, for instance, by Taylor--Hood finite elements \cite[Ch.2 Sec.4.2]{Girault}. In addition, the divergence--free constraint is another important condition which can be incorporated in the finite element approximation. However, the implementation of divergence--free finite--element spaces is not an easy task and usually the incompressibility condition is kept as constraint. 

In this paper, we investigate the exact penalization of the incompressibility condition by using the $L^1$--norm, combined with a Huber regularization of the characteristic nondifferentiable term of the cost functional associated to the Bingham model. This consideration allows us to incorporate generalized second--order information associated to both nonsmooth terms: the regularization associated to the nondifferentiable term of the velocity gradient and the $L^1$--norm of the divergence of the velocity field. The sparsity promoting effect of the  $L^1$--norm is a well known feature exploited in data analysis and optimal control problems \cite{Stadler}. In the fluid context, we seek for the sparsification of the divergence of the velocity field of the fluid. The novelty of our contribution consists on the design of a an \emph{exact penalization} based algorithm, using second--order information, for the numerical solution of viscoplastic fluids. As mentioned before, the exact penalization is proposed by using the sparsity promoting feature of the $L^1$--norm that enforces the incompressibility condition, which is equivalent to the original formulation for higher values of the penalization parameter. The benefits of imposing sparsity on the divergence of the solution can be compared with sparse solutions on optimal control problems where their approximations have exact null entries. In this context, null divergence is better computed using the $L^1$--norm.   Also, we study the quadratic $L^2$--norm penalization, which is recognizably distinct from the sparsity penalizer in that its equivalence with the original problem is reached at the limit to the infinity of the penalizing parameter. 

The usage of $L^1$--norm entails both theoretical and numerical challenges in view of its nonsmoothness. Indeed, we emphasise that the contribution of this work consists in the analysis of exact penalization approach and the numerical algorithm that we propose in order to solve efficiently this formulation. We analyze the exact penalization problem in the context of nonsmooth optimization. Further, we show the existence of a lower bound for the penalizing parameter that guarantees the equivalence of the divergence penalized optimization problem with the original regularized formulation. Therefore, an \emph{exact penalization}  formulation is obtained. In addition, we show that the fluid's pressure can be recovered from the associated multiplier of the exact penalization by using an analogous approach to the de Rham's theory.

On the other hand, the algorithm that we propose computes a descent direction which is generated by a minimum norm subgradient problem, and subsequently modified by generalized second--order information, in the spirit of \cite{Merino}. The minimum norm subgradient problem implies the numerical solution of an optimization subproblem. However, this associated subproblem is a quadratic constrained type problem that can be efficiently solved in a smaller subspace by well known methods. Therefore, its computational cost is a good price to pay in order to obtain the descend direction. The second--order information is obtained by enriching the Hessian matrix associated with the differentiable part of the cost function. This enrichment procedure  consists of adding a matrix resulting of the generalized differentiation of a Huber regularization of the $L^1$--norm penalizing term of the divergence-free condition.  We prove that our algorithm converges to the solution of the exact--penalization formulation of our problem and, a fortiori, to the solution of the regularized optimization problem of the fluid.

We organize this article  as follows: first we present known results and the classical optimization formulation of the steady-state Bingham flow and its local Huber regularization. Next, we present the exact penalization formulation and address several theoretical questions regarding its equivalence with the constrained problem and existence of multipliers. Also, the quadratic penalization and its equivalence at the limit is discussed. In section 4, the recovery of the fluid's pressure is presented. There, the multiplier associated to the exact penalization is related with the pressure. Section 5 is devoted to the numerical algorithm using generalized second--order information associated to the nondifferentiable term of the cost. A convergence analysis is carried out for the proposed algorithm using a nonsmooth theory framework in functional spaces. Finally, a set of numerical experiments is presented and a discussion of the advantages of the exact penalization approach are illustrated in different simulations of 2D and 3D flows.

\subsection{Preliminaries}
The following notation will be used thorough the rest of the paper. The Euclidean norm in $\mathbb{R}^n$ is denoted by $|\cdot|$. The duality pairing between a Banach space $Y$ and its dual $Y^ {*}$ is given by $\left\langle \cdot, \cdot \right\rangle_{Y,Y^*}$, while any real inner product defined on $Y$ will be noted by $(\cdot, \cdot)_Y$. $\Omega$ is an open and bounded subset of $\mathbb{R}^{n}$, for $n=2,3$, with Lipschitz boundary. The Frobenius scalar product in $\mathbb{R}^{n\times n}$ and its associated norm are defined by
\[
A:B=\textrm{tr}(A B^\top)\,\,\mbox{  and  }\,\, |A|= \sqrt{(A:A)},\,\,\mbox{ for $A$, $B$ $\in \mathbb{R}^{n\times n}$},
\] 
respectively. We use the bold notation for the vector spaces, such as $\mathbf{H}_0^1(\Omega)=(H_0^1(\Omega))^n $. Further, we introduce the space of symmetric matrices of $L^p$-functions as 
\[
\mathbb{L}^p(\Omega):= \left\lbrace \boldsymbol{\tau}= (\tau_{ij})_{i,j=1}^n\,:\, \tau_{ij}=\tau_{ji} \in L^p(\Omega) \right\rbrace,
\]
and the divergence--free space by  $V=\{u \in \mathbf{H}_0^1(\Omega)\,:\, \Div \mathbf{u}=0\}$. 

Finally, we use the notation $\mathcal{E} = (\nabla + \nabla ^{\top})$ for the symmetric gradient operator. Considering that $\mathcal{E}:\mathbf{H}_0^1(\Omega) \rightarrow \mathbb{L}^2(\Omega)$, we obtain that $\mathcal{E}\mathbf{u}=(\mathcal{E}_{ij}(\mathbf{u}))$, with $\mathcal{E}_{ij}(\mathbf{u}):=\frac{1}{2}(\frac{\partial u_i}{\partial x_j} + \frac{\partial u_j}{\partial x_i}) \in L^2(\Omega)$ for $1\leq i,j\leq n$. By \cite[pp. 404]{Ciarlet}, we have that $\displaystyle  \int_{\Omega} \mathcal{E}\mathbf{u} : \mathcal{E}\mathbf{u} \, dx= \|\mathcal{E}\mathbf{u}\|^2_{\mathbb{L}^2}.$ 

\section{Constrained optimization problem for viscoplastic fluids and its regularization}\label{sec:1}
It is well known ( \cite{Gonzalez-Andrade1, Gonzalez-Andrade2, Huilgol1, Huilgol2}) that the solution of the following minimization problem corresponds to the velocity field $\bar{\mathbf{u}}$ of the steady-state Bingham flow.

\begin{equation}\label{eq:prob}
\underset{\mathbf{u}\in V}{\min}\, \widetilde{J}(\mathbf{u}):= \mu \int_\Omega \mathcal{E}\mathbf{u} :\mathcal{E}\mathbf{u} \,dx + g\int_\Omega |\mathcal{E} \mathbf{u}|\, dx - \int_\Omega \mathbf{f}\cdot\mathbf{u} \,dx.
\end{equation}

The Huber regularization of the previous problem  \eqref{eq:prob} was introduced in  \cite{DlRG1}, and it reads as follows:
\begin{equation}\label{eq:probreg}
\underset{\mathbf{u}\in V}{\min}\, J(\mathbf{u}):= \mu \int_\Omega \mathcal{E}\mathbf{u} :\mathcal{E}\mathbf{u} \,dx + \int_\Omega \Psi (\mathcal{E}\mathbf{u})\, dx - \int_\Omega \mathbf{f}\cdot\mathbf{u}\,dx,
\end{equation}
where $\Psi: \mathbb{R}^{d \times d} \rightarrow \mathbb{R}$ is a local $C^1$--regularization of the Frobenius norm, defined by
\begin{equation}\label{eq:huber}
\Psi(A)=\begin{cases}
g |A | -\frac{g^2}{2 \beta} & if   \,\, |A| \geq \frac{g}{\beta}, \\
\frac{\beta}{2} |A|^2 & if \,\, |A| < \frac{g}{\beta}.
\end{cases}
\end{equation}
Here $\beta>0$ is a given approximation parameter. Clearly, we have that $\beta \rightarrow \infty$ implies that $\Psi(A) \rightarrow |A|$, see \cite{DlRG1}. 

The objective funcional $J$ is proper, strictly convex, continuous and coercive (see \cite{Lions, Gonzalez-Andrade1}). Therefore, \cite[Prop. 1.2]{Ekeland} implies the existence of a unique solution $\widetilde{\mathbf{u}} \in V$ of \eqref{eq:probreg}. 

The unconstrained optimization problem is posed  in the divergence--free space $V$. Thus, we can reformulate this problem \eqref{eq:probreg} as the following constrained optimization problem in $\mathbf{H}_0^1(\Omega)$

\begin{equation}\label{eq:prob_const}
\begin{cases}
\begin{aligned}
&\underset{\mathbf{u}\in \mathbf{H}_{0}^1(\Omega)}{\min} J(\mathbf{u}):= \mu\int_\Omega \mathcal{E}\mathbf{u}:\mathcal{E}\mathbf{u} \,dx + \int_\Omega \Psi (\mathcal{E}\mathbf{u})\, dx - \int_\Omega \mathbf{f}\cdot\mathbf{u}\,dx \\
& \text{subject to:}  \hspace{5mm}  \Div \mathbf{u}=0, 
\end{aligned}
\tag{B}
\end{cases}
\end{equation}
where  $\Div:  \mathbf{H}_{0}^1(\Omega) \rightarrow L_0^2 (\Omega)$. This is a convex differentiable optimization problem with differentiable equality constraints. 

A regularity condition is needed in order to guarantee the existence of a Lagrange multiplier, that allows to derive a Karush-Kuhn-Tucker (KKT) system for the constrained problem \eqref{eq:prob_const}. Thus, following \cite[Sec. 1. eq. (1.4)]{Zowe}, we need to prove that the following regularity condition holds
\begin{equation}\label{eq:regular}
Z = L_0^2,
\end{equation}
where $Z:=\left\lbrace \alpha \Div \mathbf{v} : \alpha \geq 0, \mathbf{v} \in \mathbf{H}_{0}^1 (\Omega)\right\rbrace$  is the cone generated by the image of the divergence operator. An immediate observation is that the condition \eqref{eq:regular} is  satisfied, since the continuous linear operator $\Div:  \mathbf{H}_{0}^1(\Omega) \rightarrow L_0^2(\Omega)$ is surjective (see \cite[Th. 6.14-1]{Ciarlet}). Therefore, by \cite[Th. 3.1]{Zowe}, we infer the existence of a Lagrange multiplier $\lambda \in L_0^2(\Omega)$ associated to the constraint $\Div \mathbf{u}=0$, such that the following system at the solution $\widetilde{\mathbf{u}}$, is satisfied:
\begin{subequations}\label{eq:kkt}
 \begin{equation}\label{eq:kkt1}
  \Div \widetilde{\mathbf{u}}  = 0, \\
 \end{equation}
\begin{equation}\label{eq:kkt2}
\begin{array}{lll}
\left\langle J'(\widetilde{\mathbf{u}}) +  \text{grad } \lambda,\mathbf{u} -\widetilde{\mathbf{u}} \right\rangle_{\mathbf{H}^{-1}, \mathbf{H}^1_{0}} &=& \left\langle J'(\widetilde{\mathbf{u}}),\mathbf{u} -\widetilde{\mathbf{u}} \right\rangle_{\mathbf{H}^{-1}, \mathbf{H}^1_{0}} \vspace{0.2cm}\\ \hspace{4.5cm}&-& ( \lambda, \Div (\mathbf{u} -\widetilde{\mathbf{u}}) )_{L^2}=0, \hspace{2mm} \forall \,\mathbf{u} \in \mathbf{H}_{0}^{1}(\Omega).
\end{array}
\end{equation}
\end{subequations}
Here, $\textrm{grad}: L_0^2(\Omega) \rightarrow \mathbf{H}^{-1}(\Omega)$ is the dual operator of $- \Div: \mathbf{H}^1(\Omega) \rightarrow L_0^2(\Omega)$ and  $J'(\mathbf{u}): \mathbf{H}_{0}^{1}(\Omega) \rightarrow \mathbb{R}$ is the Fr\'echet derivative of $J$ at $\mathbf{u}$, given by
\begin{equation}\label{eq:Jp}
\begin{array}{rl}
\left\langle J'(\mathbf{u}),\mathbf{v} \right\rangle_{\mathbf{H}^{-1}, \mathbf{H}^1_{0}} = &\displaystyle 2\mu \int_{\Omega} \mathcal{E}\mathbf{u} : \mathcal{E}\mathbf{v} +\displaystyle g \int_{\Omega} \beta\frac{\mathcal{E}\mathbf{u} : \mathcal{E}\mathbf{v}}{\max(g, \beta|\mathcal{E}\mathbf{u}|)} \,\, dx -  \int_{\Omega} \mathbf{f}\cdot \mathbf{v} \,\, dx.
\end{array}
\end{equation}
Notice that the last derivative has a nondifferentiable term involving the $\max$ function. 

The following technical result regarding this nondifferentiable term will be useful in the forthcoming analysis. 

\begin{lema}\label{lem:reg}
Let $g>0$ be given. For fixed $\beta>0$, we introduce the following notation $\theta_\beta(\mathbf{u}):= \max(g,\beta |\mathcal{E}\mathbf{u}|)$. Then, for all $\mathbf{u}$ and $\mathbf{v}$ in $\mathbf{H}_0^1(\Omega)$, the following inequality holds
\begin{equation}
\theta_\beta(\mathbf{u})- \theta_\beta(\mathbf{v})< \beta |\mathcal{E}\mathbf{u}-\mathcal{E}\mathbf{v}|,\,\,\mbox{ a.e. in $\Omega$}. \label{eq:thetalip}
\end{equation}
\end{lema}
\begin{proof}
Let $\mathbf{u}\in \mathbf{H}_0^1(\Omega)$. We introduce the following sets
\begin{equation}
\begin{array}{lll}
E_\beta^\mathbf{u}:=\{x\in\Omega\,:\, \beta|\mathcal{E}\mathbf{u}|<g\}\,\,\mbox{  and  }\,\, I_\beta^\mathbf{u}:=\{x\in\Omega\,:\, \beta|\mathcal{E}\mathbf{u}|\geq g\}.
\end{array}
\end{equation}
We analyse pointwise bounds of $\theta_\beta(\mathbf{u})- \theta_\beta(\mathbf{v})$ on the four disjoint sets: $E_\beta^\mathbf{u}\cap E_\beta^\mathbf{v}$, $E_\beta^\mathbf{u}\cap I_\beta^\mathbf{v}$, $E_\beta^\mathbf{v}\cap I_\beta^\mathbf{u}$ and $I_\beta^\mathbf{u}\cap I_\beta^\mathbf{v}$.

Consider $E_\beta^\mathbf{u}\cap E_\beta^\mathbf{v}$. In this set, \eqref{eq:thetalip} is directly satisfied since  $\theta_\beta(\mathbf{u})=\theta_\beta(\mathbf{v})=g$. Hence
$\theta_\beta(\mathbf{u})- \theta_\beta(\mathbf{v})=0
$.
On $E_\beta^\mathbf{u}\cap I_\beta^\mathbf{v}$, we have that $\theta_\beta(\mathbf{u})=g$ and $\theta_\beta(\mathbf{v})=\beta|\mathcal{E}\mathbf{v}|$ with $\beta|\mathcal{E}\mathbf{v}|\geq g$, a.e. in $\Omega$. Therefore, we have the relation:
\begin{equation}\label{th2}
\theta_\beta(\mathbf{u})- \theta_\beta(\mathbf{v})=g-\beta|\mathcal{E}\mathbf{v}|\leq g-g=0\leq \beta |\mathcal{E}\mathbf{u}-\mathcal{E}\mathbf{v}|.
\end{equation}

Next, in $E_\beta^\mathbf{v}\cap I_\beta^\mathbf{u}$ it follows that $\theta_\beta(\mathbf{v})=g$ and $\theta_\beta(\mathbf{u})=\beta|\mathcal{E}\mathbf{u}|$ with $\beta|\mathcal{E}\mathbf{u}|\geq g$, a.e. in $\Omega$. Therefore, \eqref{eq:thetalip} is fulfilled, since
\begin{equation}
\theta_\beta(\mathbf{u})- \theta_\beta(\mathbf{v})=\beta|\mathcal{E}\mathbf{u}|-g< \beta|\mathcal{E}\mathbf{u}|-\beta|\mathcal{E}\mathbf{v}|\leq \beta |\mathcal{E}\mathbf{u}-\mathcal{E}\mathbf{v}|.\nonumber
\end{equation}

Finally, in $I_\beta^\mathbf{u}\cap I_\beta^\mathbf{v}$, we have that $\theta_\beta(\mathbf{u})=\beta|\mathcal{E}\mathbf{u}|$ and $\theta_\beta(\mathbf{v})=\beta|\mathcal{E}\mathbf{v}|$ . Therefore, we have that
\begin{equation}
\theta_\beta(\mathbf{u})- \theta_\beta(\mathbf{v})=\beta|\mathcal{E}\mathbf{u}|-\beta|\mathcal{E}\mathbf{v}|\leq \beta |\mathcal{E}\mathbf{u}-\mathcal{E}\mathbf{v}|. \nonumber
\end{equation}
Thus, since the four given sets provide a disjoint partitioning of $\Omega$, inequality \eqref{eq:thetalip} is satisfied almost everywhere in $\Omega$.
\end{proof}

\begin{remark}
Let $\mathbf{u}\in \mathbf{H}_0^1(\Omega)$ be given. Then, we have that \begin{equation}\label{thleqbeta}
\dfrac{|\mathcal{E}\mathbf{u}|}{\theta_\beta(\mathcal{E}\mathbf{u})}\leq \frac{1}{\beta},\,\,\mbox{a.e. in $\Omega$}.
\end{equation}
Indeed, since $\theta_{\beta}>0$ and by recalling the sets $E_\beta^\mathbf{u}$ and $I_\beta^\mathbf{u}$ from Lemma \ref{lem:reg}, we analogously observe that the following pointwise estimates hold.
On $E_\beta^\mathbf{u}$, we have that $\theta_\beta(\mathcal{E}\mathbf{u})=g$ and that $|\mathcal{E}\mathbf{u}|<\frac{g}{\beta}$. Then, we obtain that
\[
\dfrac{|\mathcal{E}\mathbf{u}|}{\theta_\beta(\mathcal{E}\mathbf{u})}=\dfrac{|\mathcal{E}\mathbf{u}|}{g}<\frac{1}{\beta},\,\,\mbox{a.e. in $\Omega$}.
\]
On $I_\beta^\mathbf{u}$, we have that $\theta_\beta(\mathcal{E}\mathbf{u})=\beta|\mathcal{E}\mathbf{u}|$. Then, we obtain that
\[
\dfrac{|\mathcal{E}\mathbf{u}|}{\theta_\beta(\mathcal{E}\mathbf{u})}=\dfrac{|\mathcal{E}\mathbf{u}|}{\beta|\mathcal{E}\mathbf{u}|}\leq\frac{1}{\beta},\,\,\mbox{a.e. in $\Omega$}.
\]
\end{remark}
\section{Exact Penalization Formulation}\label{sec:expen}
We rewrite the Huber regularized Bingham problem as the constrained optimization problem \eqref{eq:prob_const} to characterize its solutions via a KKT system. Now, the idea behind the penalty approach is to consider the constraint as a penalization of the objective functional. This approach leads us again to an unconstrained problem to be analyzed. In the case of the steady-state Huber regularized Bingham flow, taking into account the sparsification property of the $L^1$--norm we propose an exact penalization as follows.
\begin{equation}\label{eq:penalty}
J_\sigma(\mathbf{u}) := J(\mathbf{u})+ \sigma \| \Div(\mathbf{u}) \|_{L^1},
\end{equation}
where $\sigma >0$ and $J(\mathbf{u})$ is given in \eqref{eq:prob_const}. The functional given in \eqref{eq:penalty} is continuous and strictly convex, which satisfies that
\begin{equation*}\label{eq:coerci_theta}
\displaystyle \lim_{\|\mathbf{u}\|_{\mathbf{H}_0^1} \rightarrow \infty} J_{\sigma}(\mathbf{u}) = + \infty.
\end{equation*}
Thus, we can conclude that  the minimization problem 
\begin{equation}\label{minexact}
\underset{\mathbf{u}\in \mathbf{H}_0^1(\Omega)}{\min}\, J_\sigma(\mathbf{u}) \tag{EP}
\end{equation}
has a unique solution $\bar{\mathbf{u}} \in \mathbf{H}^1_0(\Omega)$ (see \cite{Gonzalez-Andrade2}). Also, since $J_\sigma$ is a nondifferentiable proper convex functional, by Fermat's Rule \cite[Th. 16.2]{Combettes}, the optimality condition reads as follows 
\begin{equation}\label{eq:opt_cond_exat_p}
  0 \in \partial J_\sigma (\bar{\mathbf{u}} )= J'(\bar{\mathbf{u}})+ \partial h(\bar{\mathbf{u}}),
\end{equation}
where $h(\bar{\mathbf{u}})= \sigma \|\Div (\bar{\mathbf{u}})\|_{L^1}$ and $\partial h(\bar{\mathbf{u}})$ is the convex subdifferential of $h$ at $\bar{\mathbf{u}}$.  Moreover, taking into account that $h= \sigma \|\cdot\|_{L^1} \circ \Div$, and by using the rules of subdifferential calculus for the composition of functions (see  \cite[Th. 4.13]{Clarke}), we have that for $\eta \in \partial h(\mathbf{u})$, there exists $\zeta \in \sigma \partial  \| \cdot\|_{L^1}(\Div\mathbf{u})$,  such that 
\begin{equation}\label{eq:subdiff}
 \left\langle \eta,\mathbf{v}  \right\rangle_{\mathbf{H}^{-1}, \mathbf{H}^1_{0}} =  \left\langle  -\text{ grad}\,\zeta, \mathbf{v}\right\rangle_{\mathbf{H}^{-1}, \mathbf{H}^1_{0}} =  ( \zeta,\Div\,\mathbf{v} )_{L^2}, \,\, \forall\, \mathbf{v} \in \mathbf{H}^1_0(\Omega).
\end{equation}
Therefore, the optimality condition \eqref{eq:opt_cond_exat_p} turns into
\begin{equation}\label{eq:opt_cond_nondiff}
  \left\langle -J'(\bar{\mathbf{u}}),\mathbf{v}  \right\rangle_{\mathbf{H}^{-1}, \mathbf{H}^1_{0}} =  ( \zeta,\Div\,\mathbf{v} )_{L^2}, \,\, \forall\, \mathbf{v} \in \mathbf{H}^1(\Omega).
\end{equation}
Note that $\zeta \in \sigma \partial \|\cdot\|_{L^1} (\Div\mathbf{u})$ yields that $\zeta \in L^2_0(\Omega)$, and $|\zeta| \leq \sigma \text{ a.e. in } \Omega  $.

Let us now discuss about the equivalence of the constrained problem \eqref{eq:prob_const} and the penalized problem \eqref{minexact}.

\begin{teo}\label{teo:exact2} 
Let $\tilde{\mathbf{u}}$ be the solution of problem \eqref{eq:prob_const}. Then $\tilde{\mathbf{u}}$ is also the solution of \eqref{minexact}. Furthermore, let $\bar{\mathbf{u}}$ be the solution of problem \eqref{minexact}, associated to a given $\sigma$. Then, there exists $\sigma_0 > 0$ such that for all $\sigma \geq \sigma_0$ the divergence free condition $$\|\Div\bar{\mathbf{u}}\|_{L^1} =0,$$ holds. This fact implies that $\bar{\mathbf{u}}$ is solution of the constrained problem \eqref{eq:prob_const}, i.e., $\bar{\mathbf{u}}=\tilde{\mathbf{u}}$, for $\sigma>\sigma_0$.
\end{teo}
 \begin{proof}
  Since $\Div\,\widetilde{\mathbf{u}}=0$ and $J(\widetilde{\mathbf{u}}) \leq J(\mathbf{u})$, it follows that
\begin{equation*}
\begin{array}{cl}
J(\widetilde{\mathbf{u}}) + \sigma \|\Div\widetilde{\mathbf{u}}\|_{L^1} = J(\widetilde{\mathbf{u}}) \leq J(\mathbf{u}) \leq  J(\mathbf{u})+  \sigma \|\Div\mathbf{u}\|_{L^1}, \,\, \forall\, \mathbf{u} \in \mathbf{H}_0^1(\Omega).
\end{array}
\end{equation*}
Thus, the solution $\tilde{\mathbf{u}}$ of the constrained problem \eqref{eq:prob_const} is also the minimizer of the functional in \eqref{minexact}. 

We prove the second claim by contradiction. Therefore let us assume that for all $\sigma_0 >0$, there exists $ \sigma \geq \sigma_0$ such that $\|\Div \bar{\mathbf{u}}\|_{L^1} > 0$. Since $\tilde{\mathbf{u}}$ is the solution of problem \eqref{eq:prob_const}, we have that $ \Div\,\widetilde{\mathbf{u}} = 0$. Next, we know that $\bar{\mathbf{u}}$ minimizes $J_\sigma$, which yields that
\begin{equation}\label{eq:aux003}
\begin{array}{cl}
0 & \leq J_\sigma(\widetilde{\mathbf{u}}) - J_{\sigma} (\bar{\mathbf{u}}) \\
&=  J(\widetilde{\mathbf{u}}) + \sigma \| \Div\,\widetilde{\mathbf{u}} \|_{L^1} - J(\bar{\mathbf{u}}) - \sigma \| \Div\,\bar{\mathbf{u}} \|_{L^1} \\
&= J(\widetilde{\mathbf{u}}) - J(\bar{\mathbf{u}}) - \sigma \| \Div\,\bar{\mathbf{u}} \|_{L^1}. \\
\end{array}
\end{equation}
Using the fact that $J$ is convex and differentiable, \eqref{eq:aux003}  implies that
\begin{equation*}\label{eq:aux004}
\sigma \| \Div\bar{\mathbf{u}} \|_{L^1}  \leq  J(\widetilde{\mathbf{u}}) - J(\bar{\mathbf{u}})  \leq - \left\langle J'(\widetilde{\mathbf{u}}), \bar{\mathbf{u}} - \widetilde{\mathbf{u}} \right\rangle_{\mathbf{H}^{-1}, \mathbf{H}^1_{0}}.
\end{equation*}
Then, from the optimality condition \eqref{eq:kkt2} for $\tilde{\mathbf{u}}$, we obtain that
\begin{equation}\label{eq:aux005}
\begin{array}{cl}
\sigma \| \Div\bar{\mathbf{u}} \|_{L^1} & \leq- \left\langle J'(\widetilde{\mathbf{u}}), \bar{\mathbf{u}} - \widetilde{\mathbf{u}} \right\rangle_{\mathbf{H}^{-1}, \mathbf{H}^1_{0}}\\
  & = -( \lambda,  \Div(\bar{\mathbf{u}} - \widetilde{\mathbf{u}}) )_{L^2} \\
& \leq | ( \lambda,  \Div \bar{\mathbf{u}} )_{L^2} | \\
& \leq   \| \lambda \|_{L^2} \|  \Div \bar{\mathbf{u}}\|_{L^2}.\\
\end{array}
\end{equation}
Let us assume $\lambda \neq 0$, otherwise the result follows. By choosing $\sigma > \sigma_0 + 2 \|\lambda\|_{L^2} \frac{\|  \Div \bar{\mathbf{u}}\|_{L^2}}{\|  \Div \bar{\mathbf{u}}\|_{L^1}} $, we have the relation
\begin{equation*}\label{eq:aux00001}
\sigma_0 +2\| \lambda \|_{L^2}  \frac{\|  \Div \bar{\mathbf{u}}\|_{L^2}}{\|  \Div \bar{\mathbf{u}}\|_{L^1}}  < \sigma\\
 \leq  \| \lambda \|_{L^2}   \frac{\|  \Div \bar{\mathbf{u}}\|_{L^2}}{\|  \Div \bar{\mathbf{u}}\|_{L^1}}.\\
\end{equation*}
Which is a contradiction since $2\| \lambda \|_{L^2}  \frac{\|  \Div \bar{\mathbf{u}}\|_{L^2}}{\|  \Div \bar{\mathbf{u}}\|_{L^1}}   \geq   \| \lambda \|_{L^2}   \frac{\|  \Div \bar{\mathbf{u}}\|_{L^2}}{\|  \Div \bar{\mathbf{u}}\|_{L^1}}  >  \| \lambda \|_{L^2}   \frac{\|  \Div \bar{\mathbf{u}}\|_{L^2}}{\|  \Div \bar{\mathbf{u}}\|_{L^1}}  - \sigma_0 $. Thus,  there exists $\sigma_0 > 0$ such that, for all $\sigma \geq \sigma_0$, $\|\Div \bar{\mathbf{u}}\|_{L^1} =0$.

The last condition imply that if $\sigma$ is larger than $\sigma_0$ then $\bar{\mathbf{u}}$ is feasible for the constrained problem \eqref{eq:prob_const}. Further, since $ \|\Div \bar{\mathbf{u}}\|_{L^1} =0$, it follows that 
\begin{align}
J (\bar{\mathbf{u}}) = J(\bar{\mathbf{u}}) + \sigma \|  \Div \bar{\mathbf{u}}\|_{L^1} 
 = J_\sigma(\mathbf{u})
 \leq 	J_\sigma(\widetilde{\mathbf{u}}) = J(\widetilde{\mathbf{u}}).
\end{align} 
 Therefore, by the  definition of $\widetilde{\mathbf{u}}$ we have $J (\widetilde{\mathbf{u}})=J (\bar{\mathbf{u}})$, where $\bar{\mathbf{u}}$ is a global minimum for problem \eqref{eq:prob_const}.
 \end{proof}

In view of the previous result, the minimization of $J_\sigma$ is called \emph{exact penalization} formulation of \eqref{eq:prob_const}.

\begin{remark}\label{re:bound_sigma0}
For numerical computation purposes it is important to derive a sharp estimation for $\sigma_0$, which can be used a priory  in order to guarantee exact penalization. We discuss this estimation for  $\sigma_0$ by using Theorem \ref{teo:exact2}. From the embedding $L^2(\Omega) \hookrightarrow L^1(\Omega)$ we have that $ \|  \Div \bar{\mathbf{u}}\|_{L^1}\leq |\Omega|^{\frac{1}{2}} \| \Div \bar{\mathbf{u}}\|_{L^2}$. Multiplying this inequality by $\|\lambda\|_{L^2}$, we have that 
  \begin{equation}\label{eq:rem_1}
    \|\lambda\|_{L^2}  |\Omega|^{-\frac{1}{2}}\|  \Div \bar{\mathbf{u}}\|_{L^1}\leq  \|\lambda\|_{L^2}  \| \Div \bar{\mathbf{u}}\|_{L^2}.
  \end{equation}
  From the proof of Theorem \ref{teo:exact2} it is clear that, if $\bar{\mathbf{u}}$ is the solution  of the unconstrained problem \eqref{minexact}, then any  $\sigma \leq \sigma_0$ satisfies inequality \eqref{eq:aux005} in a nontrivial manner, i.e., in particular
  \begin{equation*}
    \sigma_0 \|  \Div \bar{\mathbf{u}}\|_{L^1} \leq \|\lambda\|_{L^2} \|  \Div \bar{\mathbf{u}}\|_{L^2}.
  \end{equation*}
  Therefore, from  \eqref{eq:rem_1} we might consider either  $\sigma_0 \geq \|\lambda\|_{L^2}  |\Omega|^{-\frac{1}{2}} $ or $\sigma_0 \leq \|\lambda\|_{L^2}  |\Omega|^{-\frac{1}{2}}$. If the later  holds, we have found a $\bar{\sigma}=\|\lambda\|_{L^2}  |\Omega|^{-\frac{1}{2}} \geq \sigma_0 $. Then thanks to Theorem \ref{teo:exact2} we have that  $\Div \bar{\mathbf{u}}=0$ and equation \eqref{eq:aux005} is satisfied trivially. On the other hand, if $\sigma_0 \geq \|\lambda\|_{L^2}  |\Omega|^{-\frac{1}{2}} $, we arrive to a lower bound  for $\sigma_0$. Then, with both results we can establish the estimation 
\begin{equation}
 \sigma_0  \approx \|\lambda\|_{L^2}  |\Omega|^{-\frac{1}{2}}.	\label{eq:sig0estimate}
 \end{equation}
 In practice, by solving the system $\langle J'(\bar{\mathbf{u}}),\mathbf{u}\rangle_{\mathbf{H}^{-1},\mathbf{H}_0^1}=(\lambda, \Div (\mathbf{u}))_{L^2},$ for all $ \mathbf{u} \in \mathbf{H}_0^1(\Omega)$ we can obtain $\lambda$ in order to estimate $\sigma_0$.
\end{remark}

\subsection{Quadratic Penalization}
We finish this section with a discussion of a quadratic penalty approach using the $L^2$--norm. Our aim with this section is to show the differences and similarities with the exact penalization approach for the Bingham viscoplastic flow problem. The quadratic penalty function, involving the $L^2$--norm looks as follows
\begin{equation}\label{funquad}
J_\nu(\mathbf{u}):=J(\mathbf{u})+\frac{\nu}{2}\int_\Omega |\Div \mathbf{u}|^2\,dx, \tag{QP}
\end{equation}
where $\nu>0$ and $J(\mathbf{u})$ is given in \eqref{eq:prob_const}. Similarly, as we show in Section \ref{sec:expen}, we can state that $J_\nu(\mathbf{u})$ is a strictly convex and, furthermore, differentiable functional, which satisfies
\[
\underset{\|\mathbf{u}\|_{\mathbf{H}_0^1}\rightarrow\infty}{\lim} J_\nu(\mathbf{u})=+\infty.
\]
Thus, we can conclude that the following minimization problem has a unique solution, for each $\nu>0$.
\begin{equation}\label{minquad}
\underset{\mathbf{u}\in \mathbf{H}_0^1(\Omega)}{\min} J_\nu(\mathbf{u}).
\end{equation}
Let us note by $\mathbf{u}_\nu$ the solution of \eqref{minquad}. Further, let us recall that this function must be the solution of the following PDE, which corresponds to the Euler equation associated to the optimization problem. Therefore, $\mathbf{u}_\nu$ satisfies
\begin{equation}\label{unu}
\begin{array}{lll}
\displaystyle \mu\int_\Omega \mathcal{E}\mathbf{u}_\nu:\mathcal{E}\mathbf{v}\,dx +g\beta \int_\Omega \dfrac{\mathcal{E}\mathbf{u}_\nu:\mathcal{E}\mathbf{v}}{\max(g,\beta|\mathcal{E}\mathbf{u}_\nu|)}\,dx + \nu\int_\Omega (\Div\mathbf{u}_\nu)(\Div \mathbf{v})\,dx\vspace{0.2cm}\\\hspace{6cm}\displaystyle=\int_\Omega\mathbf{f}\cdot\mathbf{v}\,dx ,\,\,\forall \mathbf{v}\in \mathbf{H}_0^1(\Omega).
\end{array}
\end{equation}

On the other hand, let us recall that the solution of the Huber regularized Bingham problem \eqref{eq:prob_const} is also a solution for the following PDE
\begin{equation}\label{Binu}
\begin{array}{lll}
\displaystyle \mu\int_\Omega \mathcal{E}\mathbf{u}:\mathcal{E}\mathbf{v}\,dx +g\beta \int_\Omega \dfrac{\mathcal{E}\mathbf{u}:\mathcal{E}\mathbf{v}}{\max(g,\beta|\mathcal{E}\mathbf{u}|)}\,dx-\int_\Omega p \Div\mathbf{v}\,dx  \vspace{0.2cm}\\\hspace{5cm} \displaystyle=\int_\Omega\mathbf{f}\cdot\mathbf{v}\,dx,\,\,\forall \mathbf{v}\in \mathbf{H}_0^1(\Omega).
\end{array}
\end{equation}
Here,  $p\in L^2_0(\Omega)$ stands for the pressure and its existence is guaranteed by the de Rahm's Theorem. This fact is deeply analyzed in \cite{DlRG1}, and it will be discussed in the context of the presented exact penalty methods  in the next section.

From the mechanical point of view, \eqref{Binu} represents the flow of an incompressible Huber regularized Bingham flow, while \eqref{unu} represents the flow of a slightly incompressible Bingham flow. This means that $\Div\mathbf{u}_\nu\neq 0$, but nearly to zero. In fact, we expect that $\Div\mathbf{u}_\nu \rightarrow 0$, as $\nu\rightarrow\infty$ (see \cite{Temam}). This last convergence property is recognizable different from the expected behavior of the exact penalization presented in Section \ref{sec:expen}, where the incompressibility of the fluid is expected to hold for a (possibly large) finite value of the penalization parameter $\sigma$.

\begin{teo}\label{th:convunu}
Let $\{\mathbf{u}_\nu\}\subset \mathbf{H}_0^1(\Omega)$ be the sequence formed by the solutions of  \eqref{unu} associated to the parameter $\nu>0$. Moreover, let $\mathbf{u}\in \mathbf{H}_0^1(\Omega)$ be the solution of the variational problem \eqref{Binu}. Then, 
\begin{equation}\label{convunu}
\mathbf{u}_\nu\rightarrow \mathbf{u},\,\,\mbox{in $\mathbf{H}_0^1(\Omega)$, as $\nu\rightarrow\infty$}.
\end{equation}
\end{teo}
\begin{proof}
By subtracting the equation \eqref{Binu} from equation \eqref{unu}, we have that the variational problem holds
\begin{equation}\label{diunuv}
\begin{array}{lll}
\displaystyle \mu\int_\Omega \mathcal{E}(\mathbf{u}_\nu -\mathbf{u}):\mathcal{E}\mathbf{v}\,dx + g\beta\int_\Omega \left(\frac{\mathcal{E}\mathbf{u}_\nu}{\theta_\beta(\mathbf{u}_\nu)} -\frac{\mathcal{E}\mathbf{u}}{\theta_\beta(\mathbf{u})} \right):\mathcal{E}\mathbf{v}\,dx\vspace{0.2cm}\\ \hspace{2.5cm}\displaystyle+\nu\int_\Omega (\Div\mathbf{u}_\nu)(\Div\mathbf{v})\,dx = -\int_\Omega p\Div\mathbf{v}\,dx, \,\,\,  \forall \mathbf{v}\in \mathbf{H}_0^1(\Omega),
\end{array}
\end{equation}
where $\theta_\beta(\mathbf{u})$ was introduced in Lemma \ref{lem:reg}. Next, we take $\mathbf{v}= \mathbf{u}_\nu - \mathbf{u}$ in the above equation, where $\mathbf{u}$ fulfills the divergence--free condition $\Div\mathbf{u}=0$. Then, it follows that
\begin{equation}\label{difunu}
\begin{array}{lll}
\displaystyle\mu\int_\Omega|\mathcal{E}(\mathbf{u}_\nu-\mathbf{u})|^2 \,dx + g\beta\int_\Omega \left(\frac{\mathcal{E}\mathbf{u}_\nu}{\theta_\beta(\mathbf{u}_\nu)} -\frac{\mathcal{E}\mathbf{u}}{\theta_\beta(\mathbf{u})} \right):\mathcal{E}(\mathbf{u}_\nu- \mathbf{u})\,dx\vspace{0.2cm}\\ \hspace{5cm}\displaystyle+\nu\int_\Omega |\Div\mathbf{u}_\nu|^2\,dx= -\int_\Omega p \Div\mathbf{u}_\nu\,dx.
\end{array}
\end{equation}
Let us focus on the second term on the left hand side of \eqref{difunu}. Here, Lemma \ref{lem:reg} and the Cauchy-Schwarz inequality imply that
\begin{equation*}
\begin{array}{lll}
\displaystyle \int_\Omega \left(\frac{\mathcal{E}\mathbf{u}_\nu}{\theta_\beta(\mathbf{u}_\nu)} -\frac{\mathcal{E}\mathbf{u}}{\theta_\beta(\mathbf{u})} \right):\mathcal{E}(\mathbf{u}_\nu- \mathbf{u})\,dx \vspace{0.2cm}\\\hspace{2.7cm}
\displaystyle= \int_\Omega\left[\mathcal{E}\mathbf{u}\left(\frac{1}{\theta_\beta(\mathbf{u}_\nu)}-\frac{1}{\theta_\beta(\mathbf{u})}\right)+\frac{\mathcal{E}\mathbf{u}_\nu -\mathcal{E}\mathbf{u}}{\theta_\beta(\mathbf{u_\nu}_{\nonumber})}\right]:\mathcal{E}(\mathbf{u}_\nu-\mathbf{u})\,dx\vspace{0.2cm}\\\hspace{2.7cm}\displaystyle
=\int_\Omega\left[\frac{|\mathcal{E}(\mathbf{u}_\nu-\mathbf{u})|^2}{\theta_\beta(\mathbf{u}_\nu)}+\left(\frac{\theta_\beta(\mathbf{u})-\theta_\beta(\mathbf{u}_\nu)}{\theta_\beta(\mathbf{u}_\nu)\theta_\beta(\mathbf{u})}\right)\,\mathcal{E}\mathbf{u}:\mathcal{E}(\mathbf{u}_\nu-\mathbf{u})\right]\,dx\vspace{0.2cm}\\\hspace{2.7cm}\displaystyle
=\int_\Omega\frac{1}{\theta_\beta(\mathbf{u}_\nu)}\left[|\mathcal{E}(\mathbf{u}_\nu-\mathbf{u})|^2-\frac{\theta_\beta(\mathbf{u}_\nu)-\theta_\beta(\mathbf{u})}{\theta_\beta(\mathbf{u})}\,\mathcal{E}\mathbf{u}:\mathcal{E}(\mathbf{u}_\nu-\mathbf{u})\right]\,dx\vspace{0.2cm}\\\hspace{2.7cm}\displaystyle
\geq \int_\Omega\frac{1}{\theta_\beta(\mathbf{u}_\nu)}\left[|\mathcal{E}(\mathbf{u}_\nu-\mathbf{u})|^2-\beta \frac{|\mathcal{E}\mathbf{u}_\nu-\mathcal{E}\mathbf{u}|}{\theta_\beta(\mathbf{u})}\,|\mathcal{E}\mathbf{u}||\mathcal{E}(\mathbf{u}_\nu-\mathbf{u})|\right]\,dx.
\end{array}
\end{equation*}

Thus, by taking into account \eqref{thleqbeta}, it holds that 
\begin{equation*}
\begin{array}{lll}
\displaystyle\int_\Omega \left(\frac{\mathcal{E}\mathbf{u}_\nu}{\theta_\beta(\mathbf{u}_\nu)} -\frac{\mathcal{E}\mathbf{u}}{\theta_\beta(\mathbf{u})} \right):\mathcal{E}(\mathbf{u}_\nu- \mathbf{u})dx\vspace{0.2cm}\\\hspace{3cm}\displaystyle
\geq \int_\Omega\frac{1}{\theta_\beta(\mathbf{u}_\nu)}\left[|\mathcal{E}(\mathbf{u}_\nu-\mathbf{u})|^2-\beta |\mathcal{E}(\mathbf{u}_\nu-\mathbf{u})|^2\frac{|\mathcal{E}\mathbf{u}|}{\theta_\beta(\mathbf{u})}\,\right]dx \geq0.
\end{array}
\end{equation*}
Next, inserting the last relation in \eqref{difunu}, and using Korn's inequality, we conclude the existence of a constant $C>0$ such that 
\begin{equation}\label{difunu2}
\begin{array}{lll}
\displaystyle
C\|\mathbf{u}_\nu -\mathbf{u}\|^2_{\mathbf{H}_0^1} + \nu\int_\Omega |\Div\mathbf{u}_\nu|^2\,dx \leq \mu \int_\Omega |\mathcal{E}(\mathbf{u}_\nu-\mathbf{u})|^2\,dx\vspace{0.2cm}\\\hspace{1cm} \displaystyle + g\beta \int_\Omega \left(\frac{\mathcal{E}\mathbf{u}_\nu}{\theta_\beta(\mathbf{u}_\nu)} -\frac{\mathcal{E}\mathbf{u}}{\theta_\beta(\mathbf{u})} \right):\mathcal{E}(\mathbf{u}_\nu- \mathbf{u})\,dx + \nu\int_\Omega |\Div\mathbf{u}_\nu|^2\,dx\vspace{0.2cm}\\\hspace{4cm}=\displaystyle-\int_\Omega p \Div\mathbf{u}_\nu\,dx\leq \|p\|_{L^2}\|\Div\mathbf{u}_\nu\|_{L^2}.
\end{array}
\end{equation}
Further, the Young's inequality implies that
\[
\|p\|_{L^2}\|\Div\mathbf{u}_\nu\|_{L^2}\leq \frac{\nu}{2}\|\Div\mathbf{u}_\nu\|^2_{L^2} + \frac{1}{2\nu}\|p\|^2_{L^2},
\]
for $\nu>0$. By using this inequality in \eqref{difunu2}, we obtain that
\[
C\|\mathbf{u}_\nu -\mathbf{u}\|^2_{\mathbf{H}_0^1} + \nu\int_\Omega |\Div\mathbf{u}_\nu|^2\,dx \leq \frac{\nu}{2} \int_\Omega |\Div\mathbf{u}_\nu|^2\,dx+ \frac{1}{2\nu}\|p\|_{L^2},
\]
which yields that
\[
C\|\mathbf{u}_\nu -\mathbf{u}\|^2_{\mathbf{H}_0^1} + \frac{\nu}{2}\int_\Omega |\Div\mathbf{u}_\nu|^2\,dx \leq \frac{1}{2\nu}\|p\|_{L^2}.
\]
Consequently,
\[
C\|\mathbf{u}_\nu -\mathbf{u}\|^2_{\mathbf{H}_0^1}  \leq \frac{1}{2\nu}\|p\|_{L^2}.
\]
Taking the limit $\nu\rightarrow\infty$, we get that
$\mathbf{u}_\nu \rightarrow \mathbf{u} \in \mathbf{H}_0^1(\Omega).$
\end{proof}

\section{Recovering the fluid's pressure}
It is well known that the existence of a function $p\in L^2_0(\Omega)$ solving the PDE given in \eqref{Binu} is guaranteed by the de Rham's theorem, see  \cite[p.14]{Temam}. This function can be seen as a Lagrange multiplier associated to the restriction $\Div\mathbf{u}=0$.  Similarly, we will argue that in the case of the penalty methods presented in this paper, there exists a function $p$ playing the role for the pressure of the fluid. This existence result is reached trough a convergence argument in the divergence of the vector field $\mathbf{u}_\nu$ that we analyze for quadratic and exact penalizations. Naturally, in the case of exact penalization, the associated analysis is more challenging due to the nondiferentaility of the $L^1$--norm and the presence of its associated subgradients. 

\subsection{Pressure in the Quadratic Penalization}
In order to start this analysis, we rewrite the equation \eqref{diunuv} as follows
\begin{equation*}
\begin{array}{lll}
\displaystyle \mu\int_\Omega \mathcal{E}(\mathbf{u}_\nu -\mathbf{u}):\mathcal{E}\mathbf{v}\,dx + g\beta\int_\Omega \left(\frac{\mathcal{E}\mathbf{u}_\nu}{\theta_\beta(\mathbf{u}_\nu)} -\frac{\mathcal{E}\mathbf{u}}{\theta_\beta(\mathbf{u})} \right):\mathcal{E}\mathbf{v}\,dx\vspace{0.2cm}\\ \hspace{2.5cm}\displaystyle+\int_\Omega (\nu\Div\mathbf{u}_\nu+ p)(\Div\mathbf{v})\,dx = 0, \,\,\,  \forall \mathbf{v}\in \mathbf{H}_0^1(\Omega),
\end{array}
\end{equation*}
which, by following \cite[Th. 6.14-1]{Ciarlet}, yields that
\begin{equation}\label{wdifunu}
\begin{array}{lll}
\displaystyle \mu\int_\Omega \mathcal{E}(\mathbf{u}_\nu -\mathbf{u}):\mathcal{E}\mathbf{v}\,dx + g\beta\int_\Omega \left(\frac{\mathcal{E}\mathbf{u}_\nu}{\theta_\beta(\mathbf{u}_\nu)} -\frac{\mathcal{E}\mathbf{u}}{\theta_\beta(\mathbf{u})} \right):\mathcal{E}\mathbf{v}\,dx\vspace{0.2cm}\\ \hspace{2.5cm}\displaystyle-\langle\nabla (\nu\Div\mathbf{u}_\nu+ p)\,,\,\mathbf{v}\rangle_{\mathbf{H}^{-1},\mathbf{H}_0^1} = 0, \,\,\,  \forall \mathbf{v}\in \mathbf{H}_0^1(\Omega).
\end{array}
\end{equation}
Next, by taking limit $\nu\rightarrow\infty$ in \eqref{wdifunu}, and by considering the continuity of $\theta_\beta$ and Theorem \ref{th:convunu}, we obtain that
\begin{equation}\label{gradunu}
\frac{\partial}{\partial x_i} \left(\nu \Div\mathbf{u}_\nu\right)\rightarrow -\frac{\partial p}{\partial x_i},\,\,\mbox{ in $\mathbf{H}^{-1}(\Omega)$ and for all $i=1,\ldots,n$.}
\end{equation}
On the other hand, since $p\in L_0^2(\Omega)$, $\mathbf{u}_\nu=0$ on $\partial\Omega$, and thanks to the divergence theorem, we have that
\[
\begin{array}{lll}
\int_\Omega (p+\nu\Div\mathbf{u}_\nu)\,dx&=& \int_\Omega p\,dx + \nu \int_\Omega \Div\mathbf{u}_\nu\,dx\vspace{0.2cm}\\ &=&  \int_{\partial\Omega} \mathbf{u}_\nu\cdot\vec{\mathbf{n}}\,dx =0.
\end{array}
\]
Therefore, \cite[Lem. 6.1, pp. 100]{Temam} yields that
\[
\|p+\nu\Div\mathbf{u}_\nu\|_{L^2}\leq \sum_{i=1}^n \left\|\frac{\partial}{\partial x_i} \left(p+\nu \Div\mathbf{u}_\nu\right)\right\|_{\mathbf{H}^{-1}(\Omega)},
\]
which, thanks to \eqref{gradunu}, implies that
\[
-\nu \Div\mathbf{u}_\nu \rightarrow p,\,\,\mbox{as $\nu\rightarrow\infty$}.
\]

\subsection{Pressure in the Exact Penalization}
Now, we turn our discussion to the pressure recovery in the context of exact penalization. We start by recalling that the optimality condition \eqref{eq:opt_cond_nondiff} implies that the solution $\mathbf{u}_\sigma$ of the exact penalized problem \eqref{minexact}, is characterized by the existence of a function $\zeta\in L_0^2(\Omega)$, such that $|\zeta|\leq \sigma$ a.e. in $\Omega$, satisfying:
\[
\left\langle-J'(\mathbf{u}_\sigma),\mathbf{v}\right\rangle_{\mathbf{H}^{-1},\mathbf{H}_0^1} = \left(\zeta,\Div\mathbf{v}\right)_{L^2},\,\forall\,\mathbf{v}\in \mathbf{H}_0^1(\Omega).
\]
By using the Fr\'echet derivative of $J$ given in \eqref{eq:Jp}, we can observe that this equation is equivalent to the following PDE: 
\begin{equation}\label{usigma}
\begin{array}{lll}
\displaystyle\mu\int_\Omega \mathcal{E}\mathbf{u}_\sigma:\mathcal{E}\mathbf{v}\,dx + g\beta\int_\Omega \frac{\mathcal{E}\mathbf{u}_\sigma:\mathcal{E}\mathbf{v}}{\theta_\beta(\mathbf{u}_\sigma)}\,dx -\int_\Omega\mathbf{f}\cdot\mathbf{v}\,dx\vspace{0.2cm}\\\hspace{5cm}= \displaystyle-\int_\Omega \zeta \Div\mathbf{v}\,dx, \,\forall\mathbf{v}\in \mathbf{H}_0^1(\Omega).
\end{array}
\end{equation}
By subtracting equation \eqref{Binu} from \eqref{usigma} we have that
\begin{equation}\label{usigdifu}
\begin{array}{lll}
\displaystyle\mu\int_\Omega \mathcal{E}(\mathbf{u}_\sigma-\mathbf{u}):\mathcal{E}\mathbf{v}\,dx + g\beta\int_\Omega \left(\frac{\mathcal{E}\mathbf{u}_\sigma}{\theta_\beta(\mathbf{u}_\sigma)}- \frac{\mathcal{E}\mathbf{u}}{\theta_\beta(\mathbf{u})}\right):\mathcal{E}\mathbf{v}\,dx \vspace{0.2cm}\\\hspace{5cm}=\displaystyle- \int_\Omega\left(\zeta+p\right)\Div \mathbf{v}\,dx, \,\forall\mathbf{v}\in \mathbf{H}_0^1(\Omega).
\end{array}
\end{equation}
Taking $\mathbf{v}=\mathbf{u}_\sigma-\mathbf{u}$ in \eqref{usigdifu}, we obtain that
\[
\begin{array}{lll}
\displaystyle\mu\int_\Omega |\mathcal{E}(\mathbf{u}_\sigma-\mathbf{u})|^2\,dx + g\beta\int_\Omega \left(\frac{\mathcal{E}\mathbf{u}_\sigma}{\theta_\beta(\mathbf{u}_\sigma)}- \frac{\mathcal{E}\mathbf{u}}{\theta_\beta(\mathbf{u})}\right):\mathcal{E}(\mathbf{u}_\sigma-\mathbf{u})\,dx \vspace{0.2cm}\\\hspace{6cm}=\displaystyle- \int_\Omega\left(\zeta+p\right)\Div(\mathbf{u}_\sigma-\mathbf{u})\,dx.
\end{array}
\]
Next, we already proved that the second term in the left hand side is non negative. Therefore, this last equation, together with  Korn's inequality and the incompressibility of $\mathbf{u}$ imply the existence of a constant $C>0$, such that
\begin{equation}\label{convusigma}
C\|\mathbf{u}_\sigma -\mathbf{u}\|^2_{\mathbf{H}^1_0}\leq -\int_\Omega (p+\zeta)\Div \mathbf{u}_\sigma\,dx.
\end{equation}
On the other hand, Theorem \ref{teo:exact2} states that there exists $\sigma_0>0$ such that $\Div\mathbf{u}_\sigma=0$, for all $\sigma\geq \sigma_0$. Therefore, \eqref{convusigma} implies that 
\begin{equation}\label{usigconvu}
\mathbf{u}_\sigma = \mathbf{u},\,\,\mbox{ in $\mathbf{H}_0^1(\Omega)$, for all $\sigma\geq \sigma_0$}.
\end{equation}
Also, equation \eqref{convusigma} and the Korn's inequality imply that
\[
\int_\Omega |\mathcal{E}\mathbf{u}_\sigma -\mathcal{E}\mathbf{u}|^2\,dx=0,\,\,\mbox{ for all $\sigma>\sigma_0$},
\]
which, thanks to \cite[Prop. 2.10]{Giaquinta}, yields that $|\mathcal{E}\mathbf{u}_\sigma-\mathcal{E}\mathbf{u}|^2=0$, a.e. in $\Omega$. Consequently, we can infer that
\begin{equation}\label{eqEusig=0}
\mathcal{E}\mathbf{u}_\sigma=\mathcal{E}\mathbf{u},\,\,\mbox{a.e. in $\Omega$ and for all $\sigma>\sigma_0$}
\end{equation}
Next, we establish pointwise bounds for $\theta_\beta(\mathbf{u}_\sigma)- \theta_\beta(\mathbf{u})$ on the next four disjoint sets: $E_\beta^{\mathbf{u}_\sigma}\cap E_\beta^\mathbf{u}$, $E_\beta^{\mathbf{u}_\sigma}\cap I_\beta^\mathbf{u}$, $E_\beta^\mathbf{u}\cap I_\beta^{\mathbf{u}_\sigma}$ and $I_\beta^{\mathbf{u}_\sigma}\cap I_\beta^\mathbf{u}$. Note that these sets were defined in Lemma \ref{lem:reg}.

In $E_\beta^{\mathbf{u}_\sigma}\cap E_\beta^\mathbf{u}$, we directly have that $\theta_\beta(\mathbf{u}_\sigma)- \theta_\beta(\mathbf{u})=0$. In $E_\beta^{\mathbf{u}_\sigma}\cap I_\beta^\mathbf{u}$, we have that $\theta_\beta(\mathbf{u}_\sigma)- \theta_\beta(\mathbf{u})=g-\beta|\mathcal{E}\mathbf{u}|\leq0$. On the other hand, thanks to \eqref{eqEusig=0}, we have that $g-\beta|\mathcal{E}\mathbf{u}| = g-\beta|\mathcal{E}\mathbf{u}_\sigma|\geq 0$. Thus, we have that
\[
\theta_\beta(\mathbf{u}_\sigma)- \theta_\beta(\mathbf{u})= 0,\,\,\mbox{a.e. in $E_\beta^{\mathbf{u}_\sigma}\cap I_\beta^\mathbf{u}$ and for all $\sigma>\sigma_0$}.
\]
In $E_\beta^\mathbf{u}\cap I_\beta^{\mathbf{u}_\sigma}$, we have that $\theta_\beta(\mathbf{u}_\sigma)- \theta_\beta(\mathbf{u})=\beta|\mathcal{E}\mathbf{u}_\sigma|-g\geq 0$. Finally, in $I_\beta^{\mathbf{u}_\sigma}\cap I_\beta^\mathbf{u}$, we have that  $\theta_\beta(\mathbf{u}_\sigma)- \theta_\beta(\mathbf{u})=\beta|\mathcal{E}\mathbf{u}_\sigma|-\beta|\mathcal{E}\mathbf{u}|$. This expression, together with  \eqref{eqEusig=0}, implies that 
\[
\theta_\beta(\mathbf{u}_\sigma)- \theta_\beta(\mathbf{u})= 0,\,\,\mbox{a.e. in $I_\beta^{\mathbf{u}_\sigma}\cap I_\beta^\mathbf{u}$ and for all $\sigma>\sigma_0$}.
\]
Summarizing, since the four given sets provide a disjoint partitioning of $\Omega$, we can conclude that $\theta_\beta(\mathbf{u}_\sigma)- \theta_\beta(\mathbf{u})\geq 0$, a.e. in $\Omega$ and for all $\sigma>\sigma_0$. Now, Lemma \ref{lem:reg} implies that $\theta_\beta(\mathbf{u}_\sigma)- \theta_\beta(\mathbf{u})< |\mathcal{E}\mathbf{u}_\sigma - \mathcal{E}\mathbf{u}|=0$, a.e. in $\Omega$ and for all $\sigma>\sigma_0$. Therefore, we can state that
\begin{equation}\label{theta=0}
\theta_\beta(\mathbf{u}_\sigma)- \theta_\beta(\mathbf{u})=0, \,\,\mbox{a.e. in $\Omega$ and for all $\sigma>\sigma_0$}.
\end{equation}
We now turn our attention to the recovery of pressure for the flow. First, note that \cite[Th. 6.14-1]{Ciarlet} allows us to rewrite the equation \eqref{usigdifu} as follows
\[
\begin{array}{lll}
\displaystyle\mu\int_\Omega \mathcal{E}(\mathbf{u}_\sigma-\mathbf{u}):\mathcal{E}\mathbf{v}\,dx + g\beta\int_\Omega \left(\frac{\mathcal{E}\mathbf{u}_\sigma}{\theta_\beta(\mathbf{u}_\sigma)}- \frac{\mathcal{E}\mathbf{u}}{\theta_\beta(\mathbf{u})}\right):\mathcal{E}\mathbf{v}\,dx \vspace{0.2cm}\\\hspace{5cm}=\langle\nabla(\zeta+p)\,,\,\mathbf{v}\rangle_{\mathbf{H}^{-1},\mathbf{H}_0^1}, \,\forall\mathbf{v}\in \mathbf{H}_0^1(\Omega).
\end{array}
\]
From this equation, \eqref{convusigma} and \eqref{theta=0}, we conclude, for all $\sigma\geq \sigma_0$, that
\begin{equation}\label{gradzeta}
-\frac{\partial}{\partial x_i}\zeta = \frac{\partial}{\partial x_i}p,\,\,\mbox{in $\mathbf{H}^{-1}(\Omega)$, for all $i=1,\ldots,n$.}
\end{equation}
Finally, since $p,\zeta\in L^2_0(\Omega)$, we have that
\[
\int_\Omega (\zeta + p)\,dx=0,
\]
which implies, together with \cite[Lem. 6.1, pp. 100]{Temam} and \eqref{gradzeta}, that
\[
\|p+\zeta\|_{L^2} \leq \sum_{i=1}^n \left\|\frac{\partial}{\partial x_i}(\zeta +p)\right\|_{\mathbf{H}^{-1}} = 0,\,\,\mbox{for $\sigma\geq \sigma_0$}.
\]
Consequently, $-\zeta=p$, for all $\sigma\geq\sigma_0$.
\section{Exact Penalization Second Order  Method }
We are in place to describe a descent algorithm using second--order information for solving the nonsmooth problem generated by the exact penalization formulation. 

This algorithm follows ideas from the nonsmooth method designed for finite dimensions in \cite{Merino}. In particular, the generalized differentiation is used to enrich second--order information of the smooth part of the cost in a functional space setting. Naturally, the extension of this method to the numerical solution of viscoplastic fluids entails several analytical and numerical challenges. One crucial step of the algorithm consist of the utilization of approximated second--order information for the computation of the descent direction. We will see that this procedure allows us to compute descent directions directly in the space $V$. With this novel feature, the algorithm solves the non-constrained problem \eqref{minexact} ensuring that, at each iteration $k$, the descent direction $\mathbf{w}_k$ satisfies $\Div \mathbf{w}_k\approx 0$, with prescribed precision. 
\subsection{First order information}
Recall that the regular term $J(\mathbf{u})$ of problem \eqref{minexact} involves a Huber regularization function of the Frobenius norm $\Psi(\mathcal{E}\mathbf{u}) $. Moreover, let us introduce the active set:
$$E_\beta:= \left\lbrace x \in \Omega : |\mathcal{E}\mathbf{u}(x)| < \frac{g}{\beta} \right\rbrace .$$
In turn, the inactive set corresponds to $\Omega \backslash E_\beta$. 
Hence,  the Fr\'echet derivative \eqref{eq:Jp} can be rewriten as follows:
\begin{equation}\label{eq:Jp_max}
\begin{array}{lll}
\left\langle J'(\mathbf{u}),\mathbf{v} \right\rangle_{\mathbf{H}^{-1}, \mathbf{H}^1_0} = \displaystyle 2 \mu \int_{\Omega} \mathcal{E}\mathbf{u}  : \mathcal{E}\mathbf{v} \, dx  +\displaystyle g \int_{\Omega \backslash E_\beta} \frac{\mathcal{E}\mathbf{u}  : \mathcal{E}\mathbf{v} }{|\mathcal{E}\mathbf{u}|} \,\, dx \vspace{0.2cm}\\\hspace{5.5cm}+  \beta \displaystyle \int_{ E_\beta} \mathcal{E}\mathbf{u} : \mathcal{E}\mathbf{v}  \,\, dx -  \int_{\Omega} \mathbf{f} \cdot\mathbf{u}  \,\, dx.
\end{array}
\end{equation}
%
In view of the nondifferentiability of the objective functional $J_\sigma$ because of the term $h(u):=\sigma \| \Div\mathbf{u} \|_{L^1}$, let us introduce the following notion of steepest descent direction for convex functions based on the projection of the zero element on the subdifferential of the function.

\begin{prop}
Let $\mathcal{H}$ be a real Hilbert space and $f: \mathcal{H} \rightarrow ]- \infty, + \infty]$ be proper and convex, and suppose that $x \in \cont f \backslash \argmin f$. Where $\cont f$ is the domain of continuity of $f$ and $\argmin f$ is the set of global minimizers of f. Set $u$ as the projection of $0$ on the subdifferential $\partial f(x)$ denoted by $ u=P_{\partial f(x)} 0$,  and set $z=-\frac{u}{\|u\|}$. Then, $z$ is the unique minimizer of $f'(x,\cdot)$ over the open unit ball  $B(0,1)$ in $\mathcal{H}$.
\end{prop}
\begin{proof}
See \cite[Prop. 17.22]{Combettes}
\end{proof}

By taking into account the last proposition, we introduce the \emph{steepest descent direction} as the solution of the problem  $\boldsymbol{\bar{d}} =P_{\partial J_{\sigma}(\mathbf{u})} 0$, which is equivalent to: 
\begin{equation}\label{eq:minsubnorm}
\boldsymbol{\bar{d}}=\underset{ \boldsymbol{d} \in \partial J_{\sigma}(\mathbf{u}) }{\argmin} \|\boldsymbol{d}\|_{\mathbf{H}^{-1}}, 
\end{equation}

where $\partial J_{\sigma}(\mathbf{u}) = J'(\mathbf{u}) + \partial h(\mathbf{u}) $. We observe that $\boldsymbol{d} \in \partial J_{\sigma}(\mathbf{u}) $ implies that $\boldsymbol{d} = J'(\mathbf{u}) + \eta $ with $\eta \in \partial h(\mathbf{u})$. Then, from \eqref{eq:subdiff} we can rewrite problem \eqref{eq:minsubnorm} in the form
 \begin{equation}\label{eq:minnorm_final}
  \boldsymbol{\bar{d}}=\underset{\zeta \in \sigma \partial \|\cdot\|(\Div u) }{\argmin} \|  J'(u)  -  \text{ grad} (\zeta) \|_{\mathbf{H}^{-1}}.
 \end{equation}

Let us notice that finding the solution $\boldsymbol{\bar{d}}$ of  problem \eqref{eq:minsubnorm} constitutes a box-constrained optimization problem. Therefore, in the next section we  develop an efficient strategy for computing $\boldsymbol{\bar{d}}$.

\subsection{Second Order information: Generalized Differentiability and Semismoothness for Superposition Operators}

We emphazise that the function $h(\mathbf{u})= \sigma\int_\Omega| \Div\mathbf{u} | \, dx$ is nondifferentiable on the kernel of the $\Div$ operator. Hence, in order to calculate first and second order information we will use the Huber regularization analogous to the one presented in Section \ref{sec:1}. Therefore, in this section we study the  second order information available in the weak sense. This information is associated to  the functional $J_{\sigma}(\mathbf{u})=J(\mathbf{u})+ h(\mathbf{u})$, by using the notion of semismoothness for superposition operators. The second order information for functional $J(\mathbf{u})$,  in the finite dimensional case, was presented in \cite{GALO}.

Let us introduce the smoothing function $|\cdot |_{\gamma}: \mathbb{R} \rightarrow \mathbb{R}$ with $\gamma > > 0$ such that
\begin{equation*}
|z|_{\gamma}=\begin{cases}
\sigma |z | -\frac{\sigma^2}{2 \gamma} & if   \,\, |z| \geq \frac{\sigma}{\gamma} \\
\frac{\gamma}{2} |z|^2 & if \,\, |z| < \frac{\sigma}{\gamma}.
\end{cases}
\end{equation*}
This function is known as the Huber (local) regularization of the absolute value (see \cite[Sec. 2.2]{Gonzalez-Andrade3}). In this case, we define the $\Div$-active set as follows:
\begin{equation}\label{eq:div_active_set}
  A_\gamma:= \left\lbrace x \in \Omega : |\Div \mathbf{u}(x)| < \frac{\sigma}{\gamma} \right\rbrace,
\end{equation}
and the $\Div$-inactive set as $\Omega \backslash A_\gamma $. Let us define the regularized mapping $h_\gamma(\mathbf{u})= \sigma\int_\Omega  | \Div\mathbf{u}|_\gamma \, dx$. Analogously as in the case of $J$, its first Fr\'echet-derivative $h'_\gamma(\mathbf{u})$ is given by:
\begin{equation*}
\begin{array}{rl}
\displaystyle \left\langle h'_\gamma(\mathbf{u}),\mathbf{v} \right\rangle_{\mathbf{H}^{-1}, \mathbf{H}^1_0} & =  \displaystyle \sigma \int_{\Omega \setminus A_\gamma} \frac{(\Div \mathbf{u},\Div\mathbf{v})}{|\Div\mathbf{u}|} \,\, dx + \int_{A_\gamma} \gamma (\Div\mathbf{u}, \Div\mathbf{v}) \,\, dx  \vspace{0.2cm}\\
&  =  \displaystyle \sigma \int_{\Omega} \frac{\gamma (\Div\mathbf{u},\Div\mathbf{v})}{\max(\sigma, \gamma |\Div\mathbf{u}|)} \,\, dx,\,\,\, \forall \mathbf{v}\in \mathbf{H}_0^1(\Omega). 
\end{array}
\end{equation*}

However, neither  $J(\mathbf{u})$ nor $h_\gamma(\mathbf{u})$ are twice Fr\'echet differentiable because of the presence of the non-differentiable $\max$ function. Therefore, we consider the notion of semismoothness for superposition operators and the generalized  differential  developed in \cite[Ch. 3]{Ulbrich}. Using these concepts of generalized differentiation, we can compute   second--order information associated to the nondiferentiabilities in $J'(\mathbf{u})$ and $h'_\gamma(\mathbf{u})$.

\subsubsection{Some results for Semismooth Superposition Operators}

 Following the work and notation introduced in \cite{Ulbrich}, we consider Nemytskii (or superposition) operators  $\Phi: Y \rightarrow L^r(\Omega),$ defined by  
 \begin{equation}\label{eq:Nemytskii}
  \Phi(u)(x)=\phi(F(u)(x))
 \end{equation} 
 for almost all $x$ on $\Omega$. Here, $Y$ is a real Banach space. The mappings $\phi:\mathbb{R}^{m} \rightarrow \mathbb{R}$ and $F: Y \rightarrow  \prod_{i=1}^{m} L^{r_i}(\Omega)$, with $1 \leq r \leq r_i < \infty$, satisfy the following conditions \cite{Ulbrich}[Assump. 3.32]:\\
 There are $1\leq r \leq r_i < q_i \leq \infty $, for $ 1\leq i \leq m$, such that:
 \begin{enumerate}[a)]
   \item The mapping $F: Y \rightarrow  \prod_{i=1}^{m} L^{r_i}(\Omega)$ is continuously Fr\'echet differentiable.\label{asum:1}
   \item The mapping $Y \ni u \mapsto F(u) \in  \prod_{i=1}^{m} L^{q_i}(\Omega)$ is locally Lipschitz continuous, i.e., for all $u \in Y$ there exists an open neighborhood $V(u)$ and a constant $L_F(V)$ such that:\label{asum:2}
   \begin{equation*}
     \sum_i \| F_i(u_1) - F_i(u_2)\|_{L^{q_i}} \leq L_F\| u_1 - u_2\| \,\,\, \forall u_1, u_2 \in U.
   \end{equation*}
   \item The function $\phi:\mathbb{R}^{m} \rightarrow \mathbb{R}$ is Lipschitz continuos, i.e.,\label{asum:3}
   \[ 
   |\phi(x_1) - \phi(x_2)|\leq L_{\phi} \|x_1 - x_2\|, \,\, \forall x_1, x_2 \in \mathbb{R}^m.
   \]
   \item $\phi$ is semismooth. \label{asum:4}
 \end{enumerate}
In addition, we use the following definition of semismoothness \cite{Ulbrich}[Def. 3.48].
\begin{defi}\label{def:semismooth}
The operator $\Phi$  is called semismooth at $y \in Y$ if it satisfies
\begin{equation}
\displaystyle \sup_{G \in \partial^{\circ} \Phi(u+h)} \|\Phi(u+h)-\Phi(u)-Gh\|_{L^r} = o(\|h\|_Y), \hspace{2mm} \text{ as } \, h \rightarrow 0 \, \text{ in } Y,
\end{equation}
where  $\partial^{\circ} \Phi$ corresponds to the generalized differential:

\begin{equation}\label{eq:aux007}
 \partial^{\circ} \Phi(u)= \left\lbrace 
 \begin{array}{l}
  G \in \mathcal{L}(Y, L^r)  \text{ such that } \, G:v \mapsto \sum_i M_i(u)(F'_i(u)v), \\
  \text{where  $M(u)$ is a measurable selection of Clarke's generalized Jacobian }\, \partial \phi(F(u))  
 \end{array}
 \right\rbrace
\end{equation}
\end{defi}
Let us point out that, by abuse of notation, we use  the same symbol for Clarke's generalized Jacobian and the subdifferential of convex functions in equation \eqref{eq:opt_cond_exat_p}.

Following this definition, we analyze the  Nemytskii operators in the lemmas below.
 \begin{lema}\label{prop:semismooth_h}
Let $\Phi:\mathbf{H}_0^1(\Omega) \rightarrow L^1(\Omega) $ be a Nemytskii  operator given by $\Phi(\mathbf{u})(x)=\phi(F(\mathbf{u})(x)=\displaystyle \sigma \gamma \frac{  \Div \mathbf{u}(x)}{\max(\sigma, \gamma |\Div \mathbf{u}(x)|)}$.  Then, $\Phi$ is semismooth and, $G(\mathbf{u}) \mathbf{w} \in  \partial^{\circ} \Phi(\mathbf{u})\mathbf{w} $  is given by:

\begin{equation}\label{eq:hpp1}
 \begin{array}{lll}
 G(\mathbf{u})\mathbf{w}\vspace{0.2cm}=\begin{cases} \displaystyle  \sigma\frac{\Div\mathbf{w}}{|\Div\mathbf{u}|} \,  - \sigma \displaystyle   \frac{(\Div\mathbf{u} \,, \Div\mathbf{w})\Div\mathbf{u}}{|\Div \mathbf{u}|^3} & \text{ if  } \, \gamma | \Div\mathbf{u}(x)| \geq \sigma  \\
 \displaystyle  \gamma (\Div\mathbf{w}) & \text{ if } \,\, \gamma | \Div\mathbf{u}(x)| < \sigma.
\end{cases}
 \end{array}
\end{equation}
 \end{lema}
 \begin{proof}

  Thanks to  \cite[Th. 3.49]{Ulbrich} and conditions \ref{asum:1}) - \ref{asum:4}) above, $\Phi$  is semismooth on  $\mathbf{H}^1_0(\Omega)$ in the sense of Definition \ref{def:semismooth}. In our context, we have that  $F:\mathbf{H}^1_0(\Omega) \rightarrow L^1(\Omega)$ is defined by $F(\mathbf{u})=\Div\mathbf{u}$ and $\phi: \mathbb{R} \rightarrow \mathbb{R}$ is given by $\phi(a)=\frac{ a}{\max(\sigma, \gamma |a|)}$.

This proof is based on showing that conditions \ref{asum:1}) - \ref{asum:4}) are satisfied, i.e., we have to prove that: $\Div: \mathbf{H}^1_0(\Omega) \rightarrow L^1(\Omega)$ is continuously Fr\'echet differentiable and locally Lipschitz continuous and that $\phi: \mathbb{R} \rightarrow \mathbb{R}$ is Lipschitz continuous and semismooth.  Since $\Div$ is a continuous linear operator in $L_0^2(\Omega)$ then, it is Fr\'echet differentiable and its derivative is the same operator $\Div$ (see \cite[Th. 6.14-1]{Ciarlet}). Additionally, the continuous embedding $L^2_0(\Omega) \subset L^1(\Omega)$ implies that the operator $\Div $ is also continuously Fr\'echet differentiable on $L^1(\Omega)$. It is clear that $\Div$ is Lipschitz continuous, hence conditions \ref{asum:1}) - \ref{asum:2})  are verified. 

The Lipschitz continuity and semismoothness proof of $\phi$  is deferred to the Appendix. Then, conditions  \ref{asum:3}) - \ref{asum:4})  are satisfied. 

Next, we obtain a measurable selection $M( \mathbf{ u})$ of Clarke's generalized Jacobian $ \partial \phi(\Div \mathbf{u}) $  as follows:
\begin{equation}
  M(\mathbf{u}(x))=\begin{cases}\label{eq:aux0006}
\displaystyle\sigma \frac{1}{|\Div\mathbf{u}(x)|} \,  - \sigma\displaystyle   \frac{(\Div\mathbf{u} (x),\Div\mathbf{u} (x))}{|\Div\mathbf{u}(x)|^3}, & if   \,\, \gamma | \Div\mathbf{u}(x)| \geq \sigma\vspace{0.2cm} \\ \gamma, & if \,\, \gamma  |\Div\mathbf{u}(x)| < \sigma.
\end{cases}
 \end{equation}
We refer the reader to  the Appendix for the proof. Then, Definition \ref{def:semismooth} - equation \eqref{eq:aux007} and $F'(\mathbf{u})\mathbf{w}= \Div \mathbf{w}  $  yields that $G(\mathbf{u})\mathbf{w}= M(\mathbf{u})(F'(\mathbf{u})\mathbf{w})=  M(\mathbf{u}) \Div \mathbf{w}$. Finally, from \eqref{eq:aux0006}  we obtain that $G(\mathbf{u})\mathbf{w} \in  \partial^{\circ} \Phi(\mathbf{u}) (\mathbf{w})$ is given by \eqref{eq:hpp1}.
 \end{proof}

To prove that $J'(\mathbf{u})$ is semismooth we introduce the Lemma  below, the arguments of the proof  are analogous to Lemma \ref{prop:semismooth_h}.
 
 \begin{lema}\label{prop:semismooth_J}
Let $\Theta: \mathbf{H}_0^1(\Omega) \rightarrow \mathbb{L}^1(\Omega) $ be an operator given by  $\Theta\mathbf{u}(x)=g \beta\displaystyle \frac{ \mathcal{E}\mathbf{u}(x)}{\max(g,\beta |\mathcal{E}\mathbf{u}(x)|)}$, where $\Theta$ maps $\mathbf{u} \in \mathbf{H}_0^1(\Omega)$ to a matrix of Lebesgue functions. Then, $\Theta$ is semismooth and  $K(\mathbf{u})\mathbf{w} \in  \partial^{\circ} \Theta(\mathbf{u})\mathbf{w} $  is given by:
\begin{equation}\label{eq:Jpp1}
\begin{array}{lll}
 K(\mathbf{u})\mathbf{w}=\vspace{0.2cm} \begin{cases} \displaystyle g  \frac{  \mathcal{E}\mathbf{w}}{|\mathcal{E}\mathbf{u}|} - \displaystyle g \frac{( \mathcal{E}\mathbf{u} :\mathcal{E}\mathbf{w} ) \mathcal{E}\mathbf{u}}{|\mathcal{E}\mathbf{u}|^3},  
 & \text{ if  } \, \beta | \mathcal{E}\mathbf{u}| \geq g,  \\
 \displaystyle \beta \,  \mathcal{E}\mathbf{w},  & \text{ if } \,\, \beta | \mathcal{E}\mathbf{u}| < g.
\end{cases}
\end{array}
\end{equation}
 \end{lema}
 \begin{proof}
  Let $\Theta : \mathbf{H}_0^1 \rightarrow \mathbb{L}^1(\Omega)$ be the matrix operator given by $\Theta\mathbf{u}=(\Theta_{kl} \mathbf{u})$  for $ 1\leq k,l \leq n.$ Further, since $\mathcal{E}\mathbf{u}=(\mathcal{E}_{kl}\mathbf{u}) \in \mathbb{L}^2(\Omega) \subset \mathbb{L}^1(\Omega)$, let us reshape this matrix as the vector $(\mathcal{E}_{j}\mathbf{u})_{j=1}^{m}=(\mathcal{E}_{1} \mathbf{u}, \cdots,\mathcal{E}_{m} \mathbf{u})$ where $m=n^2.$ Then, we have that $(\mathcal{E}_{j}\mathbf{u})_{j=1}^{m} \in \prod_{j=1}^{m} L^{1}(\Omega)$. In what follows we shall construct a Nemytskii operator, $\Theta_{j}:\mathbf{H}_0^1(\Omega) \rightarrow  L^1(\Omega),$  of the form presented in \eqref{eq:Nemytskii} associated to each element $\mathcal{E}_{j}$ as follows:
\[
  \Theta_{j}=\varphi_j(H(\mathbf{u}))=g \beta\displaystyle \frac{ \mathcal{E}_j\mathbf{u}(x)}{\max(g,\beta |\mathcal{E}\mathbf{u}(x)|)}, \,\, \text{ for }\, 1\leq j \leq m.
  \]
  Here, $H:\mathbf{H}_0^1(\Omega) \rightarrow  \prod_{j=1}^{m} L^{1}(\Omega)$ maps $\mathbf{u} \in \mathbf{H}_0^1(\Omega)$ to a vector of Lebesgue functions and it is defined by $H(\mathbf{u})= (\mathcal{E}_{j}\mathbf{u})_{j=1}^{m}$. Additionally, $\varphi_{j}: \mathbb{R}^{m} \rightarrow \mathbb{R}$  is given  by $\varphi_j(y)=\frac{y_j}{\max(\sigma, \gamma |y|)}$.

  Next, we will prove that each element $\Theta_j$ is semismooth by means of \cite[Th. 3.49]{Ulbrich}. Further, this argument will also prove that the operator $\Theta$ is semismooth. 
  Analogous to the proof of Lemma \ref{prop:semismooth_h}, we need to verify that conditions \eqref{asum:1} - \eqref{asum:4} are satisfied in order to prove semismoothness of $\Theta_{j}(\mathbf{u})$ for $1 \leq j \leq m$. The operator $H(\mathbf{u})=(\mathcal{E}_{j}(\mathbf{u}))_{j=1}^{m}$ is continuously Fr\'echet differentiable in  $ \prod_{j=1}^{m} L^2(\Omega)$. The embedding $L^2(\Omega) \subset L^1(\Omega)$  implies that $H(\mathbf{u})$ is also differentiable in $ \prod_{j=1}^{m} L^1(\Omega) $. Moreover, linearity and boundedness of the $\mathcal{E}$ implies Lipschitz continuity of the mapping $H$. The proof of Lipschitz continuity and semismoothness of $\varphi$ are deferred to the Appendix. Then, conditions \ref{asum:1}) - \ref{asum:4}) are verified. Thus, $\Theta_{j}$ for $1 \leq j \leq m$ is semismooth in the sense of Definition \ref{def:semismooth}. Now, it remains to prove that the operator $\Theta : \mathbf{H}_0^1 \rightarrow \mathbb{L}^1(\Omega)$ is also semismooth. Following \cite[Prop. 3.6]{Ulbrich} let us  define the set value mapping:
  \[
   (\partial^{\circ} \Theta_1 \times \cdots \times \partial ^{\circ} \Theta_m ):  V \rightrightarrows \mathcal{L}(\mathbf{H}_0^1,\prod_{j=1}^{m} L^1),
    \]
  where $ (\partial ^{\circ} \Theta_1 \times \cdots \times \partial \Theta_m )(\mathbf{u})$ is the set of all operators $K \in \mathcal{L}(\mathbf{H}_0^1,\prod_{j=1}^{m} L^1)$ of the form 
  \[
    K: \mathbf{w} \mapsto (K_1 \mathbf{w}, \cdots, K_m \mathbf{w})=(K_j \mathbf{w})_{j=1}\,\, \text{ with }\, K_j \in \partial ^{\circ} \Theta_{j}(\mathbf{u}), \, 1\leq j\leq m.
  \]
Moreover, let us consider the nonempty set value mapping $ \partial ^{\circ} \Theta:  V \rightrightarrows \mathcal{L}(\mathbf{H}_0^1,\prod_{j=1}^{m} L^1)$  such that $\partial^{\circ} \Theta(\mathbf{u}) \subset (\partial^{\circ} \Theta_1 \times \cdots \times \partial^{\circ} \Theta_m )(\mathbf{u}) $ for all $\mathbf{u} \in \mathbf{H}_0^1$. 
  Hence, $K \in \partial^{\circ} \Theta (\mathbf{u} + \mathbf{h}) $ implies that $K \in  (\partial^{\circ} \Theta_1 \times \cdots \times \partial^{\circ} \Theta_m )(\mathbf{u} + \mathbf{h})$, i.e., we obtain $K \mathbf{w}=(K_j \mathbf{w})_{j=1}^m$. Next, the space $\prod_{j=1}^{m} L^1$ is equipped with the norm $\|y\|_{\prod_{j} L^1}=\sum_{j=1}^m \|y_j\|_{L^1}$. Therefore, by the semismoothness of each $\Theta_j$ for $1 \leq j \leq m$ we have that
  \begin{equation*}
    \small
    \begin{array}{lll}
      \displaystyle \sup_{K \in \partial^{\circ} \Theta(\mathbf{u} + \mathbf{h})} \|\Theta(\mathbf{u} + \mathbf{h})-\Theta(\mathbf{u})-K \mathbf{h}\|_{\mathbb{L}^1} &\leq& \displaystyle \sum_{j=1}^m  \sup_{K_j \in \partial^{\circ} \Theta_j(\mathbf{u} + \mathbf{h})} \|\Theta_j(\mathbf{u} + \mathbf{h})-\Theta_j(\mathbf{u})-K_j \mathbf{h}\|_{L^1}\\
      &=&o(\|\mathbf{h}\|_{\mathbf{H}_0^1}) \hspace{6mm} \text{ as } \, \|\mathbf{h}\|_{\mathbf{H}_0^1} \rightarrow 0 .
    \end{array}
  \end{equation*}
  Thus, we conclude that $\Theta$ is semismooth in the sense of Definition \ref{def:semismooth}. 

The second part of the proof consists of finding the operator $K=(K_j)_{j=1}^m \in  \partial^{\circ} \Theta (\mathbf{u})$. Following Definition \ref{def:semismooth} - equation \eqref{eq:aux007} we have that $ K_j(\mathbf{u})\mathbf{w}= N_j(\mathbf{u})^{\top}H'(\mathbf{u})(\mathbf{w})$, where $N_j( \mathbf{ u}(x))$ (a measurable selection  of Clarke's generalized Jacobian $ \partial \varphi_j(\mathcal{E} \mathbf{u}(x))$) is given by:
  \begin{equation*}
    N_j^{\top}(\mathbf{u}(x))=\begin{cases}\label{eq:aux006}
  \displaystyle g \frac{e_j^{\top}}{|\mathcal{E} \mathbf{u}(x)|} \,  - g \displaystyle   \frac{ e_j^{\top} \mathcal{E} \mathbf{u} (x)\mathcal{E}\mathbf{u}(x)^{\top}}{|\mathcal{E}\mathbf{u}(x)|^3}, & if   \,\, \beta | \mathcal{E}\mathbf{u}(x)| \geq g, \vspace{0.2cm} \\ \beta, & if \,\, \beta  |\mathcal{E}\mathbf{u}(x)| < g,
  \end{cases}
   \end{equation*}
here $e_j$ stands for the canonical unit vector. In addition, since $H'(\mathbf{u})(\mathbf{w})= \mathcal{E} \mathbf{w}$, we obtain that
\begin{equation*}
  \begin{array}{lll}
   K_j(\mathbf{u})\mathbf{w}=\vspace{0.2cm} \begin{cases} \displaystyle g  \frac{  \mathcal{E}_j\mathbf{w}}{|\mathcal{E}\mathbf{u}|} - \displaystyle g \frac{ \mathcal{E}_j\mathbf{u} (\mathcal{E}\mathbf{u}^{\top}\mathcal{E}\mathbf{w} )}{|\mathcal{E}\mathbf{u}|^3},  
   & \text{ if  } \, \beta | \mathcal{E}\mathbf{u}(x)| \geq g,  \vspace{0.2cm}  \\
   \displaystyle \beta \,  \mathcal{E}(\mathbf{w}) (x),  & \text{ if } \,\, \beta | \mathcal{E}\mathbf{u}(x)| < g.
  \end{cases}
  \end{array}
  \end{equation*}
  Finally, since $K=(K_j)_{j=1}^m $ and the inner product  of the vectorized matrices coincides with the Frobenius product of the matrices we obtain the desired result:
  \begin{equation*}
    \begin{array}{lll}
     K(\mathbf{u})\mathbf{w}=\vspace{0.2cm} \begin{cases} \displaystyle g  \frac{  \mathcal{E}\mathbf{w}}{|\mathcal{E}\mathbf{u}|} - \displaystyle g \frac{( \mathcal{E}\mathbf{u} :\mathcal{E}\mathbf{w} ) \mathcal{E}\mathbf{u}}{|\mathcal{E}\mathbf{u}|^3},  
     & \text{ if  } \, \beta | \mathcal{E}\mathbf{u}| \geq g, \vspace{0.2cm}  \\
     \displaystyle \beta \,  \mathcal{E}\mathbf{w}, & \text{ if } \,\, \beta | \mathcal{E}\mathbf{u}| < g.
    \end{cases}
    \end{array}
    \end{equation*}
\end{proof}
\begin{coro}\label{cor:lipsch}
The mapping  $\Theta(\mathbf{u})(x)=g \beta\displaystyle \frac{ \mathcal{E}(\mathbf{u})(x)}{\max(g,\beta |\mathcal{E}\mathbf{u}(x)|)}$  is Lipschitz continuous with Lipschitz constant $\bold{L}$, i.e.,
\begin{equation}
\| \Theta(\mathbf{u_1})- \Theta(\mathbf{u_2})  \|_{\mathbb{L}^2} \leq \bold{L}  \|\mathbf{u}_1-\mathbf{u}_2\|_{\mathbf{H}^1}.
\end{equation}
\end{coro}
\begin{proof}
The Lipschitz continuity of $\Theta_{j}$  for $1\leq j \leq m$ follows immediately  from its semismoothness and applying \cite[Prop. 3.36]{Ulbrich}, i.e., we have that:
\begin{equation*}
  \| \Theta_{j}(\mathbf{u_1})- \Theta_{j}(\mathbf{u_2})  \|_{L^2} \leq L_{j}\  \|\mathbf{u}_1-\mathbf{u}_2\|_{\mathbf{H}^1_0}.
  \end{equation*}
Further, since by definition of $\Theta$ we have that $\|\Theta(\mathbf{u})\|_{\mathbb{L}^2}= \displaystyle\left(\displaystyle\sum_{j=1}^m \|\Theta_{j}(\mathbf{u})\|_{L^2}^2\right)^{1/2}$ then, it follows that:
\begin{equation*}
  \| \Theta(\mathbf{u_1})- \Theta(\mathbf{u_2})  \|_{L^2} \leq \bold{L}  \|\mathbf{u}_1-\mathbf{u}_2\|_{\mathbf{H}^1_0},
  \end{equation*}
  where $\bold{L}=\displaystyle\left(\sum_{i=1}^m  L_{j}^2\right)^{1/2}$.
\end{proof}
\hspace{1mm}
\subsubsection{Generalized Second Order Derivatives}

We proceed to compute generalized second--order derivatives for $ h'_\gamma(\mathbf{u}) $ and $ J'(\mathbf{u})$. In fact, from the semismoothness of $\Phi$ and $\Theta$ we utilize $G(\mathbf{u})\mathbf{w} \in \partial^{\circ} \Phi (\mathbf{u})\mathbf{w}$ and  $K(\mathbf{u})\mathbf{w} \in \partial^{\circ} \Theta (\mathbf{u})\mathbf{w}$ to construct the generalized derivatives of $ h'_\gamma(\mathbf{u}) $ and $ J'(\mathbf{u})$. We will denote them by $\mathcal{H}(\mathbf{u})$ and $\mathcal{J}(\mathbf{u})$ respectively. For $ h'_\gamma(\mathbf{u}) $ we have that
\[
  \displaystyle \left\langle h'_\gamma(\mathbf{u}),\mathbf{v} \right\rangle_{\mathbf{H}^{-1}, \mathbf{H}^1_0} = \displaystyle  \int_{\Omega}   \Div \mathbf{v} \,\,  \Phi(\mathbf{u}) \, dx.
\]
Thus, since $G (\mathbf{u})\mathbf{w}$ is an element of  the generalized differential $\partial^{\circ} \Phi(\mathbf{u}) \mathbf{w} $, the generalized second--order derivative for $ h'_\gamma(\mathbf{u}) $, is given by 
\begin{equation}\label{eq:hpp}
\begin{array}{lll}
\left\langle \mathcal{H}(\mathbf{u})\mathbf{v},\mathbf{w} \right\rangle_{\mathbf{H}^{-1}, \mathbf{H}^1_{0}} = \displaystyle  \int_{\Omega}  \Div\mathbf{v} \,\, G(\mathbf{u})\mathbf{w} \, \, dx\vspace{0.2cm}\\\hspace{1.5cm}
=\displaystyle \sigma \int_{\Omega \setminus A_\gamma} \frac{\Div\mathbf{v} \, \Div\mathbf{w}}{|\Div\mathbf{u}|} \,\, dx\vspace{0.2cm}\\\hspace{1.5cm}- \displaystyle \sigma \int_{\Omega \setminus A_\gamma} \frac{(\Div\mathbf{u} \, \Div\mathbf{w})(\Div\mathbf{u} \, \Div\mathbf{v})}{|\Div\mathbf{u}|^3} \,\, dx 
+ \displaystyle \int_{ A_\gamma} \gamma \Div\mathbf{v} \Div\mathbf{w} \,\, dx,
\end{array}
\end{equation}
with $A_\gamma$ defined in \eqref{eq:div_active_set}. 

Regarding to $J'(\mathbf{u})$, given by \eqref{eq:Jp}, we note that it can be written as the sum of two differentiable terms and one nonsmooth given by
\[
\displaystyle g \int_{\Omega} \beta\frac{\mathcal{E}\mathbf{u} : \mathcal{E}\mathbf{v}}{\max(g, \beta|\mathcal{E}\mathbf{u}|)} \,\, dx= \displaystyle  \int_{\Omega}  \mathcal{E} \mathbf{v} : \Theta(\mathbf{u}) \,\, dx.
\]
Next, since $K (\mathbf{u}) \mathbf{w} \in \partial^{\circ} \Theta(\mathbf{u})\mathbf{w} $, we get

\begin{equation}\label{eq:Jpp}
\begin{array}{lll}
\left\langle  \mathcal{J}(\mathbf{u})\mathbf{v},\mathbf{w} \right\rangle_{\mathbf{H}^{-1}, \mathbf{H}^1_0} = \displaystyle  2\mu \int_{\Omega} \mathcal{E}\mathbf{v} : \mathcal{E}\mathbf{w} \, \,dx+\displaystyle \int_{\Omega} K(\mathbf{u})\mathbf{w} : \mathcal{E}\mathbf{v} \, \, dx \vspace{0.2cm}\\\hspace{3cm} 
=\displaystyle  2\mu \int_{\Omega} \mathcal{E}\mathbf{v} : \mathcal{E}\mathbf{w} \, \, dx+\displaystyle g \int_{\Omega \setminus E_\beta} \frac{\mathcal{E}\mathbf{v} : \mathcal{E}\mathbf{w}}{|\mathcal{E}\mathbf{u}|} \,\, dx \vspace{0.2cm}\\ \hspace{3cm} - \displaystyle g \int_{\Omega \setminus E_\beta} \frac{(\mathcal{E}\mathbf{u} : \mathcal{E}\mathbf{w})(\mathcal{E}\mathbf{u} : \mathcal{E}\mathbf{v})}{|\mathcal{E}\mathbf{u}|^3} \,\, dx 
+\beta \displaystyle \int_{ E_\beta}\, \mathcal{E}\mathbf{v}: \mathcal{E}\mathbf{w} \,\, dx.
\end{array}
\end{equation}

Taking advantage of the semismoothness properties of $J'(\mathbf{u}) + h'_{\gamma} (\mathbf{u})$, we devise a descent direction preconditioned by the second order information   $\mathcal{G}= \mathcal{J}+\mathcal{H}$. Hereafter, without risk of confusion, we omit the dependence of $\mathbf{u}$ in  $\mathcal{G}$. In addition, since $\mathcal{G}$ is a symmetric bilinear form, the descent direction, denoted in the same manner by $\mathbf{w} \in  \mathbf{H}_0^1(\Omega)$, is obtained by solving the following system:
\begin{equation*}
\left\langle \mathcal{G}\mathbf{w},\mathbf{v}\right\rangle_{\mathbf{H}^{-1}, \mathbf{H}^1_0} = -\left\langle \boldsymbol{\bar{d}},\mathbf{v}\right\rangle_{\mathbf{H}^{-1}, \mathbf{H}^1_0}, \,\, \forall \mathbf{v} \in \mathbf{H}_0^1(\Omega),
\end{equation*}
or equivalently:
\begin{equation}\label{eq:sos}
\left\langle (\mathcal{J}+\mathcal{H})\mathbf{w} ,\mathbf{v}\right\rangle_{\mathbf{H}^{-1}, \mathbf{H}^1_0} = -\left\langle \boldsymbol{\bar{d}},\mathbf{v}\right\rangle_{\mathbf{H}^{-1}, \mathbf{H}^1_0} \,  \, \forall \mathbf{v}\in \mathbf{H}_0^1.
\end{equation}
Let us recall that $\bar{\mathbf{d}}$ is given by \eqref{eq:minsubnorm}.

\begin{lema}\label{lemma:w_unique}
System \eqref{eq:sos} has a unique solution $\mathbf{w}\in \mathbf{H}_0^1(\Omega)$, which depends continuously on the descent direction $\boldsymbol{\bar{d}} \in \mathbf{H}^{-1}(\Omega)$, i.e., there is a positive constant $C$ such that
  \begin{equation}\label{eq:dep_continu}
    \|\mathbf{w}\|_{\mathbf{H}_0^1} \leq \displaystyle\frac{1}{C} \|\boldsymbol{\bar{d}}\|_{\mathbf{H}^{-1}}.
  \end{equation}
\end{lema}
\begin{proof}
  Since $\mathcal{G} \in  \mathcal{L}(\mathbf{H}^1_0(\Omega), \mathbf{H}^{-1}(\Omega))$ the bilinear form $a(\cdot, \cdot): \mathbf{H}^1_0(\Omega) \times \mathbf{H}^1_0(\Omega) \rightarrow \mathbb{R}$, defined by $a(\mathbf{w}, \mathbf{v})= \left\langle \mathcal{G}\mathbf{w},\mathbf{v}\right\rangle_{\mathbf{H}^{-1}, \mathbf{H}^1_0}$, satisfies the hypothesis of the Babu\v{s}ca-Lax-Milgram Theorem \cite[Th. 2.1]{Babusca}. Indeed, we will prove that there exist positive constants  $C$ and $\tilde{C}$ such that, for all  $ \mathbf{v}, \mathbf{w} \in \mathbf{H}^1_0(\Omega)$, the following relations are satisfied:
  \begin{enumerate}[(i)]
    \item $a(\mathbf{w}, \mathbf{w})_{\mathbf{H}_0^1} \geq C \|\mathbf{w}\|_{\mathbf{H}_0^1}^2$
    \item $|a(\mathbf{w}, \mathbf{v})_{\mathbf{H}_0^1}| \leq \tilde{C} \|\mathbf{w}\|_{\mathbf{H}_0^1} \| \mathbf{v}\|_{\mathbf{H}_0^1}$
  \end{enumerate}

 For the first part (i), let us recall that  $E_\beta:= \left\lbrace x \in \Omega : |\mathcal{E}\mathbf{u}(x)| < \frac{g}{\beta} \right\rbrace$. Coercivity of $a$ follows from taking $\mathbf{w}=\mathbf{v}$ in \eqref{eq:Jpp}, i.e.,
  \begin{equation*}
  \begin{array}{lll}
   \left\langle  \mathcal{J}\mathbf{w},\mathbf{w}  \right\rangle_{\mathbf{H}^{-1}, \mathbf{H}^1_0} = \displaystyle  2\mu \int_{\Omega} | \mathcal{E}\mathbf{w}|^2+\displaystyle g \int_{\Omega \setminus E_\beta} \frac{| \mathcal{E}\mathbf{w}|^2 }{|\mathcal{E}\mathbf{u}|} \,\, dx \vspace{0.2cm}\\ \hspace{3cm}- \displaystyle g \int_{\Omega \setminus E_\beta} \frac{(\mathcal{E}\mathbf{u} : \mathcal{E}\mathbf{w})^2}{|\mathcal{E}\mathbf{u}|^3} \,\, dx + \displaystyle \int_{ E_\beta} \beta \, | \mathcal{E}\mathbf{w}|^2 \,\, dx.
  \end{array}
  \end{equation*}
  By applying Cauchy-Schwarz to the Frobenius product in the third term of the right hand side, we have that
  \begin{equation}\label{eq:bound1}
  \begin{array}{lll}
   \left\langle  \mathcal{J}\mathbf{w},\mathbf{w}  \right\rangle_{\mathbf{H}^{-1}, \mathbf{H}^1_0} &\geq& 2\mu  \|\mathcal{E}\mathbf{w}\|^2_{\mathbb{L}^2} + \displaystyle g \int_{\Omega \setminus E_\beta} \frac{| \mathcal{E}\mathbf{w}|^2 }{|\mathcal{E}\mathbf{u}|} \,\, dx\vspace{0.2cm}\\&& - \displaystyle g \int_{\Omega \setminus E_\beta} \frac{|\mathcal{E}\mathbf{u}|^2  |\mathcal{E}\mathbf{w}|^2}{|\mathcal{E}\mathbf{u}|^3} \,\, dx
  + \beta \displaystyle  \| \mathcal{E}\mathbf{w}_k\|^2_{\mathbb{L}^2(E_\beta)}\vspace{0.2cm}  \\
  &=& \displaystyle   2\mu \|\mathcal{E}\mathbf{w}\|^2_{\mathbb{L}^2(\Omega)} + \beta \displaystyle  \| \mathcal{E}\mathbf{w}\|^2_{\mathbb{L}^2(E_\beta)}.
  \end{array}
  \end{equation}
Further, Korn's inequality (see \cite{Kikuchi} and \cite[pp. 82]{Wilbrandt}) applied to $\|\mathcal{E}\mathbf{w}\|^2_{\mathbb{L}^2(\Omega)}$ implies that
\begin{equation}\label{eq:G_J1}
\left\langle \mathcal{J}\mathbf{w},\mathbf{w}\right\rangle_{\mathbf{H}^{-1}, \mathbf{H}^1_0} \geq C \|\mathbf{w} \|_{\mathbf{H}_0^1}^2, \,\, \forall \,  \mathbf{w} \in \mathbf{H}_0^1.
\end{equation}
Similarly, taking $\mathbf{v}=\mathbf{w}$ in \eqref{eq:hpp} yields that
\begin{equation*}
\left\langle \mathcal{H}\mathbf{w},\mathbf{w} \right\rangle_{\mathbf{H}^{-1}, \mathbf{H}^1_0} = \displaystyle \sigma \int_{\Omega \setminus A_\gamma} \frac{\Div\mathbf{w}^2}{|\Div\mathbf{u}|} \,\, dx - \displaystyle \sigma \int_{\Omega \setminus A_\gamma} \frac{(\Div\mathbf{u} \, \Div\mathbf{w})^2}{|\Div\mathbf{u}|^3} \,\, dx + \displaystyle \int_{ A_\gamma} \gamma (\Div\mathbf{w})^2.  
\end{equation*}
Where we apply Cauchy-Schwarz inequality to the inner product in the second term of the right hand side, obtaining that
\begin{equation}\label{eq:bound_2}
\begin{array}{lll}
\left\langle \mathcal{H}\mathbf{w},\mathbf{w}) \right\rangle_{\mathbf{H}^{-1}, \mathbf{H}^1_0}
&\geq &\displaystyle \sigma \int_{\Omega \setminus A_\gamma} \frac{\Div\mathbf{w}^2}{|\Div\mathbf{u}|} \,\, dx - \displaystyle \sigma \int_{\Omega \setminus A_\gamma} \frac{ |\Div \mathbf{u}|^2 \, |\Div\mathbf{w}|^2}{|\Div\mathbf{u}|^3} \,\, dx\\&& + \displaystyle \int_{ A_\gamma} \gamma (\Div\mathbf{w})^2  \\
&=&  \gamma \|\Div\mathbf{w}\|^2_{L^2(A_\gamma)}.
\end{array}
\end{equation}
This inequality, together with \eqref{eq:G_J1} imply the coercivity of the bilinear form $a$. In fact, we have that
\begin{equation}\label{eq:Jpp1}
a(\mathbf{w},\mathbf{w})_{\mathbf{H}^1_0}=\left\langle  (\mathcal{J}+ \mathcal{H})\mathbf{w},\mathbf{w}\right\rangle_{\mathbf{H}^{-1}, \mathbf{H}^1_0} \geq C \|\mathbf{w} \|_{\mathbf{H}_0^1}^2, \,\, \forall \,  \mathbf{w} \in \mathbf{H}_0^1.
\end{equation}

Continuity of $a$ follows from \eqref{eq:Jpp}. Analogous to the previous arguments, we use Cauchy-Schwarz inequality resulting  in
\begin{equation}\label{eq:aux3}
\begin{array}{lll}
\left\langle \mathcal{J}\mathbf{v},\mathbf{w}   \right\rangle_{\mathbf{H}^{-1}, \mathbf{H}^1_0} &\leq &  2\mu \|\mathcal{E}(\mathbf{v})\|_{\mathbb{L}^2(\Omega)}  \|\mathcal{E}(\mathbf{w})\|_{\mathbb{L}^2(\Omega)}+ \displaystyle g \int_{\Omega \setminus E_\beta} \frac{| \mathcal{E}\mathbf{v}|| \mathcal{E}\mathbf{w}| }{|\mathcal{E}\mathbf{u}|} \,\, dx\vspace{0.2cm}\\ &&+  \displaystyle g \int_{\Omega \setminus E_\beta} \frac{| \mathcal{E}\mathbf{u}|^2 |  \mathcal{E}\mathbf{v}||  \mathcal{E}\mathbf{w}|}{|\mathcal{E}\mathbf{u}|^3} +\beta \displaystyle \int_{ E_\beta}\, |\mathcal{E}\mathbf{v}|| \mathcal{E}\mathbf{w}| \,\, dx\vspace{0.2cm} \\
&= &   2\mu \|\mathcal{E}(\mathbf{v})\|_{\mathbb{L}^2(\Omega)}  \|\mathcal{E}(\mathbf{w})\|_{\mathbb{L}^2(\Omega)}+ \displaystyle 2g \int_{\Omega \setminus E_\beta} \frac{| \mathcal{E}\mathbf{v}| | \mathcal{E}\mathbf{w}|  }{|\mathcal{E}\mathbf{u}|} \,\, dx \\
&+& \beta \displaystyle \int_{ E_\beta}\, |\mathcal{E}\mathbf{v}|| \mathcal{E}\mathbf{w}| \,\, dx
\end{array}
\end{equation}
Observe that in the set ${\Omega \setminus E_\beta}$ we have that $ \frac{1}{ | \mathcal{E}\mathbf{u}(x)| } \leq \frac{\beta}{g}$. Thus, in view of  inequality \eqref{eq:aux3}, it follows that
\begin{equation}\label{eq:Gvw}
\begin{array}{lll}
\left\langle \mathcal{J}\mathbf{v},\mathbf{w} \right\rangle_{\mathbf{H}^{-1}, \mathbf{H}^1_0} & \leq  2\mu \|\mathcal{E}(\mathbf{v})\|_{\mathbb{L}^2(\Omega)} \|\mathcal{E}(\mathbf{w})\|_{\mathbb{L}^2(\Omega)} \vspace{0.2cm}\\\hspace{3.5cm}&+\displaystyle 2 \int_{\Omega \setminus E_{\beta}}  \beta | \mathcal{E}\mathbf{v}| | \mathcal{E}\mathbf{w}|  \,\, dx  + \displaystyle \int_{ E_\beta} \beta \, | \mathcal{E}\mathbf{v}| | \mathcal{E}\mathbf{w}|\,\, dx.\\
\hspace{3.5cm} &\leq  2\mu  \|\mathcal{E}\mathbf{v}\|_{\mathbb{L}^2(\Omega)} \|\mathcal{E}\mathbf{w}\|_{\mathbb{L}^2(\Omega)}  +\displaystyle  3 \int_{\Omega}  \beta | \mathcal{E}\mathbf{v}| | \mathcal{E}\mathbf{w}| \,\, dx. \\
\end{array}
\end{equation}
Here, we apply H\"older inequality, which results in
\begin{equation}\label{eq:aux4}
  \begin{array}{lll}
  \left\langle \mathcal{J}\mathbf{v},\mathbf{w} \right\rangle_{\mathbf{H}^{-1}, \mathbf{H}^1_0} &\leq& 
  2\mu \|\mathcal{E}\mathbf{v}\|_{\mathbb{L}^2(\Omega)} \|\mathcal{E}\mathbf{w}\|_{\mathbb{L}^2(\Omega)}\vspace{0.2cm}\\ 
  &  + &\displaystyle  3 \beta \Big( \int_{\Omega}  | \mathcal{E}\mathbf{v}|^2 \Big)^{1/2} \Big( \int_{\Omega} | \mathcal{E}\mathbf{w}|^2\Big)^{1/2}  \,\, dx\vspace{0.2cm}\\
  &=&  ( 2\mu+3\beta)  \|\mathcal{E}\mathbf{v}\|_{\mathbb{L}^2(\Omega)} \|\mathcal{E}\mathbf{w}\|_{\mathbb{L}^2(\Omega)}, \vspace{0.2cm}\\
 &\leq&  C_1  \|\mathbf{v}\|_{\mathbf{H}_0^1} \|\mathbf{w}\|_{\mathbf{H}_0^1(\Omega)},\\
  \end{array}
  \end{equation}
where $C_1$ depends on $\mu$, $\beta$ and the positive constant of the continuity property of the linear operator $\mathcal{E}$. We follow a similar procedure for the term $\left\langle \mathcal{H}\mathbf{v},\mathbf{w} \right\rangle  $. By applying Cauchy-Schwarz inequality to the inner products we estimate:
\begin{equation*}
\begin{array}{lll}
\left\langle \mathcal{H}\mathbf{v},\mathbf{w} \right\rangle_{\mathbf{H}^{-1}, \mathbf{H}^1_0} =\displaystyle \sigma \int_{\Omega \setminus A_\gamma} \frac{|\Div\mathbf{v}||\Div\mathbf{w}|}{|\Div\mathbf{u}|} \,\, dx + \displaystyle \sigma \int_{\Omega \setminus A_\gamma} \frac{|\Div \mathbf{u}|^2 \, |\Div\mathbf{v}||\Div\mathbf{w}|}{|\Div\mathbf{u}|^3} \,\, dx\vspace{0.2cm}\\\hspace{3.5cm} + \gamma\displaystyle \int_{ A_\gamma} |\Div\mathbf{v}| |\Div\mathbf{w}| \vspace{0.2cm}\\\hspace{3.5cm}
\leq  \displaystyle 2 \sigma \int_{\Omega \setminus A_\gamma}  \frac{|\Div\mathbf{v}| |\Div\mathbf{w}|}{|\Div\mathbf{u}|} \,\, dx +\gamma \displaystyle \int_{ A_\gamma} |\Div\mathbf{v}||\Div\mathbf{w}| dx \vspace{0.2cm}\\\hspace{3.5cm}
\leq   \displaystyle 3 \gamma \int_{\Omega}  |\Div\mathbf{v }||\Div\mathbf{w}| \,\, dx \\
\end{array}
\end{equation*}
Once again, we apply H\"older inequality to get:
\begin{equation}\label{eq:aux5}
  \begin{array}{lll}
  \left\langle \mathcal{H}\mathbf{v},\mathbf{w} \right\rangle_{\mathbf{H}^{-1}, \mathbf{H}^1_0} &\leq& 3 \gamma   \|\Div\mathbf{v}\|_{L^2}\|\Div\mathbf{w}\|_{L^2} \\
  &\leq& C_2  \|\mathbf{v}\|_{\mathbf{H}_0^1}  \|\mathbf{w}\|_{\mathbf{H}_0^1}
\end{array}
\end{equation}
Here, $C_2$ also depends on $\gamma$ and the continuity property of the divergence operator.

Finally, from the symmetry of $a$, equations \eqref{eq:aux4}, \eqref{eq:aux5} and taking $\tilde{C}=\max\{C_1,C_2\}$ we have that
\begin{equation}\label{eq:G_csh}
  a(\mathbf{w},\mathbf{v})_{\mathbf{H}^1_0}=\left\langle \mathcal{G} \mathbf{w},\mathbf{v}\right\rangle_{\mathbf{H}^{-1}, \mathbf{H}^1_0}  \leq \tilde{C}  \|\mathbf{w}\|_{\mathbf{H}_0^1} \|\mathbf{v}\|_{\mathbf{H}_0^1}, \,\, \forall \, \mathbf{w}, \mathbf{v} \in \mathbf{H}_0^1.
\end{equation}
Then, by the Bab\v{u}ska-Lax-Milgram Theorem there exist a unique solution $\mathbf{w} \in \mathbf{H}_0^1(\Omega)$ of the system \eqref{eq:sos}. Moreover,  $\mathbf{w}$ depends continuously on the descent direction.
\end{proof}
From the previous Lemma, we can establish the following useful property of the second order term $\mathcal{G}$.
\begin{coro}\label{prop:coercivity}
 $\mathcal{G}$ satisfy the following pair of bounds:
  \begin{equation}\label{eq:coercivity}
   C \|\mathbf{w}\|_{\mathbf{H}_0^1}^2 \leq \left\langle  \mathcal{G} \mathbf{w},\mathbf{w}\right\rangle_{\mathbf{H}^{-1}, \mathbf{H}^1_0}  \leq \tilde{C}  \|\mathbf{w}\|_{\mathbf{H}_0^1}^2.
  \end{equation}
\end{coro}
\begin{proof}
  The result follows from \eqref{eq:Jpp1} and taking $ \mathbf{v}=\mathbf{w}$ in \eqref{eq:G_csh}.
\end{proof}

\section{Exact Penalization Algorithm}
Having discussed the main properties of the generalized derivatives of the objective functional,  we  introduce the following second--order algorithm for solving the \textit{exact penalization} formulation \eqref{minexact} numerically.
\begin{algorithm}[H]
\caption{Exact Penalization Algorithm - Preliminar version}\label{alg:v1}
\begin{algorithmic}[1]
\State Initialize $\mathbf{u}_0$ such that $\Div \mathbf{u}_0=0$ and set $k=0$.
\State Until \textit{stopping criteria} is true :
\State Compute  $\mathbf{d}$ by solving problem  \eqref{eq:minnorm_final}.
\State Compute descent direction $\mathbf{w}_k$ by solving  system \eqref{eq:sos}.
\State Execute line--search to get $\alpha_k$.
\State Update $\mathbf{u}_{k+1}:=\mathbf{u}_k + \alpha_k \mathbf{w}_k$, and set $k=k+1$.
\end{algorithmic}
\end{algorithm}

From Algorithm \ref{alg:v1} we can infer some useful properties of the approximated solution of problem \eqref{minexact}.

\begin{teo}\label{teo:divw0}
Let $\mathbf{u}_k$ and $\mathbf{w}_k$ be the  approximated solution and descent direction at the $k$-th iteration of Algorithm \ref{alg:v1}. If $\mathbf{u}_k$ satisfies $\Div\mathbf{u}_k=0$, then
\begin{equation}\label{eq:bound_gamma}
   \| \Div \mathbf{w}_k\|^2_{L^2(\Omega)} \leq \frac{1}{\gamma} \frac{\|\boldsymbol{\bar{d}}_k\|^2_{\mathbf{H}^{-1}}}{C}.
\end{equation}
\end{teo}
\begin{proof}
In this proof we shall use portions of the proof of Lemma \ref{lemma:w_unique}. Let us take $\mathbf{w}=\mathbf{w}_k$ in \eqref{eq:bound_2}. Since, by hypothesis,  $\Div \mathbf{u}_k =0$ at the $k$-th iteration, it follows $A^k_\gamma=\left\lbrace x \in \Omega : |\Div \mathbf{u}_k(x)| < \frac{\sigma}{\gamma} \right\rbrace=\Omega$. Hence, \eqref{eq:bound_2} yields that
\begin{equation}\label{eq:bound2}
\begin{array}{rl}
\left\langle \mathcal{H}\mathbf{w}_k,\mathbf{w}_k \right\rangle_{\mathbf{H}^{-1}, \mathbf{H}^1_0} \geq & \gamma \|\Div \mathbf{w}_k\|^2_{L^2(A_\gamma)} = \gamma \|\Div \mathbf{w}_k\|^2_{L^2(\Omega)} .
\end{array}
\end{equation}
By collecting \eqref{eq:bound2}  and \eqref{eq:G_J1},  and using \eqref{eq:sos}  with  $\mathbf{w}=\mathbf{w}_k$, we obtain that:
\begin{equation*}
\begin{array}{rl}
 \gamma \| \Div \mathbf{w}_k\|^2_{L^2(\Omega)}\leq & \left\langle  \mathcal{H}\mathbf{w}_k,\mathbf{w}_k  \right\rangle_{\mathbf{H}^{-1}, \mathbf{H}^1_0}\\
 =& - \left\langle \boldsymbol{\bar{d}}_k,\mathbf{w}_k\right\rangle_{\mathbf{H}^{-1}, \mathbf{H}^1_0} - \left\langle  \mathcal{J}\mathbf{w}_k,\mathbf{w}_k  \right\rangle_{\mathbf{H}^{-1}, \mathbf{H}^1_0} \\
 \leq & \|\boldsymbol{\bar{d}}_k\|_{\mathbf{H}^{-1}}\|\mathbf{w}_k\|_{\mathbf{H}^1_0} - C \|\mathbf{w}_k\|^2_{\mathbf{H}_0^1} .
\end{array}
\end{equation*}
In addition, from \eqref{eq:dep_continu} we deduce that:
\begin{equation*}
  \begin{array}{rl}
   \gamma \| \Div \mathbf{w}_k\|^2_{L^2(\Omega)}\leq& \displaystyle \frac{1}{C}\|\boldsymbol{\bar{d}}_k\|^2_{\mathbf{H}^{-1}} - C \|\mathbf{w}_k\|^2_{\mathbf{H}_0^1}
   \vspace{2mm}\\
   \leq& \displaystyle \frac{1}{C}\|\boldsymbol{\bar{d}}_k\|^2_{\mathbf{H}^{-1}}.
  \end{array}
  \end{equation*}
Dividing by $\gamma$ both sides we get the desired result:
\begin{equation*}
0 \leq \| \Div \mathbf{w}_k\|^2_{L^2(\Omega)} \leq \frac{1}{\gamma} \frac{\|\boldsymbol{\bar{d}}_k\|^2_{\mathbf{H}^{-1}}}{C}.
\end{equation*}
\end{proof}

\begin{coro}\label{cor:largegamma}
  If $\gamma \rightarrow \infty$ then,  $\Div \mathbf{w}_k=0$. In addition, by the updating step-6 of Algorithm \ref{alg:v1}, it follows that the next iterate, $\mathbf{u}_{k+1}$, fulfills the divergence condition, i.e., $\Div \mathbf{u}_{k+1} = \Div \mathbf{u}_{k} + \alpha_k \Div \mathbf{w}_{k}=0$.
\end{coro}
In practice, by choosig $\gamma$ large  enough we have that at each $k$-th iteration, $\Div \mathbf{w}_k \leq \epsilon \approx 0$. Inequality \eqref{eq:bound_gamma} suggest taking $\gamma > \frac{1}{C} \|\boldsymbol{\bar{d}}_k\|^2_{\mathbf{H}^{-1}} $
\begin{remark}\label{rem:d_simplified}
Theorem \ref{teo:divw0} allows us to simplify the computation of the steepest descent direction. Indeed, there is no need to solve problem \eqref{eq:minsubnorm}. Instead of using the steepest descent direction $\boldsymbol{\bar{d}}$ in Algorithm \ref{alg:v1}-step 3, Corollary \ref{cor:largegamma} justifies a simpler computation of the descent direction: using the regular part $J(u)$ only. In fact, at each iteration $k$ we are minimizing our functional in the divergence free space $V$, i.e., the approximated solution $\mathbf{u}_k$ remains in  $V$ thanks to the direction $\mathbf{w}_k$ satisfy $\Div \mathbf{w}_k \approx 0 $. Henceforth, system \eqref{eq:sos} becomes:
\begin{equation}\label{eq:sos_v2}
  \left\langle \mathcal{G}\mathbf{w}_k,\mathbf{v}_k\right\rangle_{\mathbf{H}^{-1}, \mathbf{H}^1_0} = \left\langle -J'(\mathbf{u}_k),\mathbf{v}_k\right\rangle_{\mathbf{H}^{-1}, \mathbf{H}^1_0}, \,  \, \forall \mathbf{v}_k \in \mathbf{H}_0^1(\Omega).
\end{equation}
\end{remark}
Thanks to Corollary \ref{cor:largegamma} and the last Remark we propose the following modification to the previous version of the \textit{exact penalization} Algorithm:
\begin{algorithm}[H]
\caption{Exact Penalization Algorithm}\label{alg:v2}
\begin{algorithmic}[1]
\State Initialize $\mathbf{u}_0$ such that $\Div\mathbf{u}_0=0$ and set $k=0$.
\State Until \textit{stopping criteria} is true:
\State Compute the derivative of the regular part  $  -J'(\mathbf{u}_k)$ given by \eqref{eq:Jp}.
\State Compute descent direction $\mathbf{w}_k$ by solving the system:  $\left\langle \mathcal{G}\mathbf{w}_k,\mathbf{v}_k\right\rangle_{\mathbf{H}^{-1}, \mathbf{H}^1_0} = \left\langle -J'(\mathbf{u}_k),\mathbf{v}_k\right\rangle_{\mathbf{H}^{-1}, \mathbf{H}^1_0}, \,  \, \forall \mathbf{v}_k \in \mathbf{H}_0^1(\Omega)$.
\State Execute line-search to get $\alpha_k$.
\State Update $\mathbf{u}_{k+1}:=\mathbf{u}_k + \alpha_k \mathbf{w}_k$, and set $k=k+1$.
\end{algorithmic}
\end{algorithm}

\subsection{Line--search routine}\label{rem:line_search}
  The algorithm stops when the absolute value of the dual pair of the regular part's gradient and the descent direction drops below a given tolerance, i.e., $|\langle J'(\mathbf{u}_k),\mathbf{w}_k\rangle|$ serves as a descent indicator and stopping criteria when it is approximately zero.

  Additionally, the election of the step length $\alpha$ is key to guarantee the sufficient decrease. In the \textit{Exact  Penalization} Algorithm \ref{alg:v2}- step 5 we use a line search technique which exploits polynomial models of the objective function for backtracking. This stepsize reduction approach was proposed in \cite[Sec. 6.3.2]{Dennis}. If a stepsize does not satisfy the sufficient decrease condition, the next candidate will be constrained in an interval that depends on the previous stepsize. Hence, we have: 
  \begin{equation*}
  \alpha_k \in [c_l \alpha_{prev} , c_u   \alpha_{prev}], \,\,\, \text{for } \,\, k=0, \cdots 
  \end{equation*}
  where $c_l$ and $c_u$ are positive constants  and $\alpha_{prev}$ stands for the previous step length value. In general, it is mandatory to construct stepsizes that are bounded away from zero \cite[Sec. 3.2]{Kelly}.

To finish this section, we state the following result with some convergence related quantities. This is a direct consequence of the argumentation in Remark \ref{rem:d_simplified} and Corollary \ref{prop:coercivity}.

\begin{coro}\label{cor:1}
Let $\mathbf{w}_k$ satisfy \eqref{eq:sos_v2}. Then the following  inequalities hold:
\begin{enumerate}[(i)]
	\item $C \|\mathbf{w}_k\|_{\mathbf{H}_0^1}^2 \leq - \left\langle J'(\mathbf{u}_k), \mathbf{w}_k \right\rangle_{\mathbf{H}^{-1}, \mathbf{H}^1_0} \leq  \tilde{C}  \|\mathbf{w}_k\|_{\mathbf{H}_0^1}^2,$\label{eq:cor11}
	\item $\left\langle J'(\mathbf{u}_k), \mathbf{w}_k \right\rangle_{\mathbf{H}^{-1}, \mathbf{H}^1_0} < 0,$\label{eq:cor12}
	\item $C \|\mathbf{w}_k\|_{\mathbf{H}_0^1} \leq \|J'(\mathbf{u}_k) \|_{\mathbf{H}^{-1}}.$\label{eq:cor13}
\end{enumerate}
\end{coro}
\begin{proof}
  This follows simply by taking $\mathbf{v}_k=\mathbf{w}_k$  in \eqref{eq:sos_v2} and from Corollary \ref{prop:coercivity}.
\end{proof}
\section{Convergence Analysis}
In this section we focus in the convergence analysis of the \textit{exact penalization} algorithm introduced in the previous section. Our analysis is based on the descent properties of the solution of the second--order system \eqref{eq:sos} and the sparse divergence induced by means of the second--order information. We show  convergence to the solution of the exact penalized problem, henceforth the solution of the original constrained problem.

A crucial argument in our discussion below is the Lipschitz continuity of the first order derivative, which led us to a general sufficient decrease condition. Further, the computation of the stepsizes by backtracking assures the decreasing in the cost functional $J_{\sigma}$. Finally, we prove strong convergence of the sequence  $\left\lbrace \mathbf{u}_k \right\rbrace_{k \in \mathbb{N}}$ generated by the \textit{exact penalization}-Algorithm \ref{alg:v2}.

With the help of the previous results, we infer the following lemmas relative to the regularity of functional $J$.
\begin{lema}\label{lem:Jp_is_Lipschitz}
  The first order derivative $ J'(\mathbf{u}): \mathbf{H}_0^1 \rightarrow \mathbb{R}$ is  Lipschitz continuous, i.e.,
  \begin{equation}\label{eq:aux100}
    \begin{array}{rl}
    \left\langle J'(\mathbf{u}_1) - J'(\mathbf{u}_2),\mathbf{h}\right\rangle_{\mathbf{H}^{-1}, \mathbf{H}^1_0} \leq&  L  \|\mathbf{u}_1 - \mathbf{u}_2\|_{\mathbf{H}^1_0} \|\mathbf{h}\|_{\mathbf{H}^1_0}. 
  \end{array}
  \end{equation}
\end{lema}
\begin{proof}
  Recall the definition of $\Theta$ in Lemma \ref{prop:semismooth_J}, taking into accoun that $\displaystyle  \int_{\Omega} \mathcal{E}\mathbf{u} : \mathcal{E}\mathbf{u} \, dx= \|\mathcal{E}\mathbf{u}\|^2_{\mathbb{L}^2}$, we get 
\begin{equation*}
 \displaystyle g \beta \int_{\Omega} \frac{\mathcal{E}\mathbf{u} : \mathcal{E}\mathbf{h}}{\max(g, \beta|\mathcal{E}\mathbf{u}|)} \,\, dx \leq \|\Theta (\mathbf{u})\|_{\mathbb{L}^2}\|  \mathcal{E}\mathbf{h} \|_{\mathbb{L}^2}.
\end{equation*}
Further, from the computation of $J'$ given by \eqref{eq:Jp} we infer that:
\begin{equation*}\label{eq:aux100}
  \begin{array}{lll}
  \left\langle J'(\mathbf{u}_1) - J'(\mathbf{u}_2),\mathbf{h}\right\rangle_{\mathbf{H}^{-1}, \mathbf{H}^1_0} \leq \|\mathcal{E}(\mathbf{u}_1 - \mathbf{u}_1)\|_{\mathbb{L}^2} \|\mathcal{E}(\mathbf{h})\|_{\mathbb{L}^2} \vspace{0.2cm}\\\hspace{6cm} +  \| \Theta (\mathbf{u}_1) - \Theta (\mathbf{u}_2)\|_{\mathbb{L}^2} \|\mathcal{E}(\mathbf{h})\|_{\mathbb{L}^2}.\\
\end{array}
\end{equation*}
Due to the Lipschitz continuity of $\Theta$, see Corollary \ref{cor:lipsch}, we obtain that: 
\begin{equation*}
  \begin{array}{rl}
  \left\langle J'(\mathbf{u}_1) - J'(\mathbf{u}_2),\mathbf{h}\right\rangle_{\mathbf{H}^{-1}, \mathbf{H}^1_0} \leq&  L  \|\mathbf{u}_1 - \mathbf{u}_2\|_{\mathbf{H}^1_0} \|\mathbf{h}\|_{\mathbf{H}^1_0}. 
\end{array}
\end{equation*}
Here,  $L=\tilde{c} \, \mathbf{L}$ and $\tilde{c}$ depends on the constant associated to the continuos operator $\mathcal{E}$. 
\end{proof}
\begin{lema}\label{lem:descen_lemma} 
For each $k$-th iterate and $\alpha>0$, $J$ satisfies the following inequality
\begin{equation*}\label{eq:descent_lemma}
J(\mathbf{u}_k + \alpha \mathbf{w}_k) \leq J(\mathbf{u}_k) +  \alpha \left\langle J'(\mathbf{u}_k),\mathbf{w}_k\right\rangle_{\mathbf{H}^{-1}, \mathbf{H}^1_0} + \alpha^2  \frac{L}{2} \|\mathbf{w}_k\|_{H_0^1}^2.
\end{equation*}
\end{lema}
\begin{proof}
Since the functional $J$ is Fr\'echet differentiable, the Mean Value Theorem implies that
\begin{equation*}
\begin{array}{lll}
J(\mathbf{u}_2) - J(\mathbf{u}_1)=  \displaystyle \int_0^1 \left\langle  J'(\mathbf{u}_1),\mathbf{h}\right\rangle_{\mathbf{H}^{-1}, \mathbf{H}^1_0} \,\, dt\vspace{0.2cm}\\\displaystyle \hspace{5cm} + \int_0^1 \left\langle J'(\mathbf{u}_{1} + t\mathbf{h}) - J'(\mathbf{u}_1),\mathbf{h}\right\rangle_{\mathbf{H}^{-1}, \mathbf{H}^1_0} \,\, dt. \\
\end{array}
\end{equation*}
With $\mathbf{h}=\mathbf{u}_2-\mathbf{u}_1$. Next, by applying Lemma \ref{lem:Jp_is_Lipschitz}, we get the following relation
\begin{equation}\label{eq:aux7}
\begin{array}{rl}
J(\mathbf{u}_2) - J(\mathbf{u}_1) \leq & \displaystyle \left\langle  J'(\mathbf{u}_1),\mathbf{h}\right\rangle_{\mathbf{H}^{-1}, \mathbf{H}^1_0} +  \int_0^1 L t \|\mathbf{h}\|_{\mathbf{H}_0^1}^2 \,\, dt \\
= & \displaystyle \left\langle  J'(\mathbf{u}_1),\mathbf{h}\right\rangle_{\mathbf{H}^{-1}, \mathbf{H}^1_0} +   \frac{L}{2} \|\mathbf{h}\|_{\mathbf{H}_0^1}^2.\\
\end{array}
\end{equation}
By taking  $\mathbf{u}_1=\mathbf{u}_k$ and $\mathbf{u}_2=\mathbf{u}_k+\alpha \mathbf{w}_k$ in \eqref{eq:aux7} we obtain the desired inequality.

\end{proof}
\begin{lema}\label{prop:decrease}
 Let $ \, \alpha >0$ be sufficiently small. Then, the general sufficient decrease condition is satisfied. i.e.,
  \begin{equation}\label{eq:suff_decrease}
    \begin{array}{rl}
    J(\mathbf{u}_k + \alpha \mathbf{w}_k) -\displaystyle  J(\mathbf{u}_k) < &- c_1 \alpha \|\nabla J(\mathbf{u}_k)\|^2_{{\mathbf{H}^1_0}}, \,\, \,\text{ for } c_1 >0.
    \end{array}
    \end{equation}
\end{lema}

\begin{proof}
  From the Riesz representation theorem  \cite[Remark 2.44]{Combettes}, we have that there exists a unique vector $\nabla J(\mathbf{u}_k) \in \mathbf{H}_0^1 $ such that $\left\langle J'(\mathbf{u}_k),\mathbf{v}\right\rangle_{\mathbf{H}^{-1}, \mathbf{H}^1_0} = ( \nabla J(\mathbf{u}_k),\mathbf{v})_{\mathbf{H}_0^1}$,  for all $\mathbf{v} \in \mathbf{H}_0^1 $. Thus, by taking $\mathbf{v}=\nabla J(\mathbf{u}_k) $ in  \eqref{eq:sos_v2}, we obtain that:
\begin{equation*}
\begin{array}{rl}
\|\nabla J(\mathbf{u}_k)\|^2_{\mathbf{H}_0^1} = & \left\langle  \mathcal{G}\nabla J(\mathbf{u}_k),\mathbf{w}_k\right\rangle_{\mathbf{H}^{-1}, \mathbf{H}^1_0}. \\
\end{array}
\end{equation*}
Recalling that $\left\langle  \mathcal{G}\nabla J(\mathbf{u}_k),\mathbf{w}_k\right\rangle_{\mathbf{H}^{-1}, \mathbf{H}^1_0}$ is a symmetric bilinear form, from \eqref{eq:G_csh} we get that $a(\nabla J(\mathbf{u}_k),\mathbf{w}_k)_{\mathbf{H}^1_0}=\left\langle  \mathcal{G}\nabla J(\mathbf{u}_k),\mathbf{w}_k\right\rangle_{\mathbf{H}^{-1}, \mathbf{H}^1_0}$. Thus, we have that:
\begin{equation*}
\begin{array}{rl}
\|\nabla J(\mathbf{u}_k)\|^2_{\mathbf{H}_0^1} = & a(\nabla J(\mathbf{u}_k),\mathbf{w}_k )_{\mathbf{H}^1_0}\\
\leq & \tilde{C} \|\mathbf{w}_k\|_{\mathbf{H}_0^1} \|\nabla J(\mathbf{u}_k)\|_{\mathbf{H}_0^1}.
\end{array}
\end{equation*}
Clearly, if $\|\nabla J(\mathbf{u}_k)\|_{\mathbf{H}_0^1} \neq 0$, then Corollary \ref{cor:1}-\eqref{eq:cor11} implies that

\begin{equation}\label{eq:aux10.a}
\begin{array}{rl}
\|\nabla J(\mathbf{u}_k)\|^2_{\mathbf{H}_0^1} \leq & \tilde{C}^2 \| \mathbf{w}_k\|^2_{\mathbf{H}^{1}_0}\\
\leq & -\displaystyle \frac{\tilde{C}^2}{C} \left\langle J'(\mathbf{u}_k), \mathbf{w}_k \right\rangle_{\mathbf{H}^{-1}, \mathbf{H}^1_0}.
\end{array}
\end{equation}
In addition, by Lemma \ref{lem:descen_lemma} and \eqref{eq:aux10.a}, the following inequality holds:
\begin{equation*}
\begin{array}{rl}  
J(\mathbf{u}_k + \alpha \mathbf{w}_k) \leq & \displaystyle J(\mathbf{u}_k) -\alpha \frac{C}{\tilde{C}^2} \|\nabla J(\mathbf{u}_k)\|^2_{\mathbf{H}_0^1} + \alpha^2  \frac{L}{2} \|\mathbf{w}_k\|_{\mathbf{H}_0^1}^2. \\
\end{array}
\end{equation*}
Next, from Corollary \ref{cor:1}-\eqref{eq:cor13} we have that
\begin{equation*}
\begin{array}{rl}
J(\mathbf{u}_k + \alpha \mathbf{w}_k) \leq & \displaystyle  J(\mathbf{u}_k) -\alpha \frac{C}{\tilde{C}^2} \|\nabla J(\mathbf{u}_k)\|^2_{{\mathbf{H}^1_0}} + \alpha^2  \frac{L}{2C} \|\nabla J(\mathbf{u}_k)\|^2_{{\mathbf{H}^1_0}}\vspace{0.2cm}\\
  = &\displaystyle  J(\mathbf{u}_k) -\alpha \left(\frac{C}{\tilde{C}^2} -  \frac{\alpha L}{2C}\right) \|\nabla J(\mathbf{u}_k)\|^2_{{\mathbf{H}^1_0}}.
\end{array}
\end{equation*}
Thus, by choosing
\begin{equation}\label{eq:alpha_bound}
  0< \displaystyle  \alpha  \leq \frac{2C^2}{L\tilde{C}^2},
\end{equation}
we get $\big(\frac{C}{\tilde{C}^2} -  \frac{\alpha L}{2C}\big) >0$. Then, the result follows by taking $c_1 \leq \big(\frac{C}{\tilde{C}^2} -  \frac{\alpha L}{2C}\big)$.
\end{proof}

\begin{coro}\label{lema:f_prima}
The Exact Penalization Algorithm \ref{alg:v2}, with stepsize $\alpha_k $ satisfying \eqref{eq:suff_decrease}, generates a sequence $\{\nabla J(\mathbf{u}_k)\}_{k \in \mathbb{N}}$, such that $\displaystyle \lim_{k \rightarrow \infty}  \nabla J(\mathbf{u}_k)= 0$.
\end{coro}
\begin{proof}
In view of subsection \ref{rem:line_search} and \eqref{eq:alpha_bound} we have that $\alpha_{prev} >  \frac{2C^2}{L\tilde{C}^2}$. Then, the stepsizes $\alpha_k$ generated by the Exact Penalization Algorithm \ref{alg:v2} satisfy 
\begin{equation}\label{eq:alpha_bounded_below}
  \alpha_k >  \bar{\alpha}= \frac{l 2C^2}{L\tilde{C}^2}, \,\,\, \text{for } \,\, k=0, \cdots 
\end{equation}
Next, \eqref{eq:suff_decrease} implies that  $J(\mathbf{u}_k)$ is a decreasing sequence and, as pointed out in the definition of $J$, it is bounded from below, then
\begin{equation}\label{eq:J_convergente}
  \displaystyle \lim_{k \rightarrow \infty} J(\mathbf{u}_k)=J^*.
\end{equation}
Collecting  \eqref{eq:alpha_bounded_below}, \eqref{eq:J_convergente} and by taking the limit $k \rightarrow \infty$ on both sides of inequality \eqref {eq:suff_decrease}, we infer that:
\begin{equation*}
\begin{array}{lll}
 \displaystyle \lim_{k \rightarrow \infty} c_1 \bar{\alpha} \|\nabla J(\mathbf{u}_k)\|^2_{{\mathbf{H}^1_0}} \leq &   \displaystyle \lim_{k \rightarrow \infty} \left[\displaystyle  J(\mathbf{u}_k) - J(\mathbf{u}_{k+1}) \right]=& 0.
\end{array}
\end{equation*}
Thus, $ \displaystyle \lim_{k \rightarrow \infty} \|\nabla J(\mathbf{u}_k)\|_{\mathbf{H}_0^1}=0$. 
\end{proof}

\begin{teo}
The sequence $\left\lbrace \mathbf{u}_k \right\rbrace_{k \in \mathbb{N}}$ generated by Algorithm \ref{alg:v2} converges to the minimizer $\bar{\mathbf{u}}$ of $J_{\sigma}$.
\end{teo}
\begin{proof}
Let $\left\lbrace \mathbf{u}_k \right\rbrace_{k \in \mathbb{N}}$ be the sequence generated by Algorithm \eqref{alg:v2}. By Corollary \ref{cor:largegamma} it fulfills $\Div \mathbf{u}_k =0$ for all $k \in \mathbb{N}$. Therefore, $J_\sigma(\mathbf{u}_k)= J(\mathbf{u}_k)$. 

Moreover, let us recall that $J$ is coercive and  lower semicontinuous. Then, coercivity implies the uniformly boundedness of the sequence $\{ \mathbf{u}_k\}_{k \in \mathbb{N}}$ in $\mathbf{H}_0^1(\Omega)$. Thus, by reflexivity of $\mathbf{H}_0^1(\Omega)$ there exists a convergent subsequence $\left\lbrace \mathbf{u}_{k_j} \right\rbrace_{j \in \mathbb{N}}$ such that $\mathbf{u}_{k_j} \rightharpoonup \mathbf{u}^*$. Following \cite[Prop. 6.15]{Peypo} and thanks to the convexity of $J$, we obtain that
\begin{equation}\label{eq:weak_limit}
J(\mathbf{v}) \geq   J(\mathbf{u}_{k_j}) + ( \nabla J(\mathbf{u}_{k_j}), \mathbf{v}-\mathbf{u}_{k_j})_{\mathbf{H}_0^1},\,\,\forall\mathbf{v} \in \mathbf{H}_0^1(\Omega).
\end{equation}
However, from the weak lower semicontinuity of $J$ we have that $\displaystyle \liminf_{j \rightarrow \infty} J(\mathbf{u}_{k_j}) \geq J(\mathbf{u}^*)$. Thus, from Corollary \ref{lema:f_prima} and \eqref{eq:weak_limit} we deduce that $J(\mathbf{u}^*) \leq J(\mathbf{v})$ for all $\mathbf{v} \in \mathbf{H}_0^1$, i.e.,  every weak limit point of the subsequence $\left\lbrace \mathbf{u}_{k_j} \right\rbrace_{j \in \mathbb{N}}$ minimizes $J$. Further, since $ \mathbf{u}_{k_j} \in V $ and $V$ is weakly closed, then $\Div \mathbf{u}^*=0$. From the uniqueness of the minimizer of $J$ pointed out in Section \ref{sec:expen}, we conclude that $\mathbf{u}^*=\bar{\mathbf{u}}$. 

Strong convergence of the sequence $\lbrace \mathbf{u}_{k_j} \rbrace_{j \in \mathbb{N}}$ in $\mathbf{H}_0^1(\Omega)$ can be proved using the properties of the  cost functional. Let us rewrite  $J(\mathbf{u}_{k_j})= \tilde{J}(\mathbf{u}_{k_j}) + \tilde{G}(\mathbf{u}_{k_j})$, where 
\[
\displaystyle \tilde{J}(\mathbf{u}_{k_j})=\mu\int_\Omega \mathcal{E}\mathbf{u}_ {k_j}:\mathcal{E}\mathbf{u}_{k_j} \,dx  - \int_\Omega \mathbf{f}\cdot \mathbf{u}_{k_j}\,dx\,\,\mbox{ and }\,\,\tilde{G}(\mathbf{u}_{k_j})=\displaystyle \int_\Omega \Psi(\mathcal{E}\mathbf{u}_{k_j}) \, dx
\]
stand for the quadratic part and the regularized term, respectively.  Since $ \tilde{J}(\mathbf{u}_{k_j})$ is quadratic we obtain that
 \begin{equation}\label{eq:aux13}
 \tilde{J}(\mathbf{u}_{k_j})= \tilde{J}(\bar{\mathbf{u}})+ \left\langle \tilde{J}'(\bar{\mathbf{u}}),(\mathbf{u}_k -\bar{\mathbf{u}})\right\rangle_{\mathbf{H}^{-1}, \mathbf{H}^1_0} +\frac{1}{2} \left\langle \tilde{J}''(\bar{\mathbf{u}})(\mathbf{u}_ {k_j}-\bar{\mathbf{u}}),(\mathbf{u}_{k_j} -\bar{\mathbf{u}})\right\rangle_{\mathbf{H}^{-1}, \mathbf{H}^1_0},
 \end{equation}
 where the second derivative $\tilde{J}''(\bar{\mathbf{u}})$  is given by 
\begin{equation*}
  \begin{array}{rll}
    \left\langle \tilde{J}''(\bar{\mathbf{u}})(\mathbf{u}_k- \bar{\mathbf{u}}),(\mathbf{u}_{k_j}- \bar{\mathbf{u}})\right\rangle_{\mathbf{H}^{-1}, \mathbf{H}^1_0} &= \displaystyle  \int_{\Omega} | \mathcal{E}(\mathbf{u}_{k_j}- \bar{\mathbf{u}})|^2 \, dx &= \|\mathcal{E}(\mathbf{u}_{k_j}- \bar{\mathbf{u}})\|_{\mathbb{L}^2}^2.
  \end{array}
\end{equation*}

By applying Korn's inequality to $\| \mathcal{E}(\mathbf{u}_k- \bar{\mathbf{u}})\|^2_{\mathbb{L}^2}$, we obtain that there exists a constant $C>0$, such that $\left\langle \tilde{J}''(\bar{\mathbf{u}})(\mathbf{u}_{k_j}- \bar{\mathbf{u}}),(\mathbf{u}_{k_j}- \bar{\mathbf{u}})\right\rangle_{\mathbf{H}^{-1}, \mathbf{H}^1_0} \geq C \|\mathbf{u}_{k_j}- \bar{\mathbf{u}}\|_{\mathbf{H}_0^1}^2$. Then, from \eqref{eq:aux13} we have that
\begin{equation}\label{eq:aux15.a}
\begin{array}{lll}
\frac{C}{2}\|\mathbf{u}_{k_j}- \bar{\mathbf{u}}\|_{\mathbf{H}_0^1}^2 & \leq  \tilde{J}(\mathbf{u}_{k_j})-\tilde{J}(\bar{\mathbf{u}}) - \left\langle \tilde{J}'(\bar{\mathbf{u}}),(\mathbf{u}_{k_j}-\bar{\mathbf{u}})\right\rangle_{\mathbf{H}^{-1}, \mathbf{H}^1_0} \vspace{0.2cm}\\
\hspace{1cm} & =  \tilde{J}(\mathbf{u}_{k_j}) + \tilde{G}(\mathbf{u}_{k_j}) -\tilde{J}(\bar{\mathbf{u}}) - \tilde{G}(\mathbf{u}_{k_j}) \vspace{0.2cm}\\ &- \left\langle \tilde{J}'(\bar{\mathbf{u}}),(\mathbf{u}_{k_j}-\bar{\mathbf{u}})\right\rangle_{\mathbf{H}^{-1}, \mathbf{H}^1_0}.
\end{array}
\end{equation}
Taking the $\limsup$ on both sides of inequality \eqref{eq:aux15.a}, we get that
\begin{equation}\label{eq:strong_conv}
  \begin{array}{lll}
  \displaystyle\limsup_{j \rightarrow \infty} \frac{C}{2}\|\mathbf{u}_{k_j}- \bar{\mathbf{u}}\|_{\mathbf{H}_0^1}^2 &  \leq  \displaystyle\limsup_{j \rightarrow \infty} (\tilde{J}(\mathbf{u}_{k_j}) + \tilde{G}(\mathbf{u}_{k_j})) \vspace{0.2cm}\\  &+\displaystyle\limsup_{j \rightarrow \infty}(-\tilde{J}(\bar{\mathbf{u}}) - \tilde{G}(\mathbf{u}_{k_j})) \vspace{0.2cm}\\
  &+ \displaystyle\limsup_{j \rightarrow \infty} \left(- \left\langle \tilde{J}'(\bar{\mathbf{u}}),(\mathbf{u}_{k_j}-\bar{\mathbf{u}})\right\rangle_{\mathbf{H}^{-1}, \mathbf{H}^1_0}\right).
  \end{array}
  \end{equation}
Then from \eqref{eq:J_convergente} and \eqref{eq:strong_conv} we have:
\begin{equation}\label{eq:strong_conv2}
  \begin{array}{lll}
  \displaystyle\limsup_{j \rightarrow \infty} \frac{C}{2}\|\mathbf{u}_{k_j}- \bar{\mathbf{u}}\|_{\mathbf{H}_0^1}^2 &  \leq  J^* -\tilde{J}(\bar{\mathbf{u}}) + \displaystyle\limsup_{j \rightarrow \infty}( - \tilde{G}(\mathbf{u}_{k_j}) \vspace{0.2cm}\\&+ \displaystyle\limsup_{j \rightarrow \infty} - \left\langle \tilde{J}'(\bar{\mathbf{u}}),(\mathbf{u}_{k_j}-\bar{\mathbf{u}})\right\rangle_{\mathbf{H}^{-1}, \mathbf{H}^1_0}.\\
  \vspace{0.2cm}\\
  & =   J^* -\tilde{J}(\bar{\mathbf{u}}) - \displaystyle\liminf_{j \rightarrow \infty}\tilde{G}(\mathbf{u}_{k_j}) \vspace{0.2cm}\\&+ \displaystyle\limsup_{j \rightarrow \infty} \left(- \left\langle \tilde{J}'(\bar{\mathbf{u}}),(\mathbf{u}_{k_j}-\bar{\mathbf{u}})\right\rangle_{\mathbf{H}^{-1}, \mathbf{H}^1_0}\right).\\
  \end{array}
  \end{equation}
Next, since $\mathbf{u}_k \rightharpoonup \bar{\mathbf{u}}$, we  conclude that $ \displaystyle \left\langle \tilde{J}'(\bar{\mathbf{u}}),(\mathbf{u}_{k_j}-\bar{\mathbf{u}})\right\rangle_{\mathbf{H}^{-1}, \mathbf{H}^1_0} \displaystyle \rightarrow 0$ . Also, from the weakly lower semicontinuity of $\tilde{G}$ we have that $ -\displaystyle\liminf_{j \rightarrow \infty}\tilde{G}(\mathbf{u}_{k_j}) \leq - \tilde{G}(\bar{\mathbf{u}}) $. Collecting these arguments, from \eqref{eq:strong_conv2} we obtain
\begin{equation}\label{eq:strong_conv3}
  \begin{array}{lll}
  \displaystyle\limsup_{j \rightarrow \infty} \frac{C}{2}\|\mathbf{u}_{k_j}- \bar{\mathbf{u}}\|_{\mathbf{H}_0^1}^2 & \leq   J^* -\tilde{J}(\bar{\mathbf{u}}) - \tilde{G}(\bar{\mathbf{u}})  = 0.\\
  \end{array}
  \end{equation}
The last equality follows since $J^*=\inf J(\mathbf{u}_{k_j})= J(\bar{\mathbf{u}})$. Thus, we conclude that $\mathbf{u}_k \rightarrow \bar{\mathbf{u}}$ strongly in $\mathbf{H}^1_0(\Omega)$.
\end{proof}
\section{Numerical Experiments}
Our final section is devoted to the numerical experimentation of Algorithm \ref{alg:v2}. In the first subsection, we conduct a set of experiments linked to the algorithm's performance, including an exhaustive testing of the parameters governing the associated regularizations and, in particular, to the exact penalization parameter $\sigma$. The second set of experiments, in Section \ref{ss:comparison}, aims to compare our algorithm with the Semismooth Newton method, which is well known by its superlinear convergence properties. A third set of experiments in three dimensions are also presented to further illustrate the applicability and scalability of the exact penalization method which will be referred as EP algorithm.
\subsubsection*{Implementation details}
 We consider $\Omega \in \mathbb{R}^{n}$, $n=2,3$ be a polygonal/polyhedral domain and $\mathcal{T}_h$ a regular discretization (by triangles or tetrahedrons) on $\Omega$. The Galerkin Finite Element Method was used to approximate the desired velocity $\bar{\mathbf{u}} \in \mathbf{H}_0^1$ - of the Bingham problem (8) - by continuous piecewise quadratic vector-valued Lagrange trial functions over each element. Therefore, we consider the finite-element space $V_h=\{\mathbf{u}_h \in C(\bar{\Omega})^2 | \hspace{1mm} \mathbf{u}_h |_{\Gamma}=0 \,\, \text{and}\,\, \mathbf{u}_h|_T \in P^2, \forall \, T \in \mathcal{T}_h\}$. Our experiments were implemented in the open-source software FEniCS (https://fenicsproject.org/). 
\subsection{Algorithm's performance}\label{rotational_flow}

 In order to measure  the performance of the EP algorithm, we consider the benchmark of a rotational Bingham flow  in a square reservoir. In this case, problem \eqref{minexact} is solved with a  driven  force $f(x_1,x_2)=  300(x_2-0.5,0.5-x_1)$ over the domain $\Omega=[0,1]^2$ with $g=10 \sqrt{2}$. In the following experiments the mesh size is $h=1/50$. Figure 1 shows the computed rotational flow expected from applying the force $f$, whereas in Figure \ref{Fig:2.2} we can visualize the yielded (dark blue) and unyielded regions (red tones). This configuration presents a central solid region. 
 \begin{figure}[H]
  \centering
  \begin{subfigure}{.45\textwidth}
    \includegraphics[width=8cm]{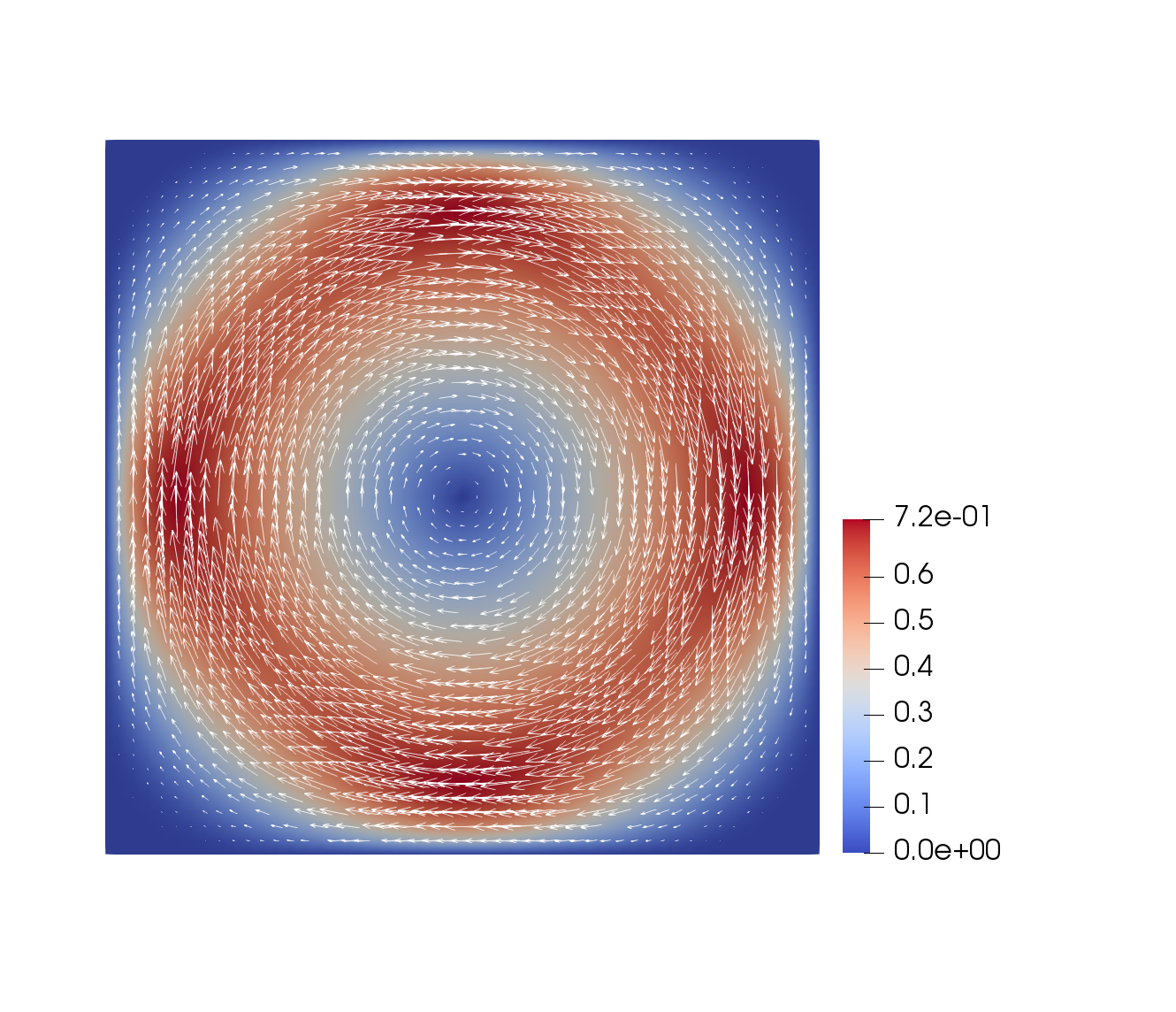}
    \caption{Velocity field $\mathbf{u}$}
    \label{Fig:2.1}
  \end{subfigure}
  \hfill
  \begin{subfigure}{.45\textwidth}
    \includegraphics[width=8cm]{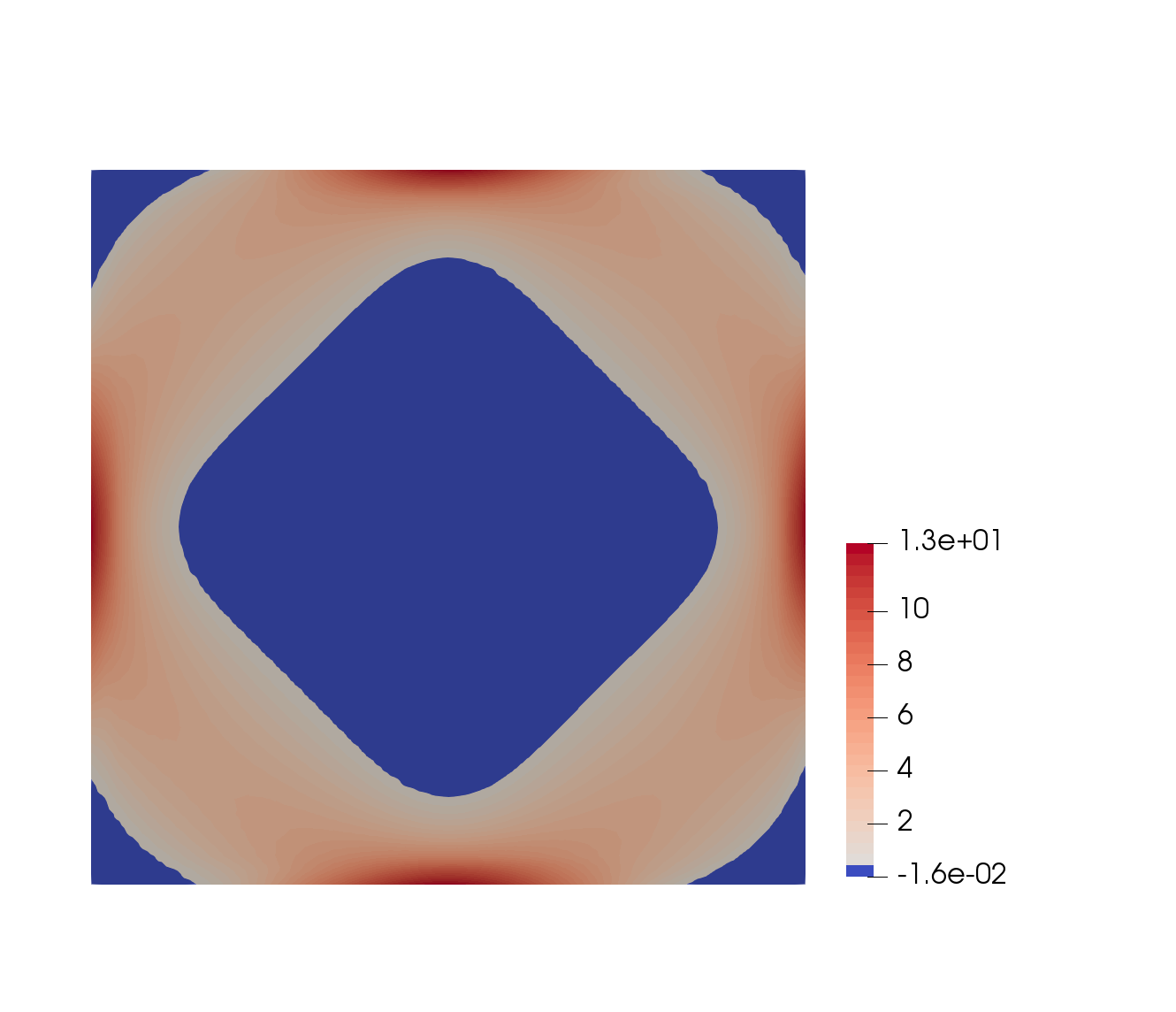}
    \caption{Plug zones given by $|\mathcal{E}(\mathbf{u})|_{\mathbb{L}^2}$}
       \label{Fig:2.2}
  \end{subfigure}
  \label{Fig:1}
  \caption{Bingham flow in the square reservoir.}
\end{figure}

 Recall that Algorithm \ref{alg:v2} has three important parameters, namely:
 \begin{itemize}
 \item  $\sigma$: exact penalization parameter, satisfying $\sigma \geq \sigma_0$, see Theorem \ref{teo:exact2}
 	\item  $\gamma$: enriching second--order information parameter 
 	\item  $\beta$ : Huber regularization parameter 
 \end{itemize}
In the next experiments, we will discuss the influence of these parameters on the numerical realization of the method.

\subsubsection{Experiment 1: exact penalization test varying $\sigma$} In previous sections, we have discussed that the equivalence between the constrained problem \eqref{eq:prob_const} and the penalized problem  \eqref{minexact} is given for all $\sigma \geq \sigma_0$, with $\sigma_0 \approx \|{\lambda}\|_{L^2}|\Omega|^{\frac12}$. In our first experiment, we look for the numerical behavior of the divergence term $\| \Div \mathbf{u} \|_{L^1}$ of the approximated solution in each iterate of the EP Algorithm, by choosing different values of $\sigma$. 

Despite the computation of $ \|{\lambda}\|_{L^2}$ can not be a--priorily done, i.e. without an approximate solution at hand, we can confirm numerically that our estimation of  $\sigma_0$ in Remark \ref{re:bound_sigma0} is sharp. Indeed, we have estimated the  upper bound: $\sigma_0 \approx \|\lambda\|_{L^2}|\Omega|^{-1/2} \leq 896.5$. This can be observed in Figures \ref{Fig:1} - \ref{Fig:3}, where the divergence is depicted with colored - solid lines for values of $\sigma \geq  900$. We can observe that the  history of the divergence values  remain between $1.0e-5$ and $1.0e-7$, recognisably lower than those for smaller values of $\sigma_0 < 900$, in dashed lines. This, also illustrates the equivalence of the exact penalization with the constrained formulation of the problem. 

Table \ref{table:1.1} summarises numerical results. For the variation of $\sigma\geq \sigma_0$, we can observe that slightly better results are obtained when it is close to $\sigma_0$ rather than when $\sigma$ is too big. Roughly speaking, once $\sigma$ is suitably chosen, its variation does not have a hard influence in the numerical performance of the Algorithm.

  \begin{figure}[H]
    \centering
      \includegraphics[width=10cm]{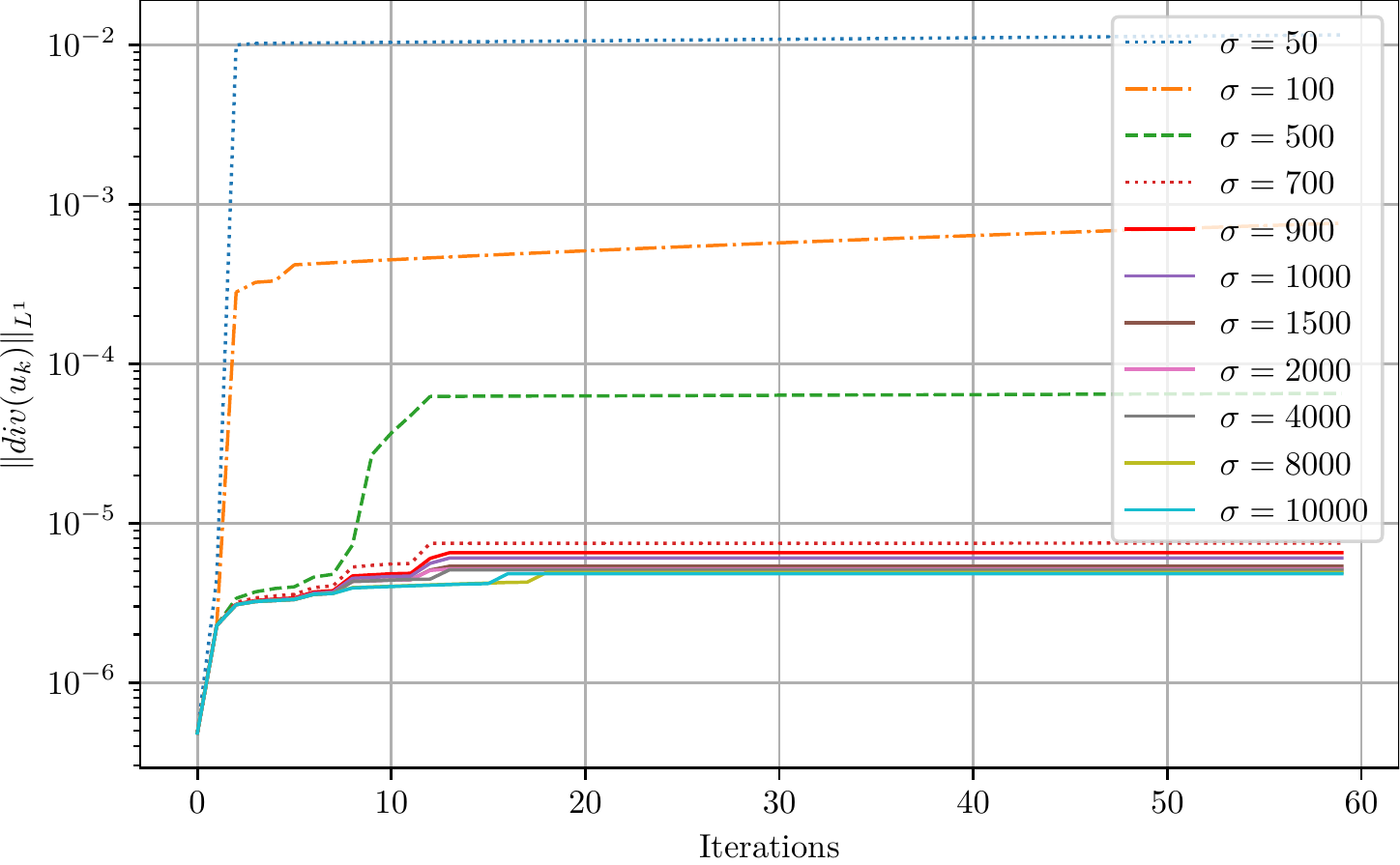}
      \captionof{figure}{\small{Experiment 1: history of the divergence for $\gamma=1e+7$. Solid lines for $\sigma \geq 900$}}
      \label{Fig:1}
    \end{figure}
        \begin{figure}[H]
          \centering
          \includegraphics[width=10cm]{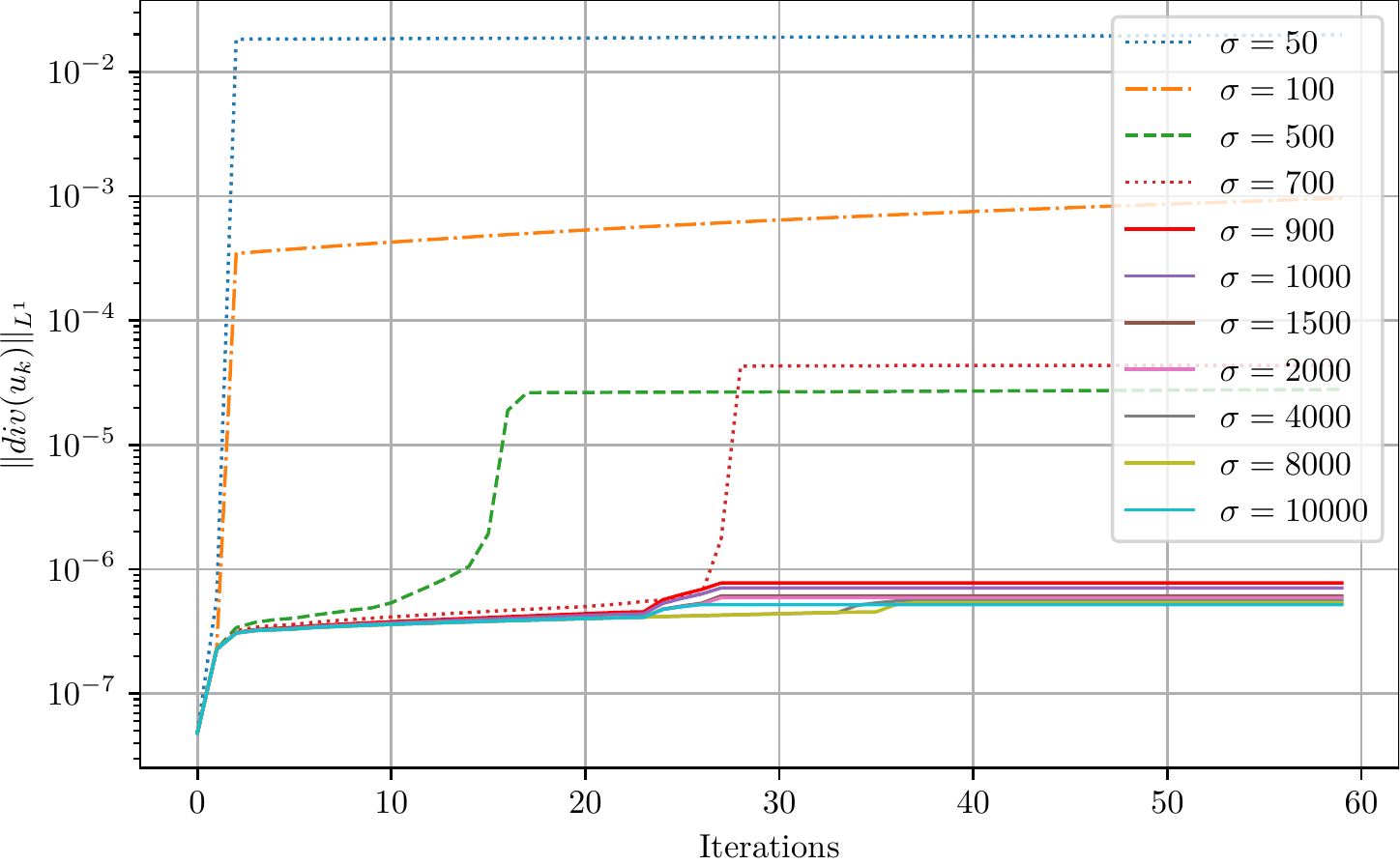}
        \captionof{figure}{\small{Experiment 1: history of the divergence for $\gamma=1e+8$. Solid lines for $\sigma \geq 900$}}
        \label{Fig:2}
        \end{figure}
    \begin{figure}[H]
          \centering
          \includegraphics[width=10cm]{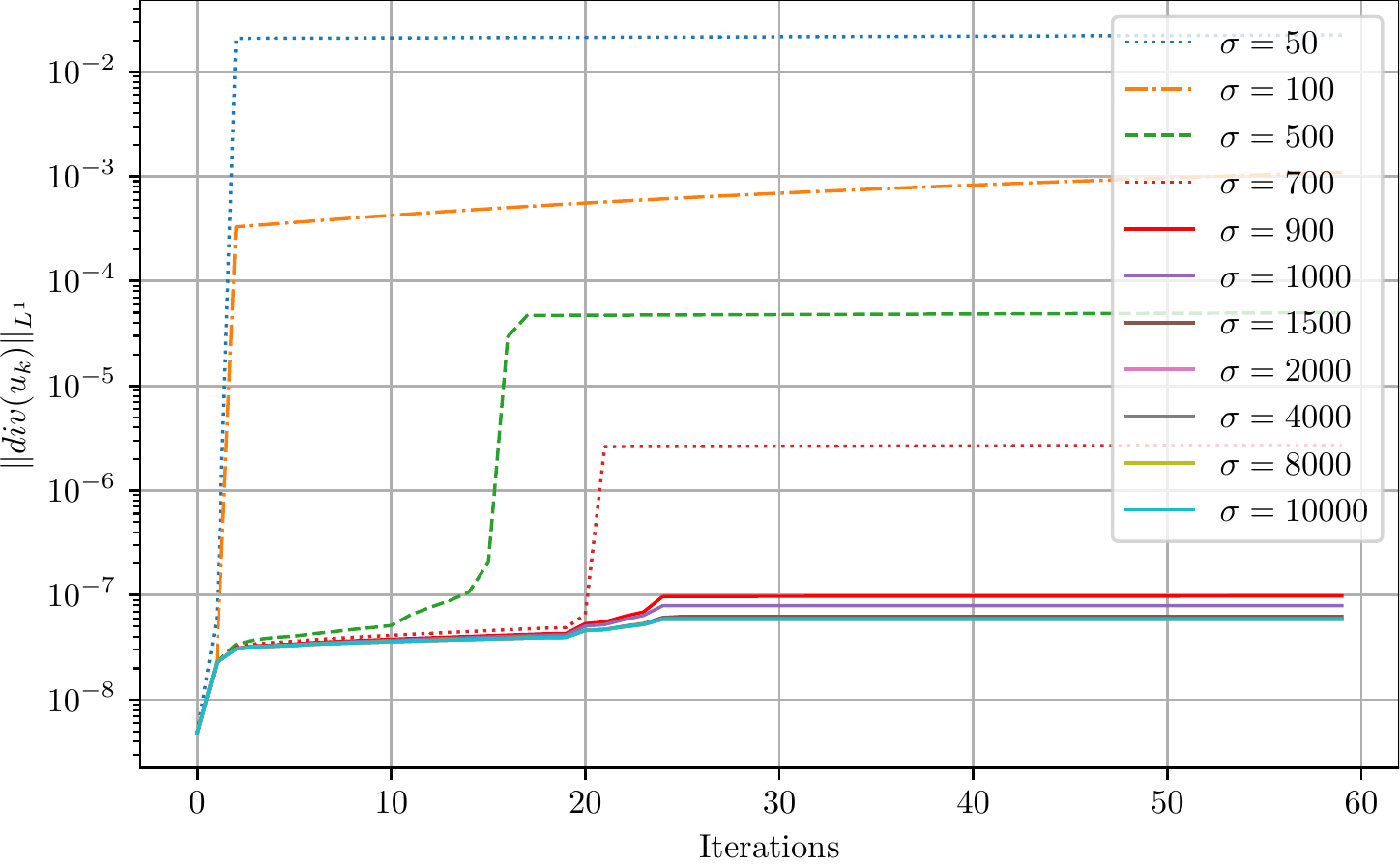}
        \captionof{figure}{\small{Experiment 1: history of the divergence for $\gamma=1e+9$. Solid lines for $\sigma \geq 900$}}
        \label{Fig:3}
    \end{figure}


\subsubsection{Experiment 2: enriched second--order information test} 
In this experiment, we solve the same problem formulated in Experiment 1 by focussing on the  parameter $\gamma$, which is associated to the generalized second--order information obtained from  term  $\|\Div \mathbf{u}_k\|_{L^1}$. From our theory, we know that in order to get a good approximation for the second--order information $\gamma$ must be sufficiently large. We may guess such $\gamma$ by using the relation obtained from Corolary \ref{cor:largegamma}:  $\gamma > \frac{1}{C} \|\boldsymbol{\bar{d}}_k\|^2_{\mathbf{H}^{-1}}$, which assures a small divergence.

As shown in Table \ref{table:1.1}, this parameter is crucial for our algorithm; in fact, if we neglect  the additional second--order information (i.e. $\gamma=0$) the algorithm fails to converge.
In contrast,  by setting icreasing values of  $\gamma=1e+7, 1e+8$ and $1e+9$ we see that $\|\Div \mathbf{u}_k\|_{L^1}$ gets smaller. In fact, $\gamma=1e+9$ achieves a divergence norm of order $1e-8$ and the algorithm seems to have a more stable cost-functional. Figures \ref{Fig:f1} - \ref{Fig:f3} depict this behavior, for $\gamma=1e+9$ and for $\sigma \geq 900$, the cost function converges to the same value. However, slightly more computing time is spent for this value of $\gamma$. Also, we observe in the fifth column that the stopping criteria $|\langle J'(\mathbf{u}_k),\mathbf{w}_k\rangle| $ (subsection \ref{rem:line_search}) is closer to zero ($1e-05$) (see Figure \ref{fig:stopping}). 
From these numerical results, we conclude that the additional second--order information, enriching the descent direction system  \eqref{eq:sos}, is essential for the exact penalization method.
\begin{figure}[H]
  \centering
  \includegraphics[width=10cm]{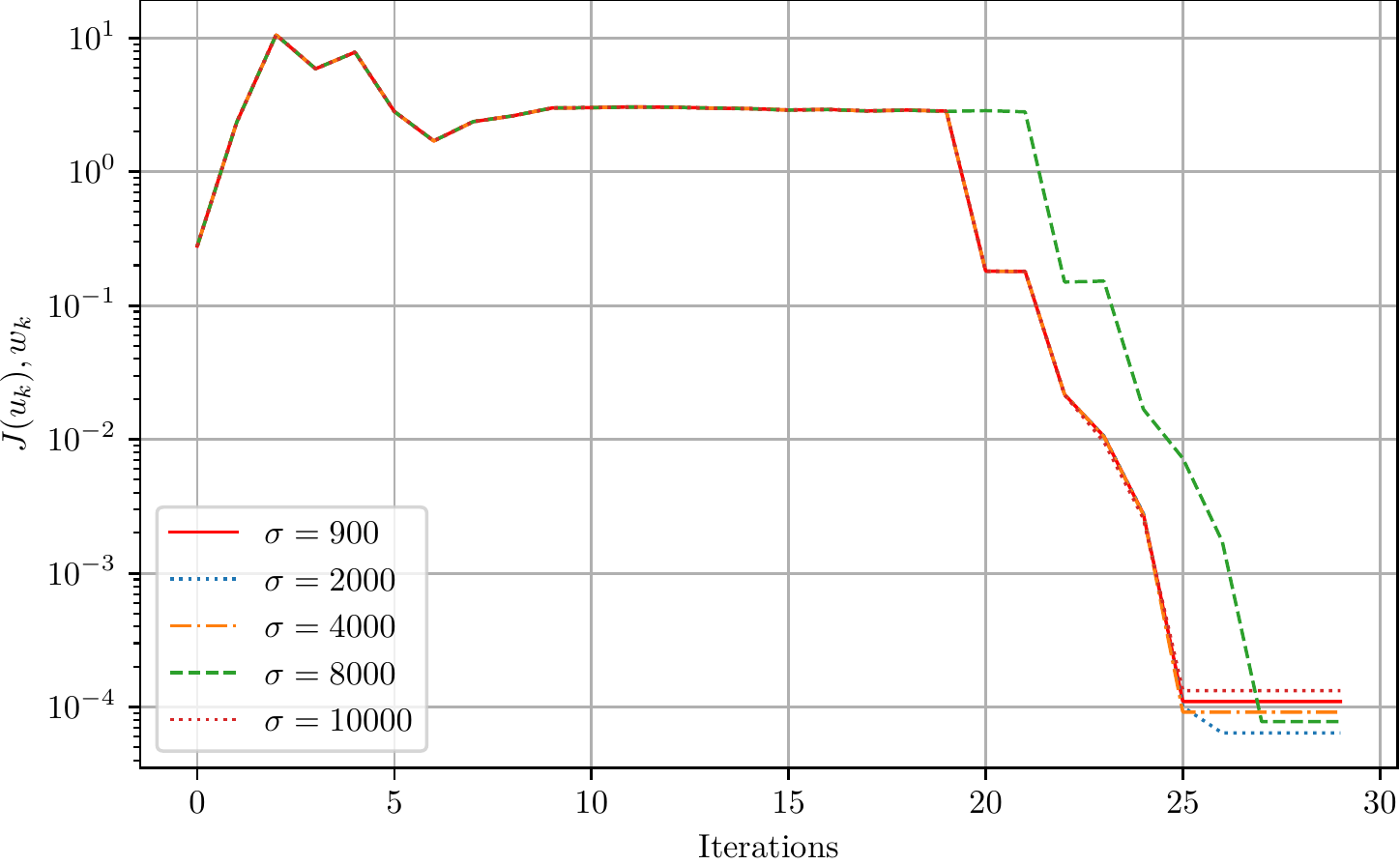}
  \captionof{figure}{Experiment 2: history of stopping criteria of EP-Algorithm with $ \gamma=1e+9$ and $\beta=1e+03$.}
  \label{fig:stopping}
\end{figure}
\begin{figure}[H]
  \centering
    \includegraphics[width=10cm]{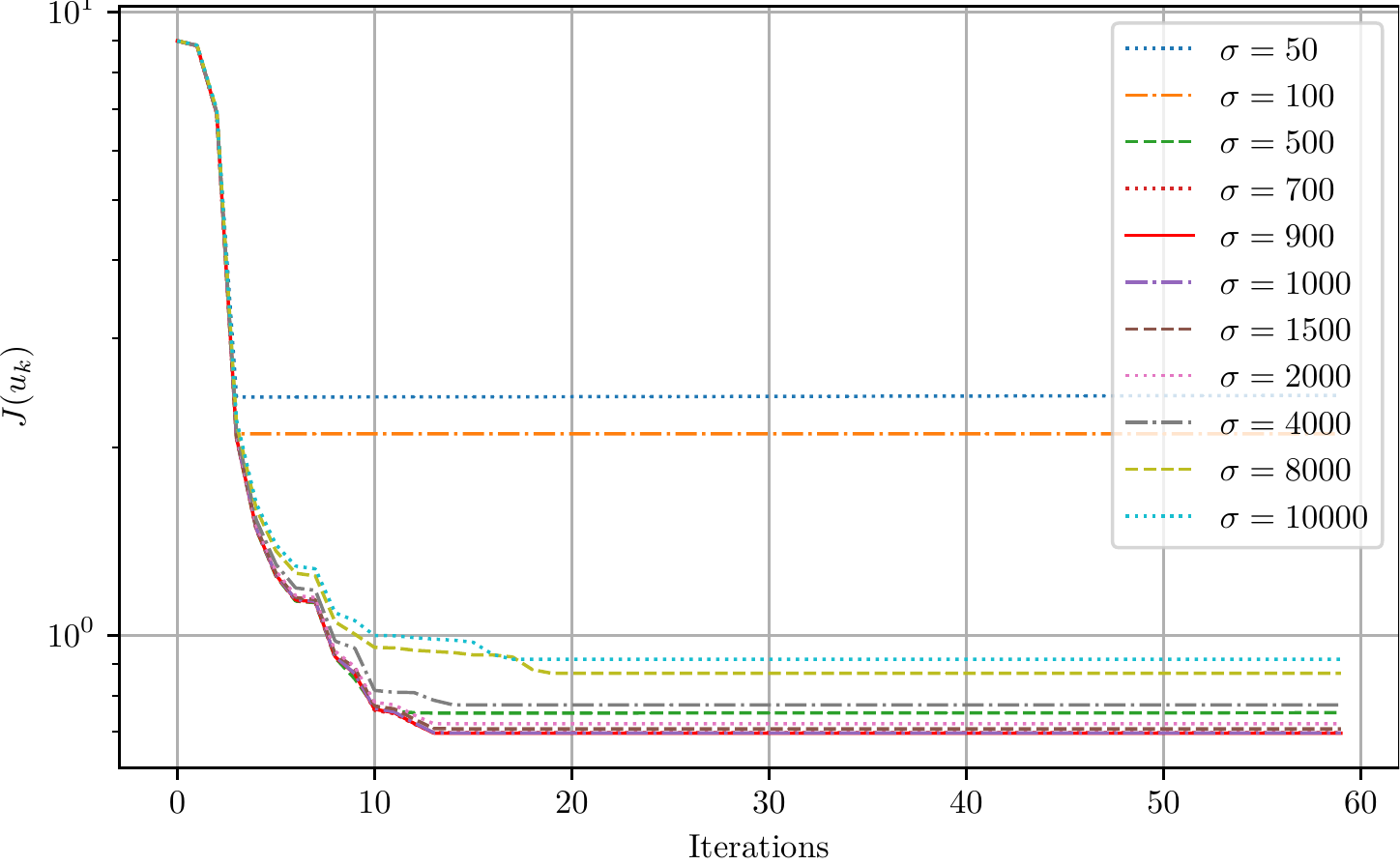}
    \captionof{figure}{Experiment 2: cost functional for $\gamma=1e+7$. Solid lines for $\sigma \geq 900$}
    \label{Fig:f1}
  \end{figure}

\begin{table}[H]
  \begin{center}
  \begin{tabular}{c|c| c c c c c c  } 
  \hline
  \multicolumn{8}{c}{Numerical performance for the EP-Algorithm}\\ \hline
   $\gamma$&$\beta$&$\sigma$ & $k$  &$|\langle J'(\mathbf{u}_k),\mathbf{w}_k\rangle|$  &$\|\Div \mathbf{u}_k\|_{L^1}$ & $J(\mathbf{u}_k)$ &t\\ \hline 
&&900&100 &   -0.626&7.55e-05& 0.06&91.29 \\
& &2000& 100&   -0.626&7.55e-05&0.14&91.29 \\
 &500 &4000& 100&   -0.626&7.55e-05 &0.29&91.18\\
 & &8000& 100&  -0.626&7.55e-05&0.59&91.22\\
 & &10000& 100&  -0.626&7.55e-05&0.74&91.20\\
    \cline{2-8}
    0   & &900& 100&  -0.313&3.78e-05& 0.03&91.46 \\
  & &2000& 100&  -0.313&3.78e-05&0.07&91.22 \\
  &1000 &4000& 100&  -0.313&3.78e-05 & 0.14&  91.28 \\
  & &8000& 100&  -0.313&3.78e-05 & 0.29&91.27 \\
  & &10000& 100&  -0.313& 3.78e-05 & 0.37&91.02 \\
   \hline
 &&900&12 &  3.7e-03&5.98e-6& -8.39&10.40 \\
& &2000& 12&  6.11e-03&4.80e-06&-8.36&10.52 \\
 &500 &4000& 10&  1.11e-02&4.55e-06 & -8.32&8.76\\
 & &8000& 10&  2.46e-02&4.28e-06&-8.23&8.83\\
 & &10000& 9&  3.41e-02&4.16e-06&-8.18&8.00\\
    \cline{2-8}
    1e+7    & &900& 14&  5.20e-03&6.53e-06& -8.30&12.20 \\
  & &2000& 14&  6.22e-03&5.15e-06&-8.27&12.16 \\
  &1000 &4000& 14&  5.89e-03&5.09e-06 & -8.22&  12.13 \\
  & &8000& 19&  6.19e-03&4.90e-06 & -8.13&14.54 \\
  & &10000& 17&  6.64e-03& 4.82e-06 & -8.08&14.45 \\
   \hline
   & &900& 12& 5.04e-04&7.58e-07& -8.41& 10.28 \\
   & &2000& 12&  4.16e-04&5.65e-07&-8.41&10.30 \\
   & 500 &4000& 12&  4.21e-04&5.64e-07 & -8.40&10.18\\
   &  &8000& 12&  2.64e-03&5.22e-07 &-8.39&10.52\\
   &&10000& 11&  2.8e-03& 5.11e-07&-8.39&9.57\\
   \cline{2-8}
   1e+8 &    &900&28 &  7.53e-04&7.75e-07& -8.33&24.02 \\ 
   & &2000& 28&  3.56e-04&5.90e-07&-8.32&24.04 \\
     &1000 &4000& 37&  6.43e-04&5.54e-07 & -8.32&31.69\\
       &&8000& 38&  1.34e-03&5.40e-07&-8.31&32.39\\
       &&10000& 27&  2.06e-03&5.19e-07&-8.30&23.33\\
      \hline
      & &900&14 & 2.31e-04&1.05e-07& -8.42& 12.19 \\
      &  &2000& 15& 3.60e-05&7.03e-08&-8.42&12.65 \\
      & 500 &4000& 15&  1.13e-04&6.60e-08 & -8.42&12.77\\
      &  &8000& 14&  2.07e-04&6.33e-08 &-8.42&11.98\\
      &  &10000& 14&  2.05e-04&6.33e-08 &-8.42&11.86\\
      \cline{2-8}
      1e+9     & &900& 25&  1.1e-04&9.74e-08& -8.33&21.09 \\
        & &2000& 26&  6.40e-05&6.18e-08&-8.33&21.89 \\
   &1000 &4000& 25&  9.16e-05&5.89e-08 & -8.33&21.17\\
        & &8000& 27&7.77e-05&6.01e-08&-8.33&22.65\\
        & &10000& 25&  1.32e-4&5.90e-08&-8.33&21.02\\
    \hline
  \end{tabular}
 \caption{Performance of the EP- Algorithm for different values of $\gamma$, $\beta$ and $\sigma$}
  \label{table:1.1}
  \end{center}
   \end{table}
      \begin{figure}[H]
        \centering
        \includegraphics[width=10cm]{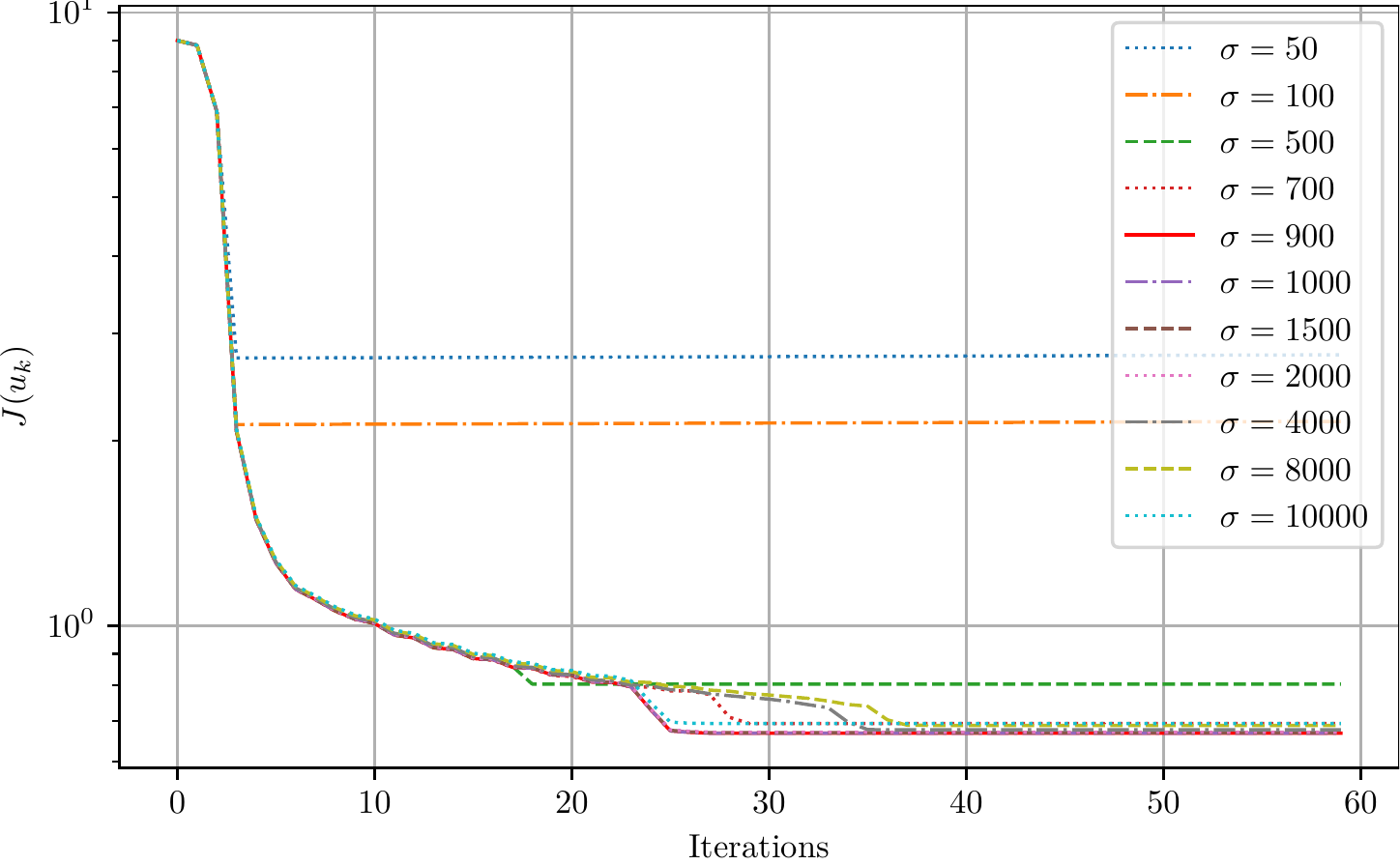}
      \captionof{figure}{Experiment 2: cost functional for $\gamma=1e+8$. Solid lines for $\sigma \geq 900$}
      \label{Fig:f2}
      \end{figure}
      \begin{figure}[H]
        \centering
        \includegraphics[width=10cm]{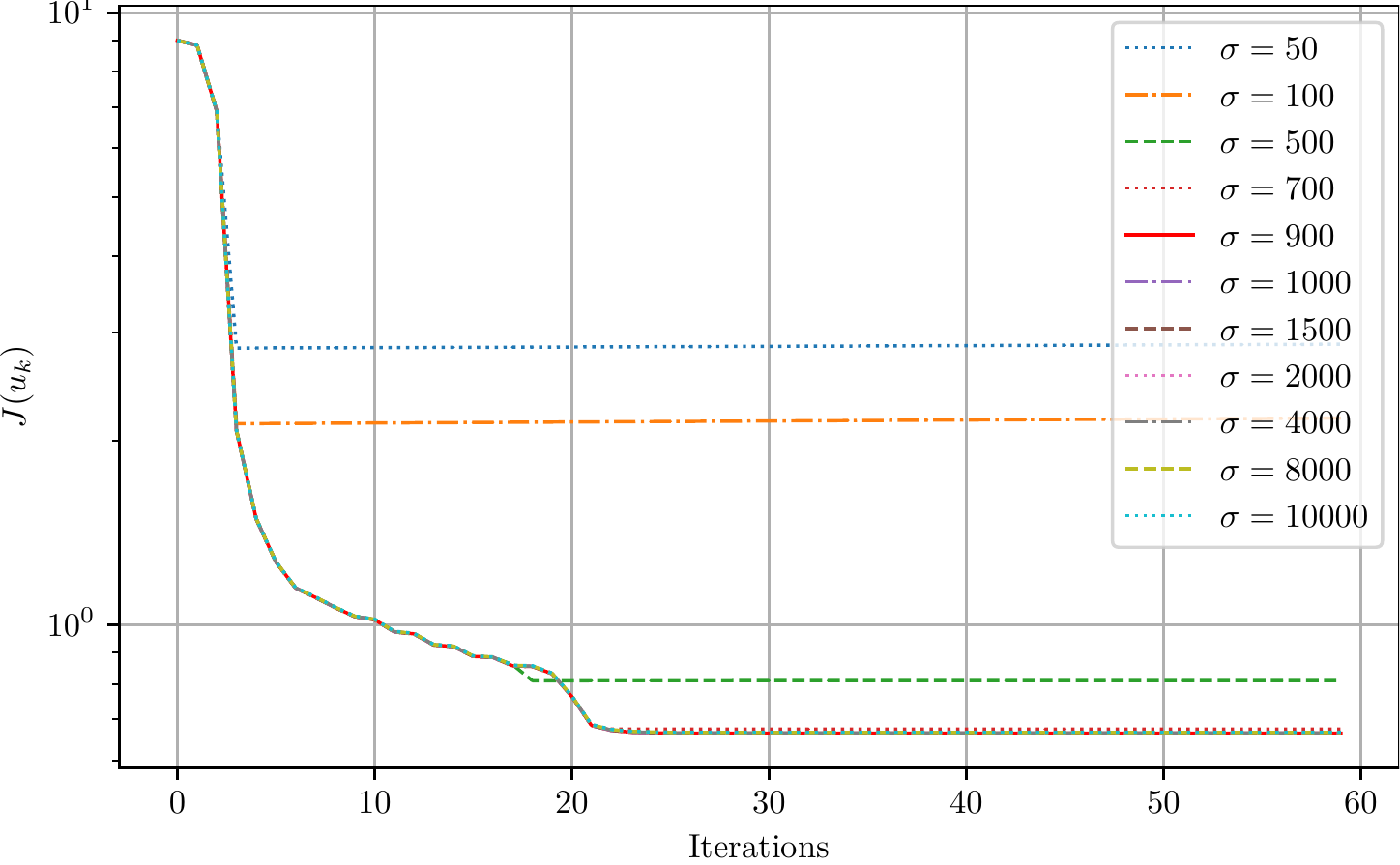}
      \captionof{figure}{Experiment 2: cost functional for $\gamma=1e+9$. Solid lines for $\sigma \geq 900$}
      \label{Fig:f3}
    \end{figure}

  \subsection{Comparison with Newton Semismooth method} \label{ss:comparison}
 We proceed to compare the exact penalization methods with the Newton semismooth method (SSN), which is also build on second-order information basis, and it is very well known by its superlinear rate of convergence properties (\cite[Ch.3]{Ulbrich}). Notice that SSN solves the same regularized problem in a context of a nonlinear system (\cite{DlRG2}). Therefore, the effect of the parameter $\beta$ is not an issue. Since the stopping criteria for each strategy differs, we compare a fix number of iterations of each algorithm. 
 In the following set of experiments we illustrate that the exact penalization method can deliver several numerical benefits and drawbacks comparing with SSN. For further comparison, we also consider the quadratic penalization (QP).
  \subsubsection{Experiment 3: convergence to a Bingham's analytical solution}
The original problem was regularized using the Huber--somoothing, depending on the parameter $\beta$. In the following experiment, the analytical solution of problem \eqref{eq:prob} is known, which allow us to compute the \textit{error}= $\| \mathbf{u}_{exact} - \mathbf{u}_k \|_{\mathbf{H}_0^1}$ at each $k$-th iteration for EP and SSN algorithms.

In this example, we solve the problem of a fluid flow between two parallel plates. Here, the velocity field $\mathbf{u}=(\mathbf{u}_1(y),\mathbf{0},\mathbf{0})$ is a scalar field that depends only on $y$ in the $x$-direction. The corresponding minimization functional is simplified since the strain-rate tensor is given by the gradient $\nabla$, i.e., we have 
\begin{equation*}
  \underset{\mathbf{u}\in \mathbf{H}_0^1}{\min}\, J(\mathbf{u}):= \frac{\mu}{2} \int_\Omega (\nabla \mathbf{u},\nabla \mathbf{u}) \,dx + g\int_\Omega |\nabla \mathbf{u}|\, dx - \int_\Omega \mathbf{f}\cdot\mathbf{u} \,dx.
\end{equation*}
The analytical solution is given by $\mathbf{u}_{exact}=(\boldsymbol{u}_1(y), 0,0)$ (see \cite[Sec. 6.1.1.1]{Aposporidis}), where:
\begin{equation*}
  \boldsymbol{u}_1(y)=
  \begin{cases}
    \frac{1}{8}[(1-2g)^2-(1-2g-2y)^2], & \text{ if } 0 \leq y < \frac{1}{2}-g,\\
    \frac{1}{8}(1-2g)^2, & \text{ if } \frac{1}{2}-g \leq y \leq \frac{1}{2} + g\\
    \frac{1}{8}[(1-2g)^2-(2y-2g-1)^2], & \text{ if } \frac{1}{2}+g < y \leq 1,
  \end{cases}
\end{equation*}
and the pressure drop $p$, is given by $p=-x$. Here, we set $g=0.3$. 

Following Remark \ref{re:bound_sigma0}, we chose  $\sigma_0 \approx \|\lambda\|_{L^2}|\Omega|^{1/2} \leq 17$ to guarantee the equivalence with the exact penalization formulation. For fixed $\gamma=1e+09$, we solve the problem varying $\beta$ as shown in Table \ref{table:order1}. As expected, as $\beta$ increases, the Huber regularization approximates better the original functional; therefore, the $error$ decays accordingly (see first column). 

 At first sight we observe, in Figure \ref{Fig:order1}, that SSN is faster in the first iterations.  However, we realize that the exact penalization second-order method continues to decrease the  $error$ with a pronounced fall in the last iterations, achieving a considerably lower error and cost (see Table \ref{table:order1} - fifth column) compared with SSN. Furthermore, the exact penalization algorithm computes approximate solutions with a more precise sparse divergence. Thanks to the continuous embedding $L^2(\Omega) \hookrightarrow L^1(\Omega)$ we have that $\|\Div \mathbf{u}_k\|_{L^1} \leq |\Omega|^{1/2} \|\Div \mathbf{u}_k\|_{L^2} $. Therefore, we chose the $L^2$--norm of the divergence as a reference for all methods. We observe that the $L^2$--norm of the divergence is effectively smaller for EP method. 

\begin{table}[H]
  \footnotesize
  \begin{tabular}{l|l cc|cc|cc|cc}
    \toprule
    \multirow{2}{*}{$\beta$}&{}&
    \multicolumn{2}{c}{$\|\mathbf{u}_{exact} - \mathbf{u}_k \|_{\mathbf{H}_0^1}$} &
      \multicolumn{2}{c}{$\|\Div \mathbf{u}_k\|_{L^2}$} &
      \multicolumn{2}{c}{ $J( \mathbf{u}_k)$} &
      \multicolumn{2}{c}{ time (s)}\\
        & {k}&{EP} & {SSN} &  {EP} & {SSN} & {EP} & {SSN} & {EP} & {SSN} \\
        \hline
        &1& 0.121& 0.100&  1.28e-13& 2.46e-08& 0.032& 4.79e-03& 1.19& 1.46\\
        &10& 0.015& 0.002& 9.59e-12& 3.21e-08& -2.43e-03& -2.75e-03& 14.88& 15.62\\
       100 &20& 1.34e-03& 2.99e-03&  9.92e-12& 3.21e-08& -2.93e-03& -2.75e-03& 26.80& 31.34\\
        &30& 1.34e-03& 2.99e-03&  9.92e-12& 3.21e-08& -2.93e-03& -2.75e-03& 39.60& 46.35\\
        &40& 1.34e-03& 2.99e-03&  9.92e-12& 3.21e-08& -2.93e-03& -2.75e-03& 52.91& 64.63\\
        \hline
        &1& 0.121& 0.100&  2.05e-13& 2.49e-08& 0.033& 0.005& 1.16& 1.33\\
        &10& 0.013& 2.82e-03& 4.94e-11& 6.35e-08& -2.40e-03& -2.42e-03& 15.32& 14.477\\
      500 &20& 2.2e-03& 2.82e-03& 5.44e-11& 6.35e-08& -2.70e-03& -2.42e-03& 27.83& 28.614\\
      &30& 2.68e-04& 2.82e-03&  5.42e-11& 6.35e-08& -2.70e-03& -2.42e-03& 40.24& 42.795\\
      &40& 2.67e-04& 2.82e-03&  5.42e-11& 6.35e-08& -2.70e-03& -2.42e-03& 52.61& 57.00\\
      \hline
      &1& 0.121& 0.100&  2.53e-13& 2.49e-08& 0.033& 0.005& 1.16& 1.29\\
      &10& 0.013& 2.81e-03&  3.44e-11& 6.48e-08& -2.29e-03& -2.38e-03& 14.94& 14.45\\
      1000&20& 7.41e-03& 2.81e-03 & 3.53e-11& 6.48e-08& -2.65e-03& -2.38e-03& 27.63& 28.71\\
      &30& 1.34e-04& 2.81e-03&  3.54e-11& 6.48e-08& -2.69e-03& -2.38e-03& 40.28& 42.99\\
      &40& 1.34e-04& 2.81e-03&  3.54e-11& 6.48e-08& -2.69e-03& -2.38e-03& 52.69& 57.31\\
      \hline
      &1& 0.122& 0.100&  5.27e-13& 2.49e-09& 0.033& 0.005& 1.14& 1.30\\
      &10& 0.013& 2.81e-03& 4.32e-11& 6.48e-08& -2.39e-02& -2.34e-03& 15.19& 14.89\\
      5000&20& 0.010& 2.81e-03&  4.57e-11& 6.48e-08& -2.59e-03& -2.34e-03& 28.29& 29.43\\
      &30& 7.79e-03& 2.81e-03&  4.67e-11& 6.48e-08& -2.63e-03& -2.34e-03& 42.35& 43.77\\
      &40& 3.16e-05& 2.81e-03&  5.36e-11& 6.48e-08& -2.67e-03& -2.34e-03& 56.34& 58.110\\
    \bottomrule
  \end{tabular}
  \caption{Experiment 3: 40 iterations of EP vs SSN with various values of $\beta$ for flow between two plates}
  \label{table:order1}
\end{table}

\begin{minipage}[b]{1\textwidth}
  \centering
  \includegraphics[width=10cm]{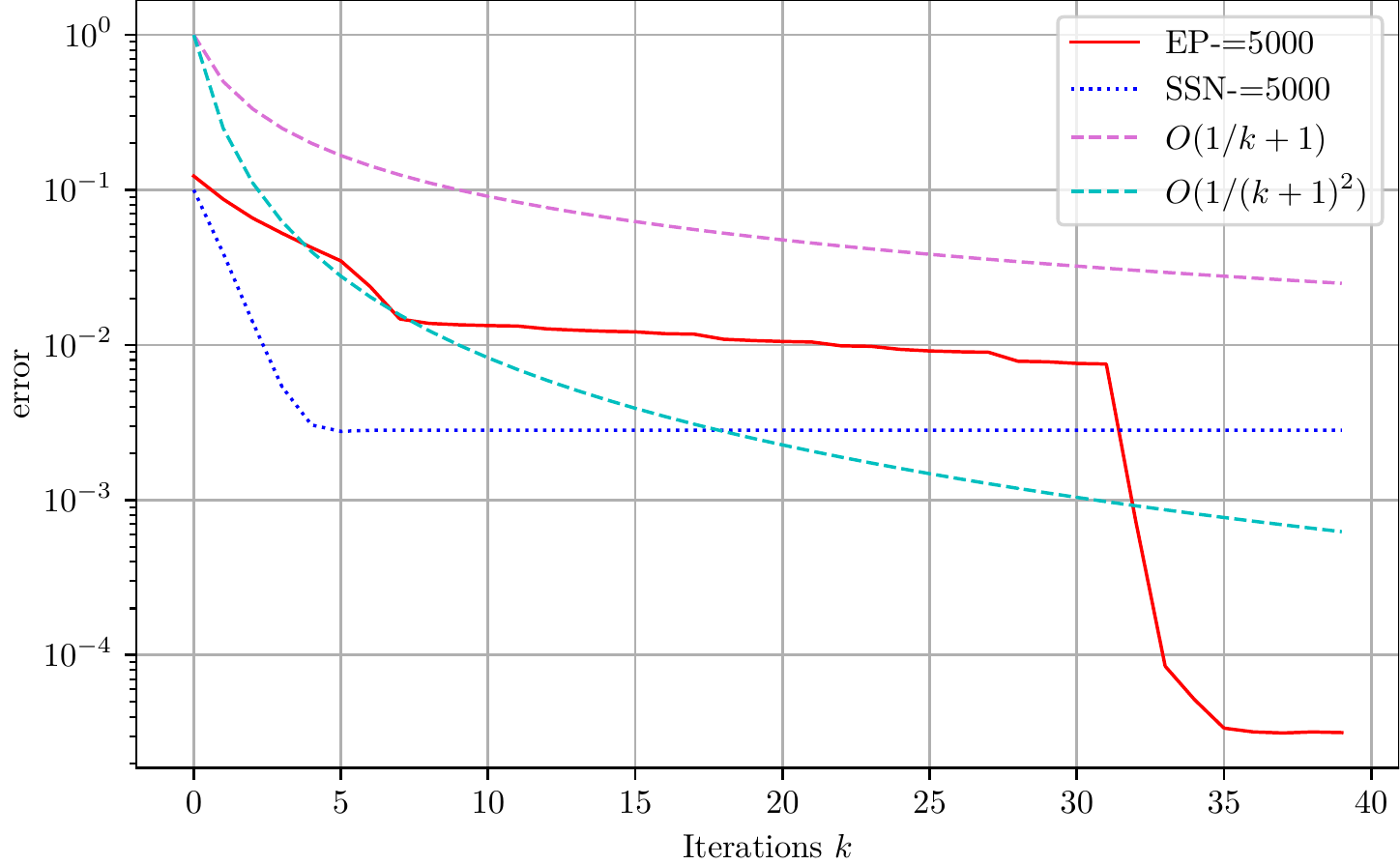}
  \captionof{figure}{Experiment 3: error decay comparison - flow between two plates}
  \label{Fig:order1}
\end{minipage}
\subsubsection{Experiment 4: Comparison with Second Order Methods}
In this experiment we compare three methods based on second other information: SSN,  EP and Quadratic penalization (QP) for the benchmark test presented in section \ref{rotational_flow}. We compare the first 100 iterations of each algorithm in Table \ref{table:2}. The EP-Algorithm is competitive with SSN and the Quadratic penalization method since no major oscillations in the algoritms's performance are shown after the 25th iteration for the three strategies. However, the advantage of the exact penalization is evident when reaching the restriction $\|\Div \mathbf{u}_k\|$ close to zero. Observe in Figures \ref{Fig:7} - \ref{Fig:7.1} that, for several $\sigma$ values, the lowest magnitud for the velocity divergence is achieved by the EP-Algorithm. This is somehow expected by sparsification properties of the $L^1$ - norm. For the same $\sigma$ values the quadratic method hardly achieves an order of $1e-04$. And, in the case of SSN, the velocity divergence norm is greater than in the EP-Algorithm. This behavior is depicted in Table \ref{table:2} for the $L^2$-norm of the velocity divergence of the approximated solution.

The decay of the objective function is plotted in Figure \ref{Fig:8}. Here, EP-Algorithm and QP-Algorithm attain smaller values than the SSN method. Let us point out that the optimal value of the objective functional is negative (see Table \ref{table:2}), therefore a positive value  has been added to the results at each iteration  in order to plot them in logarithmic scale, avoiding negative values. 

    
    \begin{table}[H]
      \begin{tabular}{l ccc|ccc|ccc}
        \toprule
        \multirow{2}{*}{$k$} &
          \multicolumn{3}{c}{  $\|\Div \mathbf{u}_k\|_{L^2}$} &
          \multicolumn{3}{c}{ $J( \mathbf{u}_k)$} &
          \multicolumn{3}{c}{ time (s)}
          \\
          \cline{2-10}
          & {EP} & {QP} & {SSN} & {EP} & {QP} & {SSN} &  {EP} & {QP} & {SSN}\\
          1& 7.170e-09& 0.00167& 0.000413& -0.07532& -0.0855& 2.520& 0.59& 0.652& 0.7827\\
          10& 4.576e-08& 0.00277& 0.0140& -7.890& -8.324& 7811.370& 5.59& 5.34&14.75\\
          20& 5.056e-08& 0.00274& 1.122e-04& -8.144& -8.3492& -6.994& 11.96& 11.35& 22.65\\
          \textbf{25}& 7.485e-08& 0.00273& 7.836e-05& -8.241& -8.352& -7.708& 15.08& 14.25& 26.42\\
          30& 7.486e-08& 0.00273& 7.808e-05& -8.241& -8.352& -7.708& 18.54& 17.18& 30.14\\
          50& 7.487e-08& 0.00273& 7.808e-05& -8.241& -8.352& -7.708& 32.53& 28.89& 45.22\\
          80& 7.490e-08& 0.00273& 7.808e-05& -8.241& -8.352& -7.708& 53.72& 46.59& 67.88\\
          100& 7.491e-08& 0.00273& 7.808e-05& -8.241& -8.352& -7.708& 67.93& 58.76& 83.15\\
        \bottomrule
      \end{tabular}
      \caption{Experiment 4: comparison between EP, QP and SSN algorithms with parameters: $\gamma=1e+9$, $\beta=1e+3$ and $\sigma=3e+3$}
        \label{table:2}
    \end{table}

      \begin{figure}[H]
        \includegraphics[width=8cm]{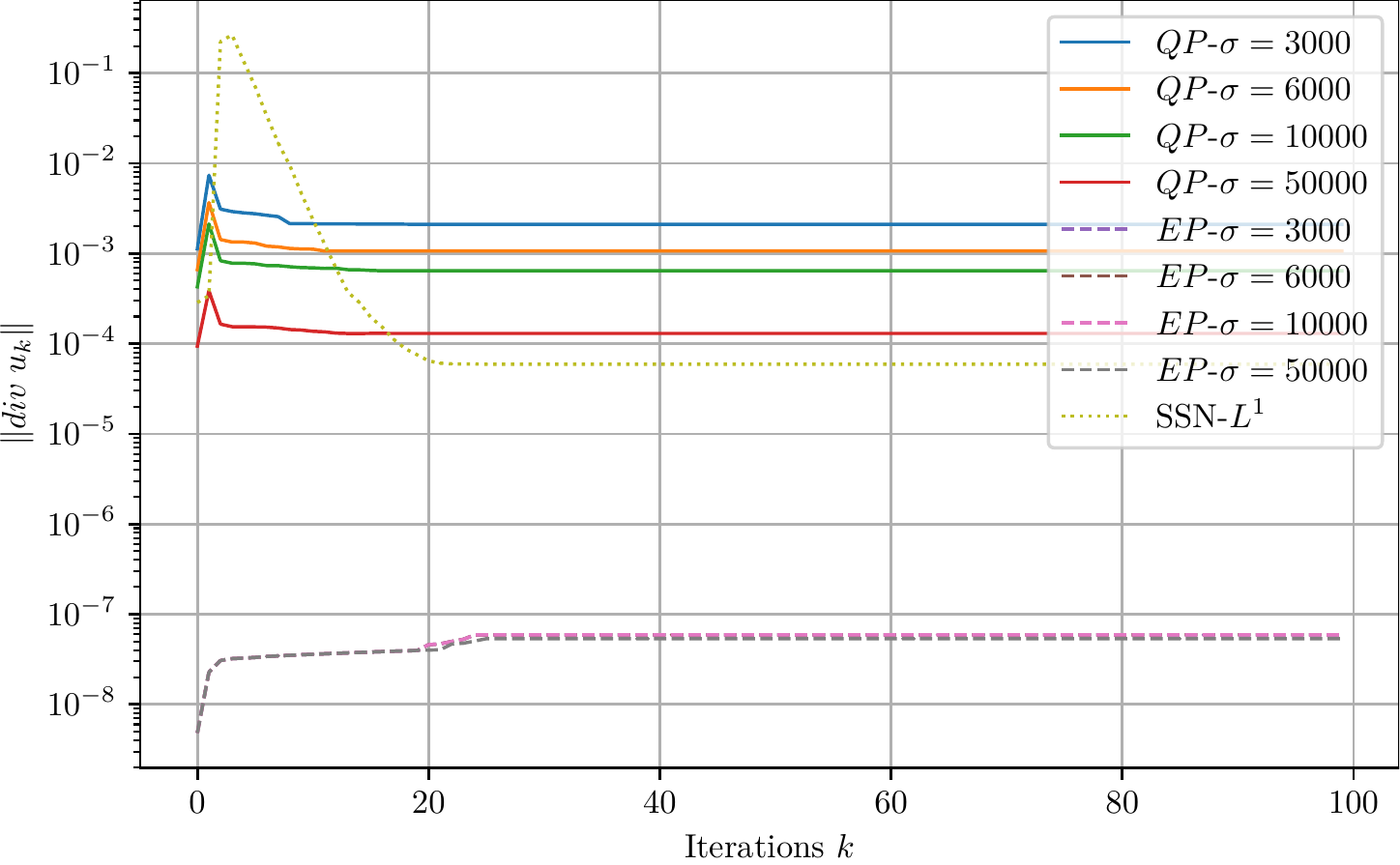}
        \captionof{figure}{\small{Experiment 4: comparison between EP, QP and SSN algorithms of the divergence history of approximated solutions in $L^1$-norm}}
        \label{Fig:7}
      \end{figure}
      \begin{figure}[H]
        \includegraphics[width=8cm]{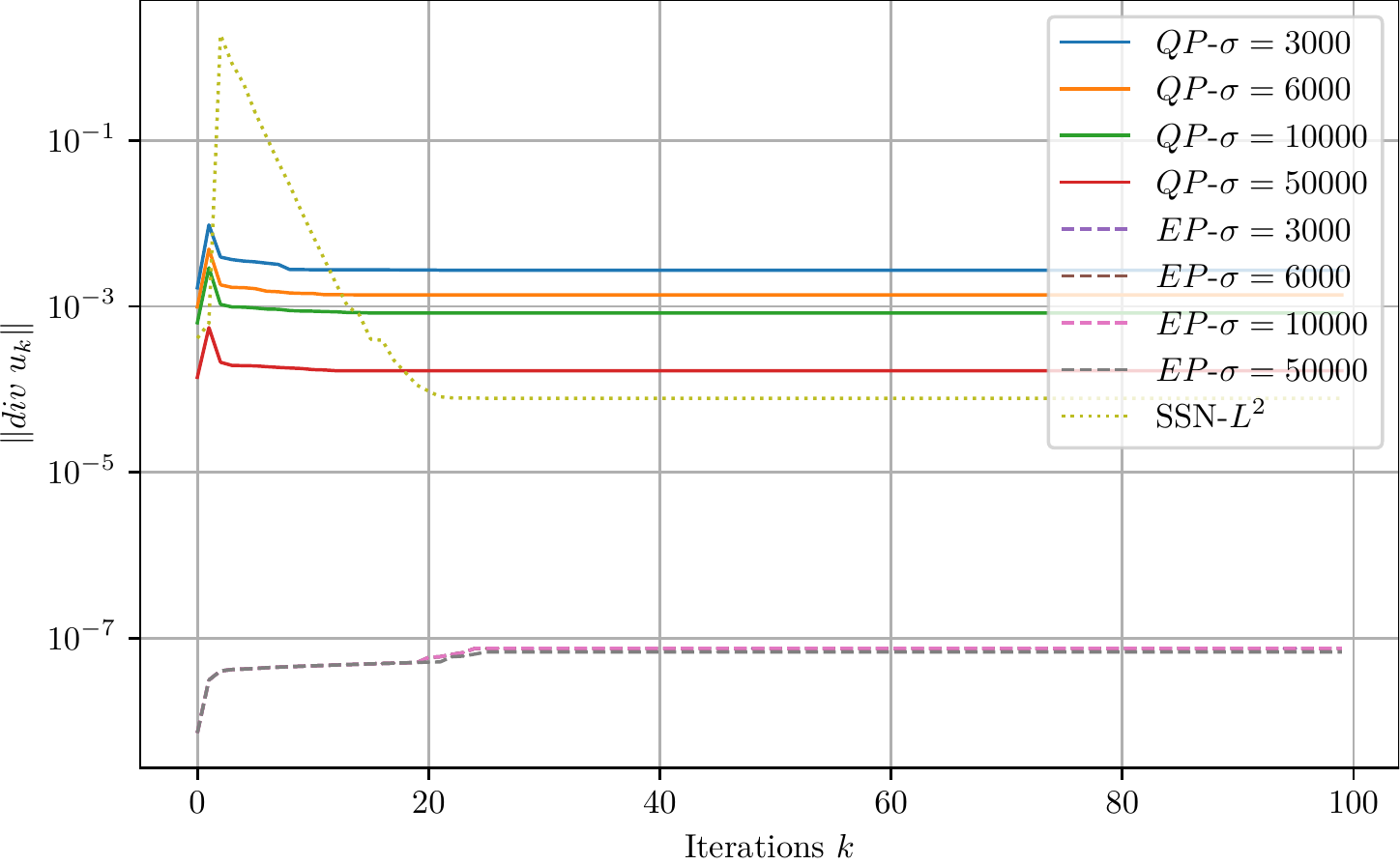}
        \captionof{figure}{\small{Experiment 4: comparison between EP, QP and SSN algorithms of the divergence history of approximated solutions in $L^2$-norm}}
        \label{Fig:7.1}
      \end{figure}
      \begin{figure}[H]
        \includegraphics[width=8cm]{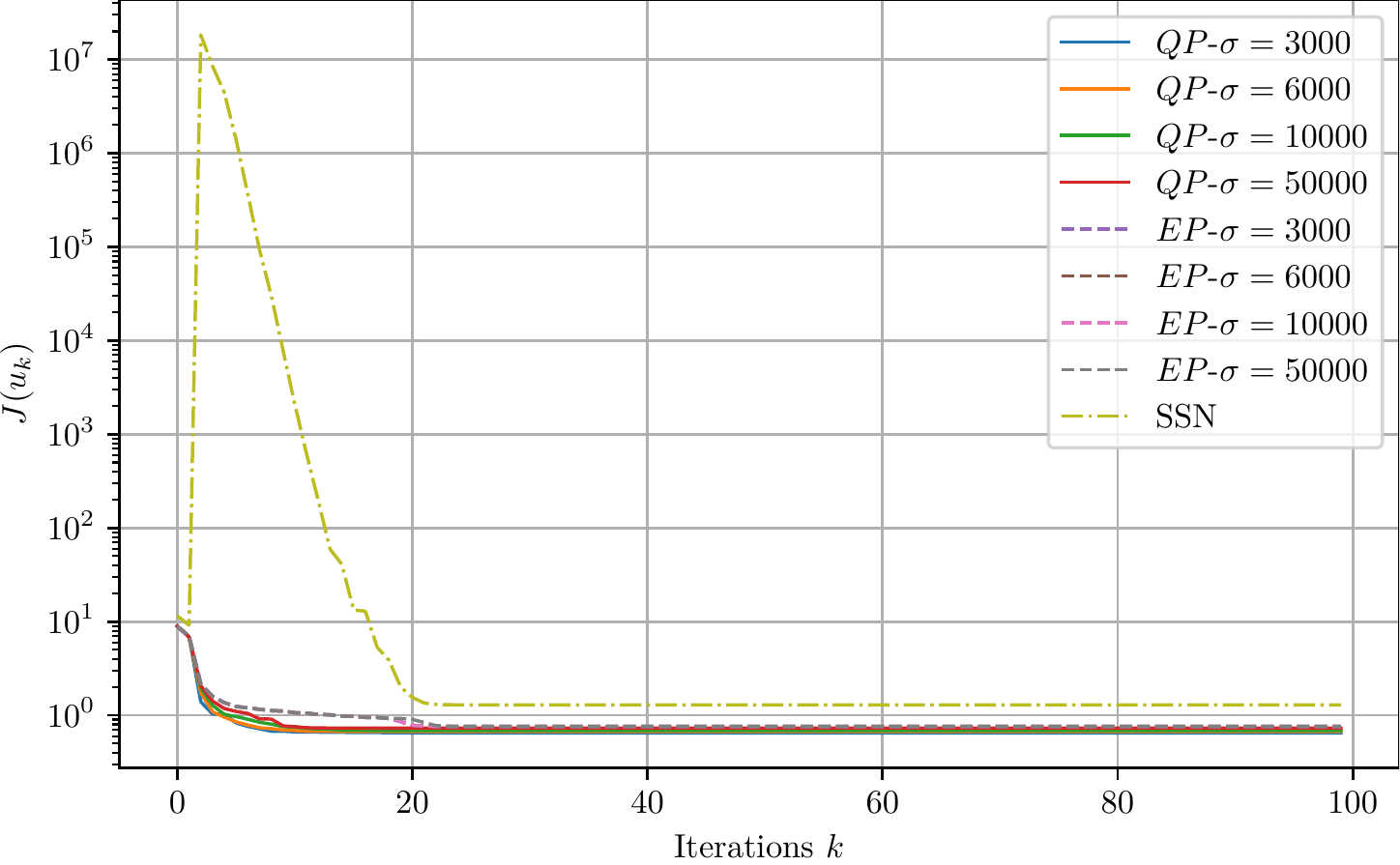}
        \captionof{figure}{\small{Experiment 4: comparison between EP, QP and SSN algorithms of the cost functional $J(\mathbf{u}_k)$}}
        \label{Fig:8}
      \end{figure}

\subsection{Numerical Experiments in 3D Geometries}
In this section we examine the EP-Algorithm in 3D geometries. We take advantage of FEniCS versatility for testing the EP-Algorithm using three-dimensional finite elements described in the Implementation Details, and the FEniCS parallelization capabilities for solving the associated variational problems within the algorithm. We run these programs on a high performance computing system HP ProLiant BL460c Gen8.
\subsubsection{Experiment 5: Cube }
In this experiment we consider a Bingham fluid in a cubic geometry. We assume a laminar flow with constant drop pressure $c$ along the z-axis. The drop in pressure is considered in the periodic boundary conditions of the model (see \cite[Sec. 6.2]{Laaber}). Therefore, the associated minimization problem reads as: 
$$J=\mu \int_\Omega \mathcal{E}\mathbf{u}:\mathcal{E} \mathbf{u} \,dx + \int_\Omega \Psi (\mathcal{E}\mathbf{u})\, dx - \int_\Omega f \mathbf{u}\,dx  - \int_\Gamma c  \mathbf{n}|_{z=0} \mathbf{u}|_{z=0} \,ds + \sigma \|\Div \mathbf{u}\|_{L^1} $$ 
here $\Gamma=[0,1] \times[0,1] \times \{0\}$, $f=0$, $c=10$, $g=0.5\sqrt{2}$, $\beta=1e+ 3 $, $\gamma=1e+7$ and $\sigma=1e+4$. The unit cube is discretized with tetrahedrons with step size $h=1/20$. Figure \ref{Fig:22} shows the velocity field. In the center of the cube the fluid acts like a rigid material as well as the corners of the cube. In Figure \ref{Fig:23} the Frobenius norm $|\mathcal{E} \mathbf{u}|$ and the discretization are depicted in the geometry cut by a parallel plane to the y axis. Here, the plug zones are colored in light blue. The algorithm performed 4 iterations with the following values $J_{\sigma}=-1.098$, $|\langle J'\mathbf{u},\mathbf{w}\rangle|=9.627 e-05$, and $\|\Div \mathbf{u}\|_{L^1}=1.57e-06$.  Taking advantage that FEniCS run in parallel  using MPI and without modifying the algorithm, in Table \ref{table:parallel} we report on the comparison between parallel runtime in several CPUs for EP-Algorithm and SSN. Despite the efficiency of the current implementation deteriorates, the time reduction is significant when more CPUs are added. Also, its execution time escalates better that the execution time in SSN. However, the percentage of time reduction is similar for both methods as more CPUs are incorporated to the computation process.
Because of memory limitations, there was not possible to run the experiment with the mesh size $h=1/20$ in one core. Therefore, we calculate the speedup and the efficiency with the execution time reference in 2 cores. 

\begin{figure}[H]
  \centering
  \begin{minipage}[b]{0.45\textwidth}
    \includegraphics[width=7.5cm]{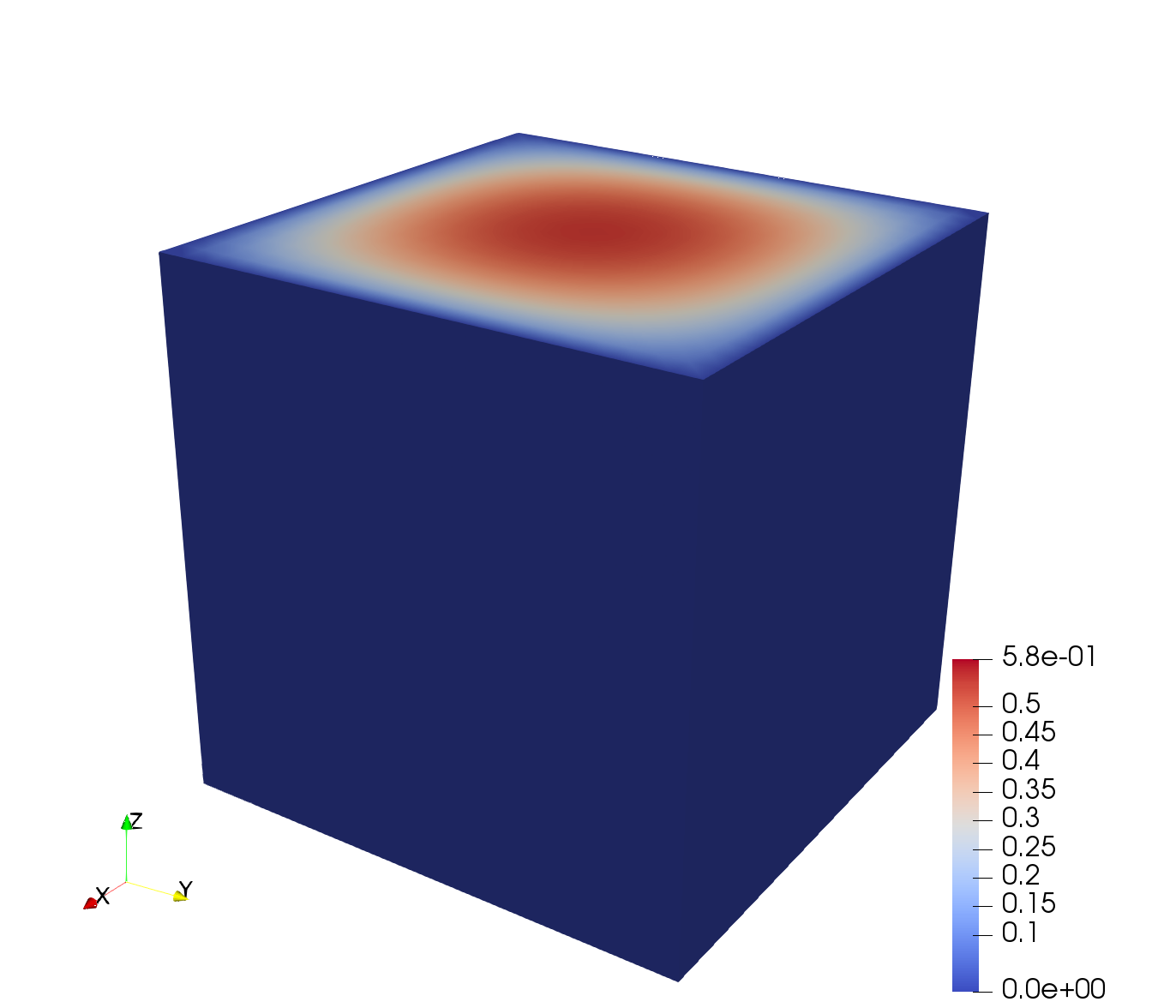}
    \caption{Velocity field}
    \label{Fig:22}
  \end{minipage}
  \begin{minipage}[b]{0.52\textwidth}
    \includegraphics[width=8.0cm]{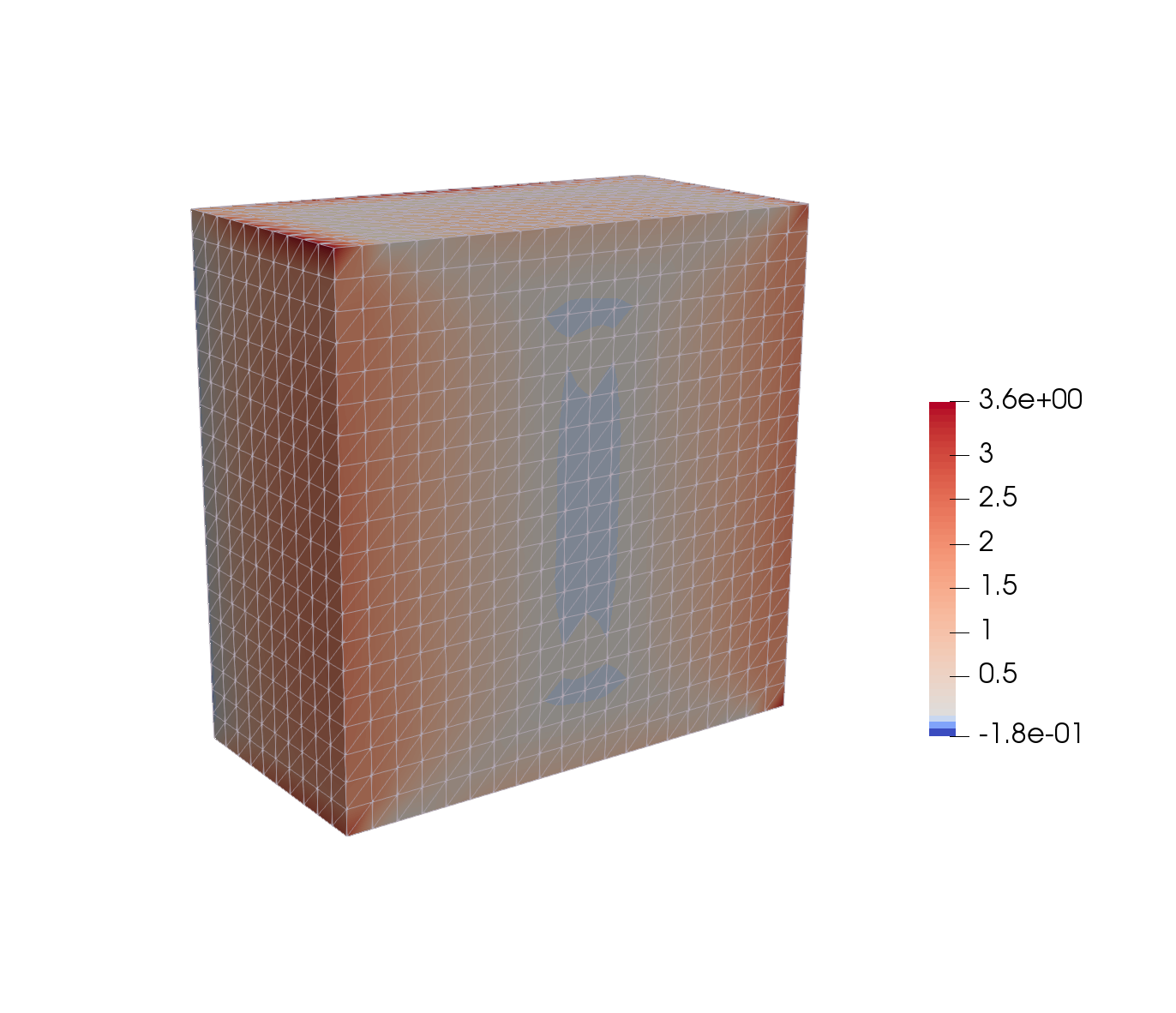}
    \caption{Plug flow }
       \label{Fig:23}
  \end{minipage}
\end{figure}

\begin{table}[H]
  \begin{center}
  \begin{tabular}{c|c |c  c c c@{\extracolsep{1cm}} }
  \toprule
  \multicolumn{6}{c}{\textbf{Parallel run time}}\\ \hline
   &\textbf{No. cores} &2 &6&12&24 \\ \cline{2-6} 
   &Time  (s) &1218.2 &709.8&395.8&240.9  \\ \cline{2-6} 
   EP& \% time reduction &100 \% &58.2\%& 32.4 \%&19.7\%  \\ \cline{2-6} 
  &Speedup &1 & 1.71& 3.07&5.05   \\ \cline{2-6} 
   &Efficiency &1 &0.28 &0.25&0.21 \\ \hline
   &Time  (s) & 5359.87 & 2828.89&2123.41&1619.87 \\ \cline{2-6}
  & \% time reduction &100 \% &52.7\%& 39.6 \%&30.2\%  \\ \cline{2-6}
  SSN&Speedup & 1&1.89  &2.52 &3.30 \\ \cline{2-6}
   &Efficiency &1 & 0.31&0.21&0.13\\ 
   \bottomrule
  \end{tabular}
  \caption{Experiment 5: EP vs. SSN algorithms performance in several cores}
  \label{table:parallel}
  \end{center}
  \end{table}

\subsubsection{Experiment 6: Lid-driven cavity}
 We present a lid-driven  viscoplactic flow inside a unite cube. The geometry is discretized with tetrahedrons with step size $h=1/30$. In this case the body forces are  $f(\textbf{x})= \textbf{0}$, since the motion is given  by a moving lid, i.e., we have:
\begin{equation*}
  \textbf{u}_D(\textbf{x})=
  \begin{cases}
    (1,0,0)^\top & \text{ if } x_3=1\\
    (0,0,0)^\top & \text{ otherwise }
  \end{cases}
\end{equation*}
These boundary conditions add a new constraint to our optimization problem, i.e.
\begin{equation}\label{eq:lid-driv}
  \begin{aligned}
  &\underset{\mathbf{u}\in \mathbf{H}^1(\Omega)}{\min} J(\mathbf{u}):= \frac{\mu}{2}\int_\Omega \mathcal{E}\mathbf{u}:\mathcal{E}\mathbf{u} \,dx + \int_\Omega \Psi (\mathcal{E}\mathbf{u})\, dx - \int_\Omega \mathbf{f}\cdot\mathbf{u}\,dx + \sigma \|\Div(\mathbf{u})\|_{L^1} \\
  & \text{subject to}  \hspace{2mm} \mathbf{u}=\mathbf{u}_D \hspace{2mm} \text{on} \hspace{2mm} \partial \Omega .
  \end{aligned}
  \end{equation}
To cope with this new constraint, let us fix $\mathbf{u_0} \in U=\{\mathbf{u} \in \mathbf{H}^1(\Omega) | \Div(\mathbf{u})=0, \mathbf{u}|_{\partial \Omega} = \mathbf{u}_D\}$. Then, we can characterize the solution $\mathbf{\bar{u}}$ of \eqref{eq:lid-driv} by  $\mathbf{\bar{u}}=\mathbf{u_0} + \mathbf{\hat{u}}$, where $\mathbf{\hat{u}}$ is the solution of
\begin{equation*}\label{eq:lid-driv2}
  \begin{array}{lll}
    \displaystyle \underset{\mathbf{u}\in \mathbf{H}^1_0(\Omega)}{\min} J(\mathbf{u}) &:=&  \displaystyle \frac{\mu}{2}\int_\Omega \mathcal{E}(\mathbf{u}+ \mathbf{u_0}):\mathcal{E}(\mathbf{u}+ \mathbf{u_0})\,dx +  \displaystyle  \int_\Omega \Psi (\mathcal{E}(\mathbf{u}+ \mathbf{u_0}))\, dx \\
    &-& \displaystyle   \int_\Omega \mathbf{f} (\mathbf{u}+ \mathbf{u_0})\,dx + \sigma \|\Div(\mathbf{u})\|_{L^1}.
  \end{array}
  \end{equation*}
The parameters have the following setting: $g=2$, $\beta=1e+ 3$, $\gamma=1e+9$, $\mu=0.5$ and $\sigma=1e+4$. Figure \ref{Fig:14} show the velocity field of the fluid in the cube cut by half in the y axis.  The fluid rotates in the interior of the geometry however, thanks to this rotation the material moves without continuous deformation in the center of the cube. Figure \ref{Fig:16} shows the plug zones in light blue. The numerical performance of the algorithm is displayed in Table \ref{table:11}. The divergence norm, the number of iterations and the stopping criteria behave similar to the 2D case. For instance $\|\Div \mathbf{u}_k\|_{L^1}$ achieves an order of $e-08$ and the stopping criteria $|\langle J'(\mathbf{u}_k),\mathbf{u}_k\rangle|$ drops close to $e-04$.
\begin{figure}[H]
  \centering
  \begin{minipage}[b]{0.45\textwidth}
    \includegraphics[width=8cm]{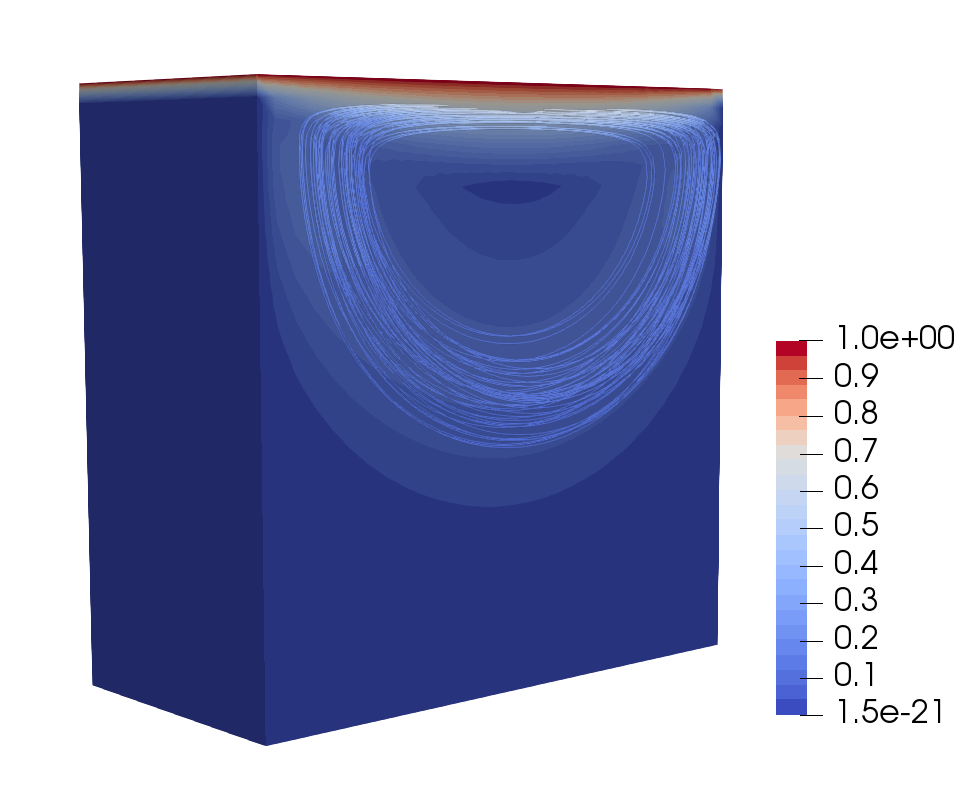}
  \end{minipage}
  \begin{minipage}[b]{0.45\textwidth}
    \includegraphics[width=7.8cm]{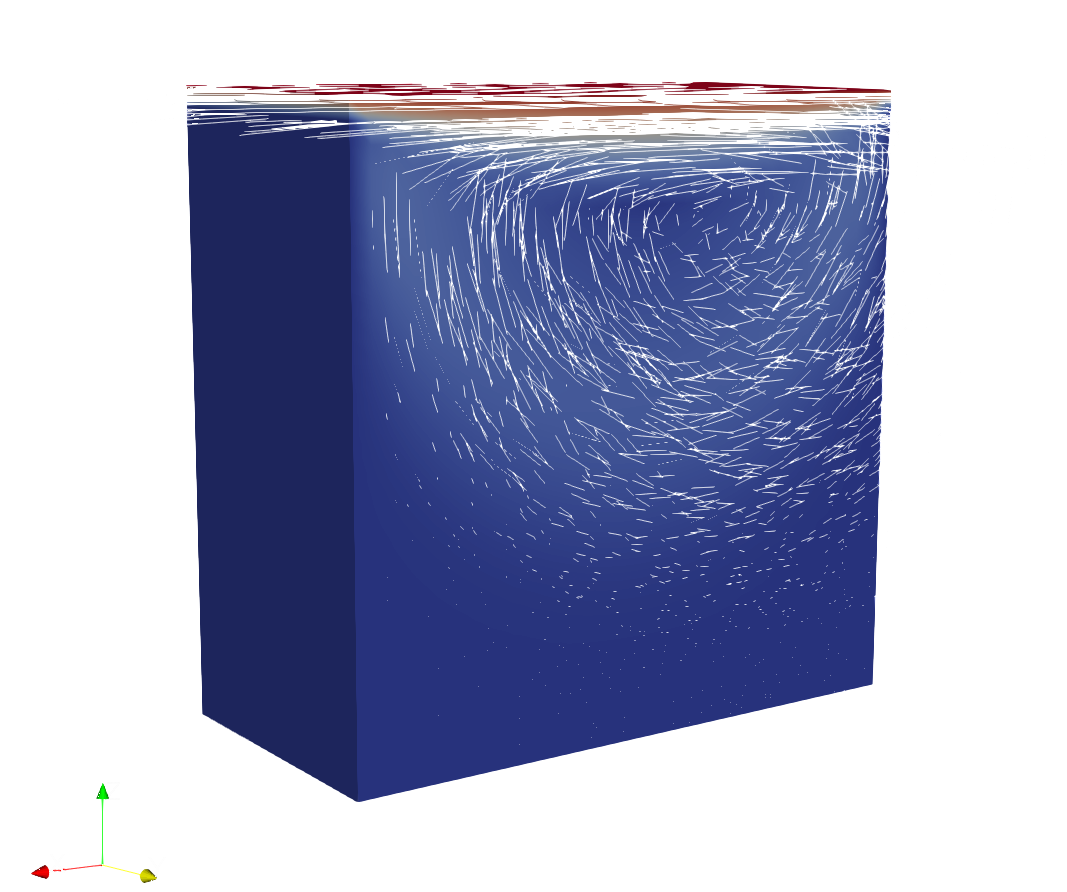}
  \end{minipage}
  \caption{Experiment 6: stream lines and velocity field of lid-driven flow}
  \label{Fig:14}
  \centering
  \begin{minipage}[b]{0.85\textwidth}
    \includegraphics[width=11cm]{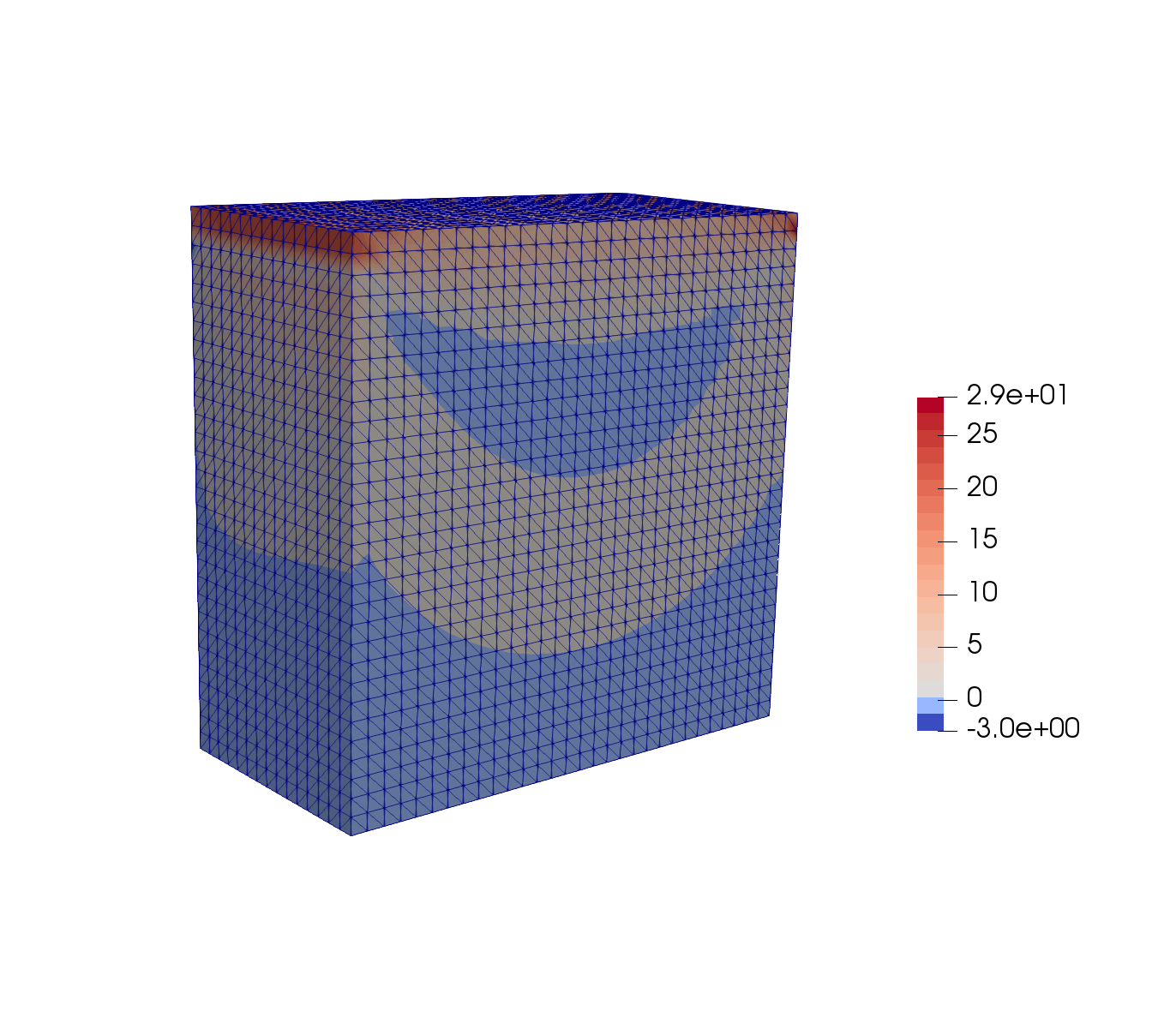}
    \caption{Experiment 6: plug zones of lid-driven flow and spatial mesh}
    \label{Fig:16}
\end{minipage}
\end{figure}
\begin{table}[H]
      \begin{center}
      \begin{tabular}{c c c  c @{\extracolsep{0.5cm}} }
      \hline
      \multicolumn{4}{c}{\textbf{Exact Penalization}}\\ \hline
       $k$ & $\|\Div \mathbf{u}_k\|_{L^1}$&  $J(\mathbf{u}_k)$ &$|\langle J'(\mathbf{u}_k),\mathbf{u}_k\rangle|$ \\ \hline
       1& 1.59e-09& 5.396&0.861 \\
      3& 3.65e-09& 5.179& 0.674\\
      5&  5.72e-09& 5.101& 0.042\\
      7& 1.06e-08&  5.088& 0.024\\
      9& 1.16e-08&  5.087& 0.018\\
      11& 1.36e-08& 5.086& 3.7e-03\\
      13& 1.44e-08& 5.085& 5.1e-04\\
      15& 1.51e-08& 5.085& 1.4e-04\\
      \hline
      \end{tabular}
      \caption{Experimemt 6: 3D Lid-driven cavity with $g=2, \beta=1e+3, \gamma=1e+9$}
      \label{table:11}
      \end{center}
\end{table}

\section{Appendix}
\begin{lema}\label{lem:phi_lips_cont}
Let $\gamma$ and $\sigma$ be two positive  constants. The function $\phi:\mathbb{R}\rightarrow\mathbb{R}$ defined by $\phi(a):=\gamma \sigma \frac{ a}{\max(\sigma, \gamma |a|)}$ is Lipschitz continuous.
\end{lema}
\begin{proof}
Let us start by rewriting $\phi(a)$ as $\phi(a)=\gamma\sigma\frac{a}{\phi_m(a)}$, with $\phi_m(a):=\max(\sigma,\gamma |a|)$. Next, we notice that the max function is globally Lipschitz continuous with constant $L_{max}$. This fact implies that
\begin{equation}\label{maxlip}
\begin{array}{lll}
|\phi_m(a_1)-\phi_m(a_2)|= |\max(\sigma,\gamma |a_1|)- \max(\sigma,\gamma |a_2|)|\vspace{0.2cm}\\\hspace{3.5cm}\leq \gamma L_{max}||a_1|-|a_2||\leq  \gamma L_{max}|a_1-a_2|,\,\, \forall a_1,a_2\in \mathbb{R}.
\end{array}
\end{equation}
We conclude that $\phi_m$ is Lipschitz continuous. Considering this result, we have that
\[
\begin{array}{lll}
|\phi(a_1)-\phi(a_2)| &=& \left|\gamma \sigma \frac{a_1}{\phi_m(a_1)} -\gamma \sigma \frac{a_2}{\phi_m(a_2)}\right|\vspace{0.2cm}\\&=& \gamma\sigma \left|\frac{a_1}{\phi_m(a_1)}-\frac{a_2}{\phi_m(a_2)} +\frac{a_1}{\phi_m(a_2)}- \frac{a_1}{\phi_m(a_2)}\right|\vspace{0.2cm}\\&=& \gamma\sigma \left|a_1\left(\frac{\phi_m(a_2)- \phi_m(a_1)}{\phi_m(a_1)\phi_m(a_2)}\right)\right| +\gamma\sigma\left|\frac{1}{\phi_m(a_2)} (a_1 - a_2)\right|.
\end{array}
\]
Now, it is clear that $0<\sigma \leq \phi_m(a_2)$, which implies that $\frac{1}{\phi_m(a_2)}\leq\frac{1}{\sigma}$. By plugging this inequality in the expression above, we have that
\[ 
\begin{array}{lll}
|\phi(a_1)-\phi(a_2)| &\leq& \gamma\sigma \left|\frac{a_1}{\phi_m(a_1)}\left(\frac{\phi_m(a_2)- \phi_m(a_1)}{\sigma}\right)\right| +\gamma|a_1 - a_2|\vspace{0.2cm}\\ &=& \gamma \left|\frac{a_1}{\phi_m(a_1)}\right|\left|\phi_m(a_2)- \phi_m(a_1)\right|+ \gamma|a_1-a_2|.
\end{array}
\]
Finally, sice $\left|\frac{a_1}{\phi_m(a_1)}\right|\leq\frac{1}{\gamma}$, we conclude, thanks to \eqref{maxlip}, that
\[
|\phi(a_1)-\phi(a_2)| \leq \gamma(L_{max} +1)|a_1-a_2|.
\]

Regarding the semismoothness of $\phi$, note that the absolute value $|\cdot |: \mathbb{R} \rightarrow \mathbb{R}$ and the function $\max(0, \cdot) : \mathbb{R} \rightarrow \mathbb{R}$ are  both semismooth (see \cite[Sec. 2.5]{Ulbrich} and \cite[Lemma 3.1]{Kunisch-Ito} respectively). Then, since the composition of semismooth functions in $\mathbb{R}^n$ is a semismooth function \cite[Prop. 2.9]{Ulbrich}, it follows that $\phi(a)$ is semismooth. 
\end{proof}
\begin{remark}
  The function $\phi_j:\mathbb{R}^m \rightarrow\mathbb{R}$ defined by $\phi_j(a):=\gamma \sigma \frac{ a_j}{\max(\sigma, \gamma |a|)}$ is also Lipschitz continuous and semismooth. The proof of this assertion is analogous to the one given in Lemma \ref{lem:phi_lips_cont}.
\end{remark}
\begin{lema}
  Let $\phi(\Div \mathbf{u}(x))=\displaystyle \sigma \gamma \frac{  \Div \mathbf{u}(x)}{\max(\sigma, \gamma |\Div \mathbf{u}(x)|)}$ with $\gamma$ and $\sigma$ positive  constants.
  A measurable selection $M_{\phi}( \mathbf{u})$ of Clarke's generalized Jacobian $ \partial \phi(\Div \mathbf{u}) $ is :
\begin{equation}
  M_{\phi}( \mathbf{u})=\begin{cases}
\displaystyle\sigma \frac{1}{|\Div\mathbf{u}|} \,  - \sigma\displaystyle   \frac{(\Div\mathbf{u} \Div\mathbf{u}) }{|\Div\mathbf{u}|^3}, & if   \,\, \gamma  |\Div\mathbf{u}(x)| \geq \sigma\vspace{0.2cm} \\ \gamma, & if \,\, \gamma  |\Div\mathbf{u}(x)| < \sigma.
\end{cases}
 \end{equation}
\end{lema}
\begin{proof}
  
Let $\phi_3=\phi_1 \circ \phi_2$, where $\phi_1(z)= \max(0,z) + \sigma $ and $\phi_2(y)= \gamma |y| - \sigma $. Then the following identity holds:
\[
  \phi_3(y)=\max(\sigma, \gamma |y|)= \max(0, \gamma | y| -\sigma) + \sigma.
\]
From \cite[pp. 869]{Kunisch-Ito} we have that $ M_{\phi_1} \in \partial \phi_1(\gamma |y| - \sigma)$  given by
\begin{equation*}
  M_{\phi_1}( \gamma|y| - \sigma )=\begin{cases}
    1 , & if   \,\, \gamma|y| - \sigma  >0\\
    0,  & if   \,\, \gamma|y| - \sigma \leq 0,
  \end{cases}
\end{equation*}
is a measurable selection of $\partial \phi_1(\gamma |y| - \sigma)$. Next, since $\phi_2$ involves the function $| \cdot |$ evaluated at $y\neq 0$. From \cite[Exaple 2.5.1]{Ulbrich} we have that 
\[
  M_{\phi_2}(y)\in \partial \phi_2(y)=\displaystyle\bigg\{
\frac{\gamma y }{|y|}\bigg\} \, \text{for } y \neq 0 .
\]
Moreover, the chain rule for Clarke's generalized Jacobian \cite[Prop. 2.3]{Ulbrich} yields that:
\[M_{\phi_3}(y) v  \in \partial \phi_3(y) v \subset co\{ M_{\phi_1} M_{\phi_2} v: M_{\phi_1}  \in \partial \phi_1( \phi_2(y)) , M_{\phi_2} \in \partial \phi_2(y)\}.
\]
Thus, since $y \neq 0$, 
\begin{equation}\label{eq:app1}
  M_{\phi_3}(y)=\begin{cases}
    \frac{\gamma y  }{|y|} , & if   \,\, \gamma|y| - \sigma  >0\\
    0,  & if   \,\, \gamma|y| - \sigma \leq 0,
  \end{cases}
\end{equation}
Clearly, $\phi(y)= \sigma \gamma \frac{y}{ \phi_3(y)}$. Then, from the composition of functions we obtain that 
\[
  M_{\phi}(y) \in \partial \phi(y)  \subseteq \sigma \gamma \frac{\phi_3(y) \cdot 1 - y \partial \phi_3(y)\,\, }{\phi_3(y)^2}. 
\]
Then, from \eqref{eq:app1} the following cases can occur:
\begin{itemize}
  \item $\gamma | y| > \sigma$. Here we have that:
  \[
    M_{\phi}(y) =  \displaystyle\sigma \frac{1}{|y|} \,  - \sigma\displaystyle   \frac{y^2}{|y|^3}.
  \]
  \item $\gamma | y| \leq \sigma$ yields that:
  \[
    M_{\phi}(y) =  \gamma.
  \]
\end{itemize}
Finally, by taking $y=\Div \mathbf{u}$ we have the desired result.
\end{proof}


\begin{thebibliography}{99}
\bibitem{Adams}
{\sc R. Adams and J. F. Fournier}, {Sobolev Spaces}, Academic Press, The Netherlands, 2003.

\bibitem{Aposporidis}
{\sc A. Aposporidis et al.}, {\em A mixed formulation of the Bingham fluid flow problem: Analysis and numerical solution}, Computer methods in applied mechanics and engineering, 200 (2011) 2434-2446.

\bibitem{Uzawa}
{\sc K. J. Arrow, H. Azawa, L. Hurwicz and H. Uzawa}, {Studies in linear and non-linear programming}, Stanford University Press, 1953.

\bibitem{Babusca}
{\sc I. Babu\v{s}ca}, {\em Error-bounds for finite element method}, Numerische Mathematik, 16 (1971) 322–333.

\bibitem{Bercovier}
{\sc M. Bercovier and M. Engelman}, {\em A finite element method for incompressible non-Newtonian flows}, J. Comput. Phys., 36 (1980) 313–326.

\bibitem{Bonnans}
{\sc J. F. Bonnans, J. C. Gilbert, C. Lemar\'echal and C. A. Sagastiz\'abal}, {Numerical optimization: theoretical and practical aspects}, Springer Science and Business Media, New York, 2006.

\bibitem{Ciarlet}
{\sc P. G. Ciarlet}, {Linear and Nonlinear Functional Analysis with Applications}, SIAM, U.S.A., 2013.

\bibitem{Clarke}
{\sc F. Clarke}, {Functional Analysis, Calculus of Variations and Optimal Control}, Springer Verlag, London, 2013.

\bibitem{Combettes}
{\sc H. H. Bauschke and P. L. Combettes}, {Convex analysis and monotone operator theory in Hilbert spaces}, Springer, New York, 2011.

\bibitem{DlR}
{\sc J.C. De los Reyes, E. Loayza and P. Merino }, {\em Second-order orthant-based methods with enriched Hessian information for sparse $\l _1 $-optimization}, Computational Optimization and Applications, 67 (2017) 225--258.

\bibitem{DlRG1}
{\sc J. C. De los Reyes and S. Gonz\'alez-Andrade}, {\em Path following methods for steady laminar Bingham flow in cylindrical pipes}, ESAIM Math. Model. Numer. Anal., 43 (2009) 81--117.

\bibitem{DlRG2}
{\sc J. C. De los Reyes and S. Gonz\'alez-Andrade}, {\em Numerical Simulation of Two-Dimensional Bingham Fluid Flow by Semismooth Newton Methods}, Journal of Computational and Applied Mathematics, 235 (2010) 11--32.

\bibitem{DlRG3}
{\sc J. C. De los Reyes and S. Gonz\'alez-Andrade}, {\em A combined BDF-semismooth Newton approach for time-dependent Bingham flow}, Numerical Methods for Partial Differential Equations, 28 (2012) 834--860.

\bibitem{Dennis}
{\sc J.E. Dennis and R.B. Schnabel }, {Numerical methods for unconstrained optimization and nonlinear equations}, Society for Industrial and Applied Mathematics, 1996.

\bibitem{Ekeland}
{\sc I. Ekeland and R. T\'emam }, {Convex Analysis and Variational Problems}, SIAM, Philadelphia, 1999.

\bibitem{Gabay}
{\sc D. Gabay and B. Mercier}, {\em A dual algorithm for the solution of nonlinear variational problems via finite element approximation}, Computers and mathematics with applications, 2 (1976) 17–-40.

\bibitem{Girault}
{\sc V. Girault and P. A. Raviart}, {Finite element methods for Navier-Stokes equations: theory and algorithms}, Springer Science and Business Media, 2012.

\bibitem{Giaquinta}
{\sc M. Giaquinta and G. Modica}, {Mathematical analysis: An introduction to functions of several variables}, Birkh\"auser, Boston, 2010.

\bibitem{Glowinski}
{\sc R. Glowinski}, {On alternating direction methods of multipliers: a historical perspective. In Modeling, Simulation and Optimization for Science and Technology}, Springer, (2014) 59–-82.

\bibitem{Gonzalez-Andrade1}
{\sc S.  Gonz\'alez-Andrade}, {\em Semismooth Newton and path-following methods for the numerical simulation of Bingham fluids}, PhD thesis, EPN Quito, 2008.

\bibitem{Gonzalez-Andrade2}
{\sc S.  Gonz\'alez-Andrade}, {\em A preconditioned descent algorithm for variational inequalities of the second kind involving the p-Laplacian operator}, Computational Optimization and Applications, 66 (2017) 123--162

\bibitem{Gonzalez-Andrade3}
{\sc J. C. De los Reyes and S. Gonz\'alez-Andrade}, {\em Path following methods for steady laminar Bingham flow in cylindrical pipes}, ESAIM: Mathematical Modelling and Numerical Analysis, 43 (2009) 81--117

\bibitem{GALO}
{\sc S.  Gonz\'alez-Andrade and S. L\'opez-Ord\'o\~nez}, {\em A multigrid optimization algorithm for the numerical solution of quasilinear variational inequalities involving the p-Laplacian}, Computers and Mathematics with Applications, 75 (2018) 1107--1127

\bibitem{Hinze}
{\sc M. Hinze, R. Pinnau, M. Ulbrich and S. Ulbrich}, {Optimization with PDE constraints }, Springer, Secaucus, 2009.

\bibitem{Huilgol1}
{\sc R. R. Huilgol and Q. D. Nguyen}, {\em Variational principles and variational inequalities for the unsteady flows of a yield stress fluid}, International Journal of Non-Linear Mechanics, 36 (2001) 49--67.

\bibitem{Huilgol2}
{\sc R. R. Huilgol and Z. You}, {\em Application of the augmented Lagrangian method to steady pipe flows of Bingham, Casson and Herschel–Bulkley fluids}, Journal of Non-Newtonian Fluid Mechanics, 128 (2005) 126--143.

\bibitem{Kelly}
{\sc C. T. Kelly}, {Iterative Methods for Linear and Nonlinear Equations}, Philadelphia: SIAM, 1995.

\bibitem{Kikuchi}
{\sc N. Kikuchi and J. T. Oden. }, {Contact problems in elasticity: a study of variational inequalities and finite element methods}, Studies in Applied Mathematics, SIAM, U.S.A., 1988.

\bibitem{Kunisch-Ito}
{\sc M. Hintermüller, K. Ito and K. Kunisch }, {\em The primal-dual active set strategy as a semismooth Newton method} SIAM Journal on Optimization, 13 (2002) 865--888.

\bibitem{Laaber}
{\sc P. Laaber}, {\em Numerical simulation of a three-dimensional Bingham fluid flow}, 2008

\bibitem{Lions}
{\sc J.L. Lions}, {Optimal Control of Systems Governed by Partial Differential Equations}, Springer - Verlag, Germany, 1971.

\bibitem{Merino}
{\sc  J. C. De los Reyes and P. Merino}, {\em The Second Order Method with Enriched Hessian Information for Imaging Composite Sparse Optimization Problems, preprint 2020.}

\bibitem{Peypo}
{\sc J. Peypouquet}, {Convex optimization in normed spaces: theory, methods and examples}, Springer, London, 2015.

\bibitem{Qi}
{\sc X. Chen, Z. Nashed and L. Qi }, {\em Smoothing methods and semismooth methods for nondifferentiable operator equations}, SIAM Journal on Numerical Analysis, 38 (2001) 1200--1216.

\bibitem{Scholtes}
{\sc S. Scholtes}, {Introduction to piecewise differentiable equations}, Springer, London, 2012.

\bibitem{Stadler}
{\sc G. Stadler}, {\em Elliptic optimal control problems with L 1-control cost and applications for the placement of control devices}, Computational Optimization and Applications, 44 (2009) 159.

\bibitem{Temam}
{\sc R. Temam}, {Navier-Stokes Equations. Theory and Numerical Analysis}, AMS Chelsea Publishing, U.S.A., 2001.

\bibitem{Treskatis}
{\sc T. Treskatis, M. Moyers-Gonz\'alez and C.J. Price}, {\em An accelerated dual proximal gradient method for applications in viscoplasticity}, Journal of Non-Newtonian fluid mechanics, 238  (2016) 115--130.

\bibitem{Ulbrich}
{\sc M. Ulbrich}, {Semismooth Newton methods for variational inequalities and constrained optimization problems in function spaces}, SIAM, Philadelphia, 2011.

\bibitem{Wilbrandt}
{\sc U. Wilbrandt }, {Stokes - Darcy Equations Analytic and Numerical Analysis}, Birkh\"auser, Cham, 2019.

\bibitem{Zowe}
{\sc J. Zowe and S. Kurcyusz}, {\em Regularity and Stability for the Mathematical Programming Problem in Banach Spaces}, Appl. Math. Optim., 5 (1979) 49--62 .
\end{thebibliography}
\end{document}